\documentclass{amsart}
\usepackage[utf8]{inputenc}

\usepackage{a4wide}
\usepackage{mathabx}
\usepackage[english]{babel}
\usepackage{amsthm}
\usepackage{enumitem}
\usepackage{comment}

\usepackage{xcolor}
\usepackage[colorlinks,
linkcolor={red!50!black},
citecolor={blue!50!black},
urlcolor={blue!80!black}]{hyperref}

\usepackage{tikz}
\usepackage[all]{xy}
\usepackage{amsmath}
\usepackage{amsthm}
\usepackage{amssymb}
\usepackage{amsfonts}
\usepackage{amsopn}
\usepackage{amsrefs}
\usepackage{hyperref}
\usepackage[color=green!30]{todonotes}
\usepackage{fdsymbol}

\usepackage{import}

\usetikzlibrary{calc,cd,intersections,decorations.markings,patterns,decorations.pathreplacing,matrix,arrows}

\def\Z{\mathbb{Z}}
\def\Q{\mathbb{Q}}
\def\R{\mathbb{R}}
\def\F{\mathbb{F}}

\def\C{\mathbb{C}}

\def\L#1{{#1}[t^{\pm 1}]}
\def\O#1{#1(t)}
\def\LF{\L{\F}}
\def\LR{\L{\R}}
\def\LC{\L{\C}}

\def\OC{\O{\C}}

\def\ol#1{\overline{#1}}

\def\makeithash#1{#1^\#}
\def\makeithashT#1{#1^{\#T}}

\newcommand{\id}{\operatorname{id}}
\newcommand{\sign}{\operatorname{sign}}
\newcommand{\im}{\operatorname{im}}

\newcommand{\PD}{\operatorname{PD}}

\DeclareMathOperator{\lk}{lk}
\DeclareMathOperator{\ord}{ord}
\DeclareMathOperator{\res}{\operatorname{res}}

\DeclareMathOperator{\coker}{coker}

\DeclareMathOperator{\Hom}{Hom}
\DeclareMathOperator{\ev}{ev}

\DeclareMathOperator{\ind}{\operatorname{ind}}

\DeclareMathOperator{\Bl}{Bl}

\def\op{\operatorname}
\newcommand{\bsm}{\left(\begin{smallmatrix}}
    \newcommand{\esm}{\end{smallmatrix}\right)}

\newtheorem{theorem}{Theorem}[section]

\newtheorem{corollary}[theorem]{Corollary}
\newtheorem{lemma}[theorem]{Lemma}
\newtheorem{proposition}[theorem]{Proposition}

\theoremstyle{definition}
\newtheorem{definition}[theorem]{Definition}

\newtheorem{example}[theorem]{Example}

\theoremstyle{remark}
\newtheorem{remark}[theorem]{Remark}
\newtheorem{construction}[theorem]{Construction}
\theoremstyle{claim}

\newtheorem*{claim*}{Claim}

\newcommand{\smargin}[1]{\marginpar{\tiny{#1}}}

\numberwithin{equation}{section}

\title[Twisted Blanchfield pairings and twisted signatures II]{Twisted Blanchfield pairings and twisted signatures II: Relation to Casson-Gordon invariants}

\author{Maciej Borodzik}
\address{Institute of Mathematics, University of Warsaw, ul. Banacha 2, 02-097 Warsaw, Poland}
\email{mcboro@mimuw.edu.pl}
\author{Anthony Conway}
\address{Massachusetts Institute of Technology, Cambridge MA 02139}
\email{anthonyyconway@gmail.com}
\author{Wojciech Politarczyk}
\address{Institute of Mathematics, University of Warsaw, ul. Banacha 2, 02-097 Warsaw, Poland.}
\email{wpolitarczyk@mimuw.edu.pl}

\dedicatory{This paper is dedicated to the memory of Andrew Ranicki}

\begin{document}
\begin{abstract}
  This paper studies twisted signature invariants and twisted linking forms, with a view towards obstructions to knot concordance.
  Given a knot $K$ and a representation~$\rho$ of the knot group, we define a twisted signature function~$\sigma_{K,\rho} \colon S^1 \to \Z$.
  This invariant satisfies many of the same algebraic properties as the classical Levine-Tristram signature~$\sigma_K$.
  When the representation is abelian,~$\sigma_{K,\rho}$ recovers~$\sigma_K$, while for appropriate metabelian representations,~$\sigma_{K,\rho}$ is closely related to the Casson-Gordon invariants.
  Additionally, we prove satellite formulas for~$\sigma_{K,\rho}$ and for twisted Blanchfield forms.
\end{abstract}

\maketitle

\section{Introduction}

For a knot $K$, several invariants such as the Alexander polynomial $\Delta_K$, the Levine--Tristram signature
$\sigma_K \colon S^1 \to \Z$ and the Blanchfield form $\Bl(K)$ can be computed from homological data of
the infinite cyclic cover of the knot complement $S^3\setminus K$.
To access information about $K$ beyond the abelian covers of $S^3\setminus K$, the idea of twisted invariants is to additionally keep track of a representation of the group $\pi_1(S^3 \setminus K).$

In the case of polynomial invariants, the resulting theory of twisted Alexander polynomials, originally due to X.S. Lin~\cite{Lin} has been thoroughly developed, and we refer to~\cite{FriedlVidussiSurvey} for a survey of the vast literature on the subject.
The present article aims to further develop the theory of twisted signatures and twisted Blanchfield forms.
Before delving into the theory, we provide some motivation from knot concordance.

\subsection{Knot concordance}

A knot~$K$ is called \emph{slice} if it bounds a locally flat embedded disc in~$D^4$~\cite{FoxMilnor}.
The first obstructions to sliceness date back to Fox and Milnor’s original paper on the subject and involve the Alexander polynomial~$\Delta_K$.
Later work of Levine~\cite{LevineInvariants} describes the algebraic concordance group, which roughly speaking, encapsulates the obstructions to sliceness that come from the Seifert matrix, including~$\Delta_K$ and the piecewise constant Levine-Tristram signature 
$$ \sigma_K \colon S^1 \to \Z.~$$
Both~$\Delta_K$ and~$\sigma_K$ (and in fact the entire algebraic concordance class of~$K$) can be recovered from the so-called \emph{Blanchfield form}, a non-singular, sesquilinear, Hermitian form on the Alexander module $H_1(X_K;\Z[t^{\pm 1}])$ of $K$:
$$ \Bl(K) \colon H_1(X_K;\Z[t^{\pm 1}]) \times H_1(X_K;\Z[t^{\pm 1}]) \to \Q(t)/\Z[t^{\pm 1}].$$
A dozen years after Levine’s work, Casson and Gordon developed further signature invariants to obstruct certain algebraically slice knots from being slice~\cite{CassonGordon2}.
To a knot $K$ and a prime power order character~$\chi \colon H_1(\Sigma_n(K)) \to \C^\times$ (here $\Sigma_n(K)$ denotes the $n$-fold branched cover of $K$), their invariant takes the form of a Witt class
$$\tau(K,\chi) \in W(\C(t)) \otimes \Q.$$
Due to its 4-dimensional nature, this invariant is notoriously difficult to compute.
In fact, for~$\omega \in S^1$, even the signatures~$\sign_\omega^{av}(\tau(K,\chi))$, obtained by taking the (averaged) signature at $t=\omega$ of a matrix representative of $\tau(K,\chi)$, remain difficult to approach in general.
For instance, while~$\sign_1^{av}(\tau(K,\chi))$ can be computed for genus one knots~\cite[Theorem 3.5]{GilmerSliceKnots}, in general it can only be estimated in terms of simpler invariants; 
see e.g.~\cite[Theorem 3]{CassonGordon2} and~\cite[Theorem~3.4]{GilmerSliceKnots}.
While applications of Casson-Gordon theory are numerous, 
our aim is to develop new methods to calculate $\sign_\omega^{av}(\tau(K,\chi))$ and, following~\cite{HeddenKirkLivingston}, to apply them to study the linear independence of algebraic knots in the knot concordance~\cite{ConwayKimPolitarczyk}.

Returning to the broader picture, additional information from the Witt class~$\tau(K,\chi)$ was later made accessible by Kirk and Livingston, by means of twisted Alexander polynomials~\cite{KirkLivingston}.
As in the abelian setting, efforts to encompass both signature and polynomial obstructions, then led Miller and Powell~\cite{MillerPowell} to define a metabelian Blanchfield form
$$ \Bl_{\alpha(n,\chi)}(K) \colon H_1(M_K;\LC^n) \times H_1(M_K;\LC^n) \to \C(t)/\LC $$
in order to efficiently detect knots in the quotient~$\mathcal{F}_2/\mathcal{F}_{2.5}$ of the solvable filtration of~\cite{CochranOrrTeichner}.
Here,~$M_K$ denotes the $3$-manifold obtained by $0$-framed surgery on $K$.

Our goal is to use twisted Blanchfield pairings to define an algorithmically computable twisted signature function that generalizes both the Levine-Tristram signature $\sigma_K$ and the Casson-Gordon signatures~$\sign_\omega^{av}(\tau(K,\chi))$; the algorithm is described in~\cite[Subsection 1.4]{BCP_Compu} and can be implemented in a computer for any knot $K$, as long as the complex roots of the twisted Alexander polynomial $\Delta_K^{\alpha(K,\chi)}(t)$ that have modulus one can be determined exactly.
The analogy we have in mind being that the Levine-Tristram signature~$\sigma_K$ can be recovered from the Blanchfield form~$\Bl(K)$.
This goal can be formulated more broadly in the context of twisted knot invariants.

\subsection{The twisted signature function of a knot}

We now recall the set-up that goes into defining twisted invariants and discuss the properties of our twisted signature function.
Given a field $\F$ and a space $X$, a representation $\alpha \colon \pi_1(X) \to GL_n(\LF)$ is called \emph{acyclic} if the twisted homology groups $H_k(X;\LF^n_\alpha)$ are $\LF$-torsion for all $k$.
Write~$M_K$ for the 3-manifold obtained by 0-framed surgery on the knot~$K$, and fix a unitary acyclic representation~$\alpha \colon \pi_1(M_K) \to GL_n(\LF)$.
Here,  a matrix $A \in GL_n(\LF)$ is called \emph{unitary} if $A(A^\#)^T=\id$, where $\#$ denotes the involution whose componentwise effect is $\sum a_it^i \mapsto \sum \overline{a}_it^{-i}$ (we write $\overline{a}$ for the complex conjugate of $a \in \F$), and a representation \(\alpha\) is \emph{unitary} if \(\alpha(g)\) is a unitary matrix, for any \(g \in \pi_{1}(M_{K})\).

Associated to this data, there is a twisted Alexander polynomial~$\Delta^\alpha_{K}(t)$~\cite{Lin,Wada,KirkLivingston} and a twisted Blanchfield form $\Bl_\alpha(K)$~\cite{Powell,MillerPowell}.
When~$\alpha$ is induced by the abelianisation $ab$, these invariants respectively coincide with the classical Alexander polynomial $\Delta_K(t)$ and Blanchfield form $\Bl(K)$, while metabelian invariants, such as the Casson-Gordon invariants, are recovered by considering a representation 
\begin{equation}
  \label{eq:MetabelianIntro}
  \alpha_K(n,\chi) \colon \pi_1(M_K) \to GL_n(\LC)
\end{equation}
whose precise definition will be recalled in Subsection~\ref{sub:Metabelian}.

This paper develops a corresponding theory of twisted signature functions and further studies twisted Blanchfield forms.
The following theorem, which combines the results of Proposition~\ref{prop:LocallyConstant}, Remark~\ref{rem:LevineTristram}, Proposition~\ref{prop:SignatureSlice} and Corollary~\ref{cor:SignatureSatellite}, summarises some of the properties of our twisted signatures. 



\begin{theorem}
  \label{thm:Recap}
  Suppose~$\F$ is either~$\R$ or~$\C$.
  Given a knot $K$ and a unitary acyclic representation~$\alpha \colon \pi_1(M_K) \to GL_d(\LF)$, there is a function~$\sigma_{K,\alpha}\colon S^1 \to \Z$ with the following properties:
  \begin{enumerate}
  \item Piecewise continuity. The twisted signature function~$\sigma_{K,\alpha}$ is locally constant on the complement of the zeroes of the twisted Alexander polynomial $\Delta_K^\alpha(t)$.
  \item Obstruction to sliceness. The averaged signature function
    \begin{equation}\label{eq:averageKnotIntro}
      \sigma^{av}_{(K,\alpha)}(\omega)=\frac12\left(\lim_{\theta\to 0^+}\sigma_{(K,\alpha)}(e^{i\theta}\omega)+\lim_{\theta\to 0^-}\sigma_{(K,\alpha)}(e^{i\theta}\omega)\right)
    \end{equation} 
    provides an obstruction to sliceness: if $K$ is sliced by a disc $D \subset D^4$ and if~$\alpha$ extends to a unitary acyclic representation $\pi_1(D^4 \setminus \mathcal{N}(D)) \to GL_d(\LF)$, then $\sigma_{K,\alpha}^{av} \equiv 0$. 
  \item Behavior under satellite operations. Let $P,J$ be knots, and let~$\eta \subset S^3\setminus P \subset M_P$ be a simple closed curve.
    Assume that $\gamma \colon \pi_1(M_{P(J,\eta)}) \to GL_d(\LF)$ is a unitary acyclic representation such that $\det(\id -\gamma(\eta)) \neq 0$ and $\gamma(\mu_\eta)=\id$, where $\mu_\eta$ denotes the meridian of $\eta$.
    Then, for every $\omega \in S^1$, the averaged twisted signature satisfies
    $$  \sigma^{av}_{P(K,\eta),\gamma}(\omega)= \sigma^{av}_{P,\gamma_P}(\omega)+\sigma^{av}_{K,\gamma_K}(\omega).$$
  \item Relationship with the Levine-Tristram signature: if the representation~$\alpha$ is induced by the abelianisation, then~$\sigma_{K,\alpha}(\omega)$ coincides with $\sigma_K(\omega)$ for all $\omega \in S^1$.
  \end{enumerate}
\end{theorem}

Theorem~\ref{thm:Recap} shows that the twisted signature function is related to the twisted Alexander polynomial in the same way as the Levine-Tristram signature is related to the Alexander polynomial.
We refer to Proposition~\ref{prop:LocallyConstant} for some elementary properties of the twisted signature function, but now describe how the Casson-Gordon signatures can be recovered from our signature function.

\begin{theorem}
  \label{thm:Recover}
  Fix a prime power order character $\chi \colon H_1(\Sigma_n(K)) \to \Z_m$ and $\omega \in S^1$.
  When $\alpha=\alpha(n,\chi)$ is the metabelian representation mentioned in~\eqref{eq:MetabelianIntro}, the averaged twisted signature function is related to the Casson-Gordon signature $\sign_\omega^{av}(\tau(K,\chi))$ in the following way:
  $$ \sigma_{K,\alpha(n,\chi)}^{av}(\omega)=-\sign_\omega^{av}(\tau(K,\chi))+\sign_1^{av}(\tau(K,\chi)).$$
\end{theorem}

Theorem~\ref{thm:Recover} leads to what appears to be the first algorithm to explicitly compute the difference $-\sign_\omega^{av}(\tau(K,\chi))+\sign_1^{av}(\tau(K,\chi))$ of Casson-Gordon signatures from a knot diagram; the algorithm is described and implemented in \cite{BCP_Compu}.
In particular, while both~$\sign_\omega^{av}(\tau(K,\chi))$ and~$\sign_1^{av}(\tau(K,\chi))$ are rational numbers (as we recall in Definition~\ref{def:CassonGordonSignature}), Theorem~\ref{thm:Recover} implies that their difference is in fact an integer.

Summarising, the twisted signature function $\sigma_K \colon S^1 \to \Z$ satisfies properties analogous to the Levine-Tristram signature $\sigma_K$, recovers it when $\alpha$ is induced by the abelianisation, and is related to Casson-Gordon signatures for $\alpha=\alpha(K,\chi)$.

\begin{remark}
  In the abelian setting, the Levine-Tristram signature~$\sigma_K(\omega)$ of a knot $K$ at $\omega \in S^1$ is defined as
  $$\sigma_K(\omega)=\sign \left( (1-\omega)V +(1-\overline{\omega})V^T \right),$$
  where~$V$ is any Seifert matrix for $K$.
  In the twisted setting, the lack of a Seifert matrix means that there is no straightforward way of defining a twisted analogue of $\sigma_K$ by 3-dimensional means.

  On the other hand while $4$-dimensional approaches based on the eta invariant~\cite{AtiyahPatodiSinger} have been studied~\cite{LevineEta,FriedlEta}, the resulting signature invariants are unwieldly and, as far as we can tell, their relation to other twisted invariants is less apparent.
\end{remark}

As the signatures of Theorem~\ref{thm:Recap} are defined using twisted Blanchfield forms, we briefly review the set-up underlying these linking forms and our results concerning them.

\subsection{The twisted Blanchfield forms}

Given a closed, connected, oriented 3-manifold~$N$ and a unitary acyclic representation~$\alpha \colon \pi_1(N) \to GL_n(\LF)$, 
the twisted Blanchfield form is a Hermitian, non-singular sesqulinear form 
$$ \Bl_\alpha(N) \colon H_1(N;\LF^n_\alpha) \times H_1(N;\LF^n_\alpha) \to \F(t)/\LF.$$
While this pairing was first studied systematically by Powell~\cite{Powell} and, for $\alpha=\alpha(n,\chi)$, by Miller and Powell~\cite{MillerPowell}, its roots are to be found in the earlier work of Cochran-Orr-Teichner~\cite{CochranOrrTeichner} in the non-commutative setting.

The work of Powell laid the technical groundwork to the study of $\Bl_\alpha(N)$ and proved that the pairing is Hermitian and non-singular.
On the whole however, few general properties of $\Bl_\alpha(N)$ have been established so far.
A second aim of this paper is to prove further properties of $\Bl_\alpha(N)$, when $N=M_K$ is the $0$-framed surgery along $K$, in which case we write
$$\Bl_\alpha(K):=\Bl_\alpha(M_K).$$
When $\alpha$ is induced by abelianisation, $\Bl_\alpha(K)=\Bl(K)$ is the ``classical" Blanchfield pairing of $K$.

Our first result about twisted Blanchfield forms (which is similar to the work of Miller and Powell in the metabelian setting) describes how $\Bl_\alpha(K)$ leads to sliceness obstructions.

\begin{proposition}[see Proposition~\ref{prop:SignatureSlice}]
  Let \(K\) be a slice knot and let \(\beta \colon \pi_{1}(M_K) \to GL_d(\F[t^{\pm 1}])\) be a unitary acyclic representation. If $D\subset D^4 $ is a slice disk for $K$ and there is a unitary acyclic representation \(\gamma \colon \pi_{1}(D^4 \setminus \mathcal{N}(D)) \to GL_d(\F[t^{\pm 1}])\) extending $\beta$,  then the twisted Blanchfield pairing~\(\Bl_{\beta}(K)\) is metabolic.\footnote{\(\Bl_{\beta}(K)\) being \emph{metabolic} means that there exists a submodule $P \subset H_1(M_K;\LF^d_\beta)$ with $P=P^\perp$, where $P^\perp=\lbrace x \in H_1(M_K;\LF^d_\beta) \ | \ \Bl_{\beta}(K)(x,p)=0 \text{ for all } p \in P \rbrace.$}
\end{proposition}

One result we wish to emphasise concerns the behaviour of the twisted Blanchfield form under satellite operations.
For the classical Blanchfield form, a satellite formula was proved by Livingston-Melvin~\cite{LivingstonMelvin}: given a winding number $n$ satellite operator~$P \subset S^1 \times D^2$, this formula reads as
$$\Bl_{P(K)}(t) \cong \Bl(P)(t) \oplus \Bl(K)(t^n)$$ and can be seen as a generalization of the satellite formula for the Levine-Tristram signature~\cite{LitherlandIterated,Shinohara}.
In the metabelian case, a satellite formula for the Casson-Gordon invariants was proved by Litherland~\cite{LitherlandCobordism}, while satellite formulas for higher order signatures can be found in~\cite{CochranOrrTeichnerStructure}.
Building on the algebraic machinery from~\cite{BCP_Alg} and on work of Friedl-Leidy-Nagel-Powell~\cite{FriedlLeidyNagelPowell}, Theorem~\ref{thm:CablingTheorem} provides a satellite formula for twisted Blanchfield forms.
A particular case of this theorem
reads as follows.

\begin{theorem}
  \label{thm:SatelliteFormulaIntro}
  Let~$J,P$ be knots and let~$\eta \subset S^3\setminus P \subset M_P$ be a simple closed curve.
  Assume that~$\gamma \colon \pi_1(M_{P(J,\eta)}) \to GL_d(\LF)$ is a unitary acyclic representation such that $\det(\id -\gamma(\eta)) \neq~0$ and $\gamma(\mu_\eta)=\id$, where $\mu_\eta$ denotes the meridian of $\eta$.
  Then the twisted Blanchfield form $Bl_\gamma(P(J,\eta))$ is Witt equivalent
  to~$\Bl_{\gamma_P}(P) \oplus \Bl_{\gamma_J}(J).$
\end{theorem}

Here,  two linking forms \((M_1,\lambda_1)\) and \((M_2,\lambda_2)\) are \emph{Witt equivalent} if \((M_{1} \oplus M_{2}, \lambda_{1} \oplus -\lambda_{2})\) is metabolic. 
In particular, Theorem~\ref{thm:SatelliteFormulaIntro} implies the third item of Theorem~\ref{thm:Recap} because, as we recall in Proposition~\ref{prop:SignatureFunctionAlgebra} below,  Witt equivalent linking forms have identical averaged signature functions.

In Theorem~\ref{thm:metabelian-cabling-formula}, we prove a satellite formula for the metabelian Blanchfield form $\Bl_{\alpha(n,\chi)}(K)$; the proof uses Theorem~\ref{thm:SatelliteFormulaIntro}, and some representation theory.
This result then underpines further work of the second and third author with Min Hoon Kim concerning applications of our theory to the study of linear independence of algebraic knots in knot concordance~\cite{ConwayKimPolitarczyk}.


\subsection{Linking forms and their signature functions}

In order to outline the construction of the twisted signature $\sigma_{K,\alpha}$, we recall some properties of the signature function of a linking form from~\cite{BCP_Alg}.
First recall that given a PID~$R$ with involution $x \mapsto x^\#$
and field of fractions~$Q$, a \emph{linking form} refers to a sesquilinear, Hermitian form~
$$\lambda \colon M \times M \to Q/R$$
on a torsion $R$-module $M$.
Such a linking form is \emph{metabolic} if there exists a submodule $P \subset M$ with $P=P^\perp$, while~$(M,\lambda)$ is \emph{represented by} an~$(n\times n)$-non-degenerate Hermitian matrix~$A(t)$ if it is isometric to the linking form
\begin{align*}
  \lambda_{A(t)} \colon \LF^n/A(t) \times \LF^n/A(t) &\to \F(t)/\LF \\
  (x,y) &\mapsto x^TA(t)^{-1}y^\#.
\end{align*}
The next proposition summarises several properties of the signature functions built in~\cite{BCP_Alg}.

\begin{proposition}
  \label{prop:SignatureFunctionAlgebra}
  Given a non-singular linking form $(M,\lambda)$ over $\LF$, where \(\F = \R\) or \(\C\), there exists a signature function $\sigma_{(M,\lambda)} \colon S^1 \to \Z$ that satisfies the following properties:
  \begin{enumerate}[label=(S-\arabic*)]
  \item The signature function is locally constant on $S^1 \setminus \lbrace \xi \in S^1 \ | \ \Delta_M(\xi)=0 \rbrace$,  where $\Delta_M$ denotes the order of the $\LF$-module $M$.\label{item:S1}
  \item If \(\F = \R\), 
  then we have \label{item:S2}
    $$\sigma_{(M,\lambda)}(\overline{\xi})=\sigma_{(M,\lambda)}(\xi).$$
  \item The signature function is additive: \label{item:S3}
    $$\sigma_{(M_1,\lambda_1) \oplus (M_2,\lambda_2)}(\xi)=\sigma_{(M_1,\lambda_1)}(\xi)+\sigma_{(M_2,\lambda_2)}(\xi).$$
  \item The \emph{averaged signature function}\label{item:S4}
    $$  \sigma^{av}_{(M,\lambda)}(\xi)=\frac12\left(\lim_{\theta\to 0^+}\sigma_{(M,\lambda)}(e^{i\theta}\xi)+\lim_{\theta\to 0^-}\sigma_{(M,\lambda)}(e^{i\theta}\xi)\right)
    $$
    is locally constant, additive and vanishes if $(M,\lambda)$ is metabolic.
  \item If $(M,\lambda)$ is represented by a matrix $A(t)$, then the averaged signature function satisfies\label{item:S5}
    $$  \sigma^{av}_{(M,\lambda)}(\xi) =\sign^{av} A(\xi)-\sign^{av}A(1).$$
      \item \label{item:S6} A linking form $(M,\lambda)$ over $\LC$ is representable if and only if $\sigma^{av}_{(M,\lambda)}(1)=0$.
  \end{enumerate}
\end{proposition}
\begin{proof}
The five first properties are in listed in~\cite[Subsection 1.5]{BCP_Alg}.
Item~\ref{item:S6} is proved in~\cite[Proposition 5.7 and Corollary 5.15]{BCP_Alg}.
\end{proof}

The twisted signature function of Theorem~\ref{thm:Recap} is defined as the signature function of the non-singular linking form $(H_1(M_K;\LF_\alpha^n),\Bl_\alpha(K))$.
As a consequence, properties of the twisted signature (such as those listed in Theorem~\ref{thm:Recap}) are established by studying the behavior of the twisted Blanchfield form and then applying
Proposition~\ref{prop:SignatureFunctionAlgebra}.

\subsection*{Organisation}

In Section~\ref{sec:TwistedHomology}, we recall some background on twisted homology and twisted Blanchfield forms; we also relate twisted Blanchfield forms to their non-commutative analogues.
In Section~\ref{sec:TwistedSign3ManifKnot}, we define our twisted signature invariants, and prove several of the properties stated in Theorem~\ref{thm:Recap}; additionally in Theorem \ref{thm:CablingTheorem}, we prove our satellite formula for the twisted Blanchfield form.
In Section~\ref{sec:CassonGordon}, we focus on metabelian representations, relate our invariants to Casson-Gordon signatures (Theorem~\ref{thm:BlanchfieldCG}) and prove the satellite formulas for the metabelian Blanchfield form (Theorem~\ref{thm:metabelian-cabling-formula}).

\subsection*{Acknowledgments}
This paper together with~\cite{BCP_Alg,BCP_Compu} was originally part of a single paper titled ``Twisted Blanchfield pairings and Casson--Gordon invariants". 
During that project, the second named author was at University of Geneva and later at Durham University, supported by the Swiss National Science Foundation; during a subset of the revision and splitting process, he was a visitor at the Max Planck Institute for Mathematics.
The first named author is supported by the National Science Center grant 2019/B/35/ST1/01120.
The third named author is supported by the National Science Center grant 2016/20/S/ST1/00369.
All three authors wish to thank the University of Geneva and the University of Warsaw at which part of this work was conducted.

\subsection*{Conventions}
  If $R$ is a commutative ring and $f,g\in R$, we write $f\doteq g$ if there exists a unit $u\in R$ such that $f=ug$.
  For a ring $R$ with involution, we denote this involution by $x\mapsto x^\#$; the symbol~$\ol{x}$ is reserved for the complex conjugation.
  In particular, for $R=\LC$ the involution~$(-)^\#$ is the composition of the complex conjugation with the map $t\mapsto t^{-1}$.
  Given an $R$-module $M$, we denote by $M^\#$ the $R$-module that has the same underlying additive group as $M$, but for which the action by $R$ on $M$ is precomposed with the involution on $R$.
  For a matrix $A$ over $R$, we write $\makeithashT{A}$ for the transpose followed by the involution.

We work in the topological category unless otherwise stated.
All manifolds are assumed to be compact, connected, based and oriented; if a manifold has a nonempty boundary, then the basepoint is assumed to be in the boundary.
Maps between based spaces are understood to be basepoint preserving.

\section{Twisted homology and twisted Blanchfield pairings}
\label{sec:TwistedHomology}

This section is organized as follows: in Subection~\ref{sub:TwistedHomology}, we review twisted (co)homology, in Subsection~\ref{sub:Blanchfield}, we discuss twisted Blanchfield forms, and in Subsection~\ref{sub:Ore}, we relate twisted Blanchfield forms to their non-commutative counterparts.

\subsection{Twisted homology and cohomology}
\label{sub:TwistedHomology}
In this subsection, we briefly review twisted (co)ho\-mo\-lo\-gy and twisted intersection forms.
While standard references include~\cite{KirkLivingston, FriedlKim},  we also refer the reader to~\cite[Section 3]{ChaMillerPowell} for a thorough discussion of the twisted homology of disconnected manifolds and the resulting long exact sequence of pairs.
Additionally,  for ease of exposition, we will focus on CW complexes.
This is not cause a cause for concern,  when working with 3-manifolds and smooth 4-manifolds both of which admit the structure of finite CW complexes of the appropriate dimension; we refer to~\cite[Section  4]{FriedlNagelOrsonPowell} for the precise statements and references.
On the other hand,  while (topological) $4$-manifolds need not admit CW structures,  arguments involving twisted homology and Poincar\'e duality can nevertheless be carried out as in the smooth category (using singular chain complexes instead of cellular chain complexes): this is non-trivial and we refer the interested reader to~\cite[Appendix A]{FriedlNagelOrsonPowell} for many more details and references.\footnote{Most of this article is  $3$-dimensional and this is why our background sections are in the cellular setting: the only moment where we encounter a (potentially) non smooth  $4$-manifold is in the proof of Proposition~\ref{prop:witt-equiv-concordance}.
}

\medbreak


Let $X$ be a path connected CW complex and let $Y \subset X$ be a possibly empty subcomplex.
Use $p \colon \widetilde{X} \to X$ to denote the universal cover of $X$ and set $\widetilde{Y}:=p^{-1}(Y)$.
The left action of $\pi_1(X)$ on~$\widetilde{X}$ endows the chain complex $C_*(\widetilde{X},\widetilde{Y})$ with the structure of a left $\Z[\pi_1(X)]$-module.
Moreover, let~$R$ be a ring and let $M$ be a $(R,\Z[\pi_1(X)])$-module.
The chain complexes
\begin{align*}
  & C_*(X,Y;M):=M \otimes_{\Z[\pi_1(X)]}C_*(\widetilde{X},\widetilde{Y}) \\
  & C^*(X,Y;M):=\op{Hom}_{\text{right-}\Z[\pi_1(X)]}(\makeithash{C_*(\widetilde{X}, \widetilde{Y} )},M)
\end{align*}
of~left $R$-modules will be called the \textit{twisted (co)chain complex} of $(X,Y)$ with coefficients in $M$.
The corresponding homology left $R$-modules~$H_*(X,Y;M)$ and $H^*(X,Y;M)$ will be called the \textit{twisted (co)homology} modules of $(X,Y)$ with coefficients in $M$.

Assume that $R$ is endowed with an involution $x \mapsto x^\#$. Let $M,M'$ be $(R,\Z[\pi_1(X)])$-bimodules and let $S$ be a $(R,R)$-bimodule.
Furthermore, let $\langle -,-\rangle \colon M \times M' \to S$ be a non-singular
$\pi_1(X)$-invariant sesquilinear pairing, in the sense that $\langle m\gamma , n\gamma  \rangle=\langle m,n\rangle$ and $\langle rm,sn \rangle=r \langle m,n \rangle \makeithash{s}$ for all $\gamma \in  \pi_1(X)$, all $r,s \in R$ and all $m\in M, n \in M'$.
Non-singularity means that the induced map $M\to \Hom_{\text{left-}R}(M',S)^\#$ is an isomorphism.
In this setting, as recalled in the appendix (specifically Construction~\ref{cons:Evaluation}),
there is an \textit{evaluation map}
$$ \ev \colon H^i(X,Y;M) \to \makeithash{\op{Hom}_{\text{left-}R}(H_i(X,Y;M'),S)}. $$
The evaluation map need not be an isomorphism.
If $R$ is a principal ideal domain, then the universal coefficient theorem implies that the cohomology group $H^i(X,Y;M)$ decomposes as the direct sum of $\makeithash{\op{Hom}_{\text{left-}R}(H_i(X,Y;M'),S)}$ and $\makeithash{\op{Ext}^1_{\text{left-}R}(H_i(X,Y;M'),S)}$.
In general, the evaluation map can be studied using the universal coefficient spectral sequence~\cite[Theorem 2.3]{LevineKnotModules}.
Returning to $\pi_1(X)$-invariant pairings, we now discuss two examples which we shall use throughout this section.

\begin{example}
  \label{ex:UsualCoeff}
  Let $R$ be an integral domain with field of fractions $Q$.
  Given a positive integer~$d$, let $\beta \colon \pi_1(X) \to GL_d(R)$ be a representation and use~$R^d_\beta$ to denote the $(R,\Z[\pi_1(X)])$-module whose right $\Z[\pi_1(X)]$-module structure is given by right multiplication by~$\beta(\gamma)$ on row vectors for~$\gamma \in \pi_1(X)$.
  Use $\widecheck{\beta} \colon \pi_1(X) \to GL_d(R)$
  to denote the representation defined by $\widecheck{\beta}(\gamma)=\makeithashT{\beta(\gamma^{-1})}$ and consider the pairings
  \begin{align*}
    \left(Q/R \otimes_R R_\beta^d\right) \times R_{\widecheck{\beta}}^d &\to Q/R   \ \ \ \ \ \ \ \ \ \ \ \ \ \
                                                    R_\beta^d \times R_{\widecheck{\beta}}^d \to R \\
    (q \otimes v, w) & \mapsto v \makeithashT{w} \cdot q \ \ \ \ \ \ \ \ \ \ \ \ \ \   (v, w) \mapsto v \makeithashT{w}.
  \end{align*}
  Note that the use of the representation $\widecheck{\beta}$ guarantees that these pairings are indeed $\pi_1(X)$-invariant.
  In particular, if $\beta$ is a unitary representation (so that $\widecheck{\beta}=\beta$), then the two pairings displayed above give rise to evaluation maps  $H^i(X,Y;Q/R \otimes_R R_\beta^d) \to \makeithash{\op{Hom}_{\text{left-}R}(H_i(X,Y;R_{\beta}^d),Q/R)}$ and  $H^i(X,Y; R_\beta^d) \to \makeithash{\op{Hom}_{\text{left-}R}(H_i(X,Y; R_{\beta}^d),R)}$.
\end{example}

Next, we briefly review the definition of twisted intersection forms.
Let $W$ be a compact oriented $n$-manifold and let $M$ be a $(R,\Z[\pi_1(W)])$-bimodule.
As we mentioned above,  if $W$ does not admit a CW structure, we can use singular homology; see~\cite[Appendix A]{FriedlNagelOrsonPowell}.
The cap product with the fundamental class \([W, \partial W] \in H_{4}(W, \partial W)\) induces twisted Poincar\'e duality isomorphisms
$$ H_k(W,\partial W;M) \cong H^{n-k}(W;M) \quad H_k(W;M) \cong H^{n-k}(W,\partial W;M),$$
both of which we denote by $\text{PD}$, 
for more details we refer to~\cite[Appendix A]{FriedlNagelOrsonPowell}.

In order to describe twisted intersection pairings, fix an $(R,R)$-bimodule $S$ and a non-singular $\pi_1(W)$-invariant sesquilinear pairing $M \times M \to S$.
Compose the homomorphism induced by the inclusion $(W,\emptyset) \to (W,\partial W)$ with Poincar\'e duality and the evaluation homomorphism described above.
The result is the following homomorphism of left $R$-modules.
$$ \Phi \colon H_k(W;M) \to H_k(W,\partial W;M) \xrightarrow{\text{PD}} H^{n-k}(W;M) \xrightarrow{\ev}\makeithash{\op{Hom}_{\text{left-}R}(H_{n-k}(W;M),S)}.$$
Restricting to the case where $W$ is $2n$-dimensional, the \emph{twisted intersection pairing}
$$\lambda_{M,W} \colon H_n(W;M) \times H_n(W;M) \to S $$
is defined by $\lambda_{M,W}(x,y)=\Phi(y)(x)$.
Note that while $\lambda_{M,W}$ is $(-1)^n$-Hermitian, it may be singular: the submodule $\im (H_{n}(\partial W;M) \to H_{n}(W;M)) $ is annihilated by~$\lambda_{M,W}$ and the evaluation map need not be an isomorphism.

\subsection{The twisted Blanchfield pairing}
\label{sub:Blanchfield}
In this subsection, we review twisted Blanchfield pairings.
While the construction can be performed over non-commutative rings~\cite{Powell, CochranOrrTeichner,CochranNonCommutative}, we focus first on the case $R=\F[t^{\pm 1}]$.
References on twisted Blanchfield pairings include~\cite{MillerPowell, Powell}.
\medbreak

Let $N$ be a closed $3$-manifold.
Just as in Example~\ref{ex:UsualCoeff}, let $\beta \colon \pi_1(N) \to GL_d(\LF)$ be a unitary representation and use $\F[t^{\pm 1}]^d_\beta$ to denote the $(\LF,\Z[\pi_1(X)])$-bimodule whose right~$\Z[\pi_1(N)]$-module structure is given by right multiplication by $\beta(\gamma)$ on row vectors. In what follows, we shall think of $\F[t^{\pm 1}], \F(t)$ and $\F(t)/\F[t^{\pm 1}]$ as $(\F[t^{\pm 1}],\F[t^{\pm 1}])$-bimodules.
Since
\[\F(t)^{d}_{\beta} := \F(t^{\pm 1}) \otimes_{\F[t^{\pm 1}]}~\F[t^{\pm 1}]^d_\beta \quad \text{and} \quad (\F(t)/\F[t^{\pm1}])^{d}_{\beta} := \F(t)/\F[t^{\pm 1}] \otimes_{\F[t^{\pm 1}]}~\F[t^{\pm 1}]^d_\beta\]
are $(\F[t^{\pm 1}],\Z[\pi_1(N)])$-bimodules, there is a short exact sequence
$$ 0 \to C^*(N; \F[t^{\pm 1}]^d_\beta) \to C^*(N,\F(t)^{d}_{\beta}) \to C^*(N;(\F(t)/ \F[t^{\pm1}])^{d}_{\beta}) \to 0$$
of cochain complexes of left $\F[t^{\pm 1}]$-modules, where we identified $\F[t^{\pm 1}]^d_\beta$ with $\F[t^{\pm 1}] \otimes_{\LF} \F[t^{\pm 1}]^d_\beta $.
Passing to cohomology, we obtain a long exact sequence in which the connecting map 
$$ \operatorname{BS} \colon  H^{k}(N;(\F(t)/\F[t^{\pm 1}])^{d}_{\beta}) \to H^{k+1}(N;\F[t^{\pm 1}]^d_\beta) $$
is called the \emph{Bockstein homomorphism}.
We say that $\beta$ is \emph{$H_1$-null}
 if $H_1(N;\F(t)^{d}_{\beta})$ vanishes, i.e.~if the $\LF$-module $H_1(N;\LF_\beta^d)$ is torsion. 
Observe that if $\beta$ is $H_1$-null, then the corresponding Bockstein homomorphism is an isomorphim: indeed Poincar\'e duality and the universal coefficient theorem respectively imply that $H^2(N;\F(t)^d_\beta)$
and $H^1(N;(\F(t) /\F[t^{\pm1}])^d_\beta)$ vanish.

Consider the composition
\begin{align*}
  \Theta \colon H_1(N;\F[t^{\pm 1}]^d_\beta) \,
  & \xrightarrow{PD}\, H^2(N;\F[t^{\pm 1}]^d_\beta)  \\
  & \xrightarrow{BS^{-1}}   H^1(N;(\F(t)/\F[t^{\pm 1}])^d_\beta) \\
  & \xrightarrow{\text{ev}}  \makeithash{\Hom_{\F[t^{\pm 1}]}(H_1(N;\F[t^{\pm 1}]^d_\beta),\F(t)/\F[t^{\pm 1}])}
\end{align*}
of the three following $\F[t^{\pm 1}]$-homomorphisms: Poincar\'e duality, the inverse Bocktein and the evaluation map described in Example~\ref{ex:UsualCoeff}.
The main definition of this section is the following.

\begin{definition}
  \label{def:Blanchfield}
  Let $N$ be a closed oriented $3$-manifold and let $\beta \colon \pi_1(N) \to GL_d(\F[t^{\pm 1}])$ be a unitary $H_1$-null representation.
  The \emph{twisted Blanchfield pairing associated to $\beta$}
  $$ \Bl_\beta(N) \colon  H_1(N;\F[t^{\pm 1}]^d_\beta) \times H_1(N;\F[t^{\pm 1}]^d_\beta) \to \F(t)/\F[t^{\pm 1}]$$
  is defined as $\Bl_\beta(N)(x,y)=\Theta(y)(x)$.
  If $N=M_K$ is the $0$-framed surgery along a knot $K$, then we write $\Bl_\beta(K)$ instead of $\Bl_\beta(M_K)$.
\end{definition}

We refer to~\cite{Powell} for a more general treatment of twisted Blanchfield pairings, but make two remarks nonetheless.
Firstly, Powell proved that $\Bl_\beta(N)$ is Hermitian~\cite{Powell}.
Secondly, since $\F[t^{\pm 1}]$ is a PID, it follows that $\Bl_\beta(N)$ is non-singular: indeed $\F(t)/\F[t^{\pm 1}]$ is an injective $\F[t^{\pm 1}]$-module and thus the evaluation map is an isomorphism.

Finally, although we mostly focus on \emph{closed} 3-manifolds, we conclude this subsection with a remark on manifolds with boundary.  To describe this case, we say that a representation $\beta \colon \pi_1(N) \to~GL_d(\LF)$ is \emph{acyclic} if $H_i(N;\F[t^{\pm 1}]_\beta^d)$ is a torsion $\F[t^{\pm 1}]$-module for each $i$.

\begin{remark}
  \label{rem:Boundary}
  If $N$ is a 3-manifold with boundary and $\beta \colon \pi_1(N) \to GL_d(\LF)$ is a unitary acyclic representation, then one also obtains a twisted Blanchfield pairing $\operatorname{Bl}_\beta(N)$.
  Indeed, one starts from the map $i \colon H_1(N;\LF_\beta^d) \to H_1(N,\partial N;\LF_\beta^d)$ induced by the inclusion $(N,\emptyset) \to~(N,\partial N)$ and then proceeds as in the closed case, using duality, the inverse of the Bockstein homorphism and evalutation.
  The acyclicity assumption is used to guarantee that the Bockstein homomorphism is an isomorphism: $H_1$-nullity is not enough when the boundary is non-empty, although an alternative is to work on the torsion submodule of $H_1(N;\LF_\beta^d)$ and to proceed as in~\cite{Powell, Hillman} or~\cite[Subsection 2.2]{ConwayBlanchfield}.
  Finally, note that when $\partial N$ is non-empty, $\operatorname{Bl}_\beta(N)$ may be singular: the map~$i$ need not be an isomorphism.
\end{remark}

\subsection{Relation to Blanchfield pairings over noncommutative rings}
\label{sub:Ore}
Twisted Blanchfield pairing are also defined in the non-commutative setting, using Ore rings~\cite{CochranOrrTeichner, FriedlLeidyNagelPowell,Powell}.
In this subsection, after briefly reviewing some basics on these rings, we relate the resulting non-commutative Blanchfield pairings to the  twisted Blanchfield pairings of Subsection~\ref{sub:Blanchfield}.
The reason for this incursion into the non-commutative setting is that the proof of Theorem~\ref{thm:CablingTheorem} requires that we import a result  from~\cite{FriedlLeidyNagelPowell} (which uses non-commutative Blanchfield forms) into the setting of twisted Blanchfield forms.

\medbreak
We start with some brief recollections on Ore rings, referring to~\cite[Chapter 4, Section 10]{Lam-modules} for further details.
Let $\mathcal{R}$ be a ring.
A multiplicative subset $S \subset \mathcal{R}$ is \emph{right permutable} if, for any~$r \in \mathcal{R}$ and $s \in S$, we have $r S \cap s\mathcal{R} \neq \emptyset$.
An element $r \in \mathcal{R}$ is \emph{regular} if it is neither a left zero divisor nor a right  zero divisor.
We say that $\mathcal{R}$ is a \emph{right Ore ring} if the set $S$ of regular elements is right permutable.
In this case, $\mathcal{R}$ can be localized at $S$ and $\mathcal{R}S^{-1}$ is called the \emph{right ring of quotients of $\mathcal{R}$}.
The left analogues of these notions are defined similarly.
If $\mathcal{R}$ is both a left Ore ring and a right Ore ring, then $\mathcal{R}$ is called an \emph{Ore ring}.
In this case, the right and left rings of quotients agree, and we simply write $\mathcal{Q}$.

Commutative rings are Ore rings~\cite[Chapter 4, Section 10, 10.18]{Lam-modules}.
In what follows however, we shall mostly be concerned with rings of matrices.

\begin{example}
  \label{ex:MatrixOre}
  If $R$ is a commutative ring, then the matrix ring \(\mathcal{M}_{d}(R)\) is an Ore ring~\cite[Example 11.21(1)]{Lam-modules}.
  For concreteness, we focus on the case where $R=\LF$.
  In this case, by~\cite[Example~11.21(1)]{Lam-modules} and~\cite[Proposition~10.21]{Lam-modules} the ring of quotients of \(\mathcal{M}_{d}(\F[t^{\pm1}])\) is isomorphic to \(\mathcal{M}_{d}(\F(t))\).
\end{example}

Next, suppose that \(R\) is a commutative ring with unit.
Right multiplication on row vectors endows $R^{d}$ with a right \(\mathcal{M}_{d}(R)\)-module structure. We write \(R^{d}_{\mathcal{M}}\) for emphasis.
An alternate description of \(R^{d}_{\mathcal{M}}\) is obtained by considering the elementary matrix \(E_{i,j}\) whose only nonzero coefficient is a~$1$ in its $(i,j)$-entry.
Indeed, one has the isomorphim \( E_{i,i}\mathcal{M}_d(R) \cong R^{d}_{\mathcal{M}}\) of right $\mathcal{M}_d(R)$-modules.

The next lemma shows that $\mathcal{M}_d(R)$ decomposes as a direct sum of $d$ copies of~$R^{d}_{\mathcal{M}}$.

\begin{lemma}\label{obs:proj-modules}
  If \(R\) is a commutative ring with unit, then \(R^{d}_{\mathcal{M}}\) is a projective right \(\mathcal{M}_{d}(R)\)-module.
  In fact, we have the following isomorphism of right \(\mathcal{M}_{d}(R)\)-modules:
  \begin{equation}
    \label{eq:DecompositionOre}
    \mathcal{M}_{d}(R) \cong  \bigoplus_{i=1}^d E_{i,i}\mathcal{M}_{d}(R) \cong  \bigoplus_{i=1}^d R^{d}_{\mathcal{M}}.
  \end{equation}
\end{lemma}
\begin{proof}
  First, notice that the \(E_{i,i}\) are orthogonal idempotents: one has  $E_{i,i} E_{j,j}=\delta_{ij}E_{i,j}$.
  The conclusion of the lemma now follows since $\id = E_{1,1} + \cdots +E_{d,d}$.
\end{proof}

Next, assume that $X$ is a CW complex together with a homomorphism $\gamma \colon \Z[\pi_1(X)] \to \mathcal{R}$ to an Ore ring $\mathcal{R}$.
Just as in Subsection~\ref{sub:TwistedHomology}, this endows $\mathcal{R}$ with a right $\Z[\pi_1(X)]$-module structure which we denote by $\mathcal{R}_\gamma$ for emphasis.
Given a commutative ring $R$ with unit, note that any representation \(\gamma \colon \pi_{1}(X) \to GL_d(R)\) canonically extends to a ring homomorphism \(\gamma \colon \Z[\pi_{1}(X)] \to \mathcal{M}_{d}(R)\).
In particular $\gamma$ endows $R^d$ and $\mathcal{R}=\mathcal{M}_d(R)$ with right $\Z[\pi_{1}(X)]$-module structures.

The following proposition is a topological application of Lemma~\ref{obs:proj-modules}.

\begin{proposition}
  \label{prop:OreTwisted}
  Let \(R\) be a commutative ring with unit.
  If $\gamma \colon \pi_1(X) \to GL_d(R)$ is a representation, then we have the following chain isomorphisms of chain complexes of left $R$-modules:
  \begin{align*}
    & C_*(X;\mathcal{M}_d(R)_\gamma) \cong \bigoplus_{i=1}^d C_*(X;R_\gamma^d),  \\
    &C^*(X;\mathcal{M}_d(R)_\gamma) \cong \bigoplus_{i=1}^d C^*(X;R_\gamma^d).
  \end{align*}
  Furthermore, one also has the following isomorphisms of left $R$-modules:
  \begin{align*}
    & H_*(X;R_\gamma^d) \cong R^d_{\mathcal{M}} \otimes_{\mathcal{M}_d(R)} H_*(X;\mathcal{M}_d(R)_\gamma) ,  \\
    & H^*(X;R_\gamma^d) \cong R^d_{\mathcal{M}} \otimes_{\mathcal{M}_d(R)} H^*(X;\mathcal{M}_d(R)_\gamma).
  \end{align*}
\end{proposition}
\begin{proof}
  The chain isomorphisms follow from the decomposition of $\mathcal{M}_d(R)$ displayed in~\eqref{eq:DecompositionOre}.
  To deal with the second assertion, we once again use Lemma~\ref{obs:proj-modules} to deduce that $\mathcal{M}_d(R)$ is a projective (and thus flat) $R_\mathcal{M}^d$-module.
  We therefore obtain the following isomorphisms:
  $$ H_*(X;R_\gamma^d) \cong H_*(X; R_{\mathcal{M}}^d \otimes_{\mathcal{M}_d(R)} \mathcal{M}_d(R)_\gamma) \cong R_{\mathcal{M}}^d \otimes_{\mathcal{M}_d(R)} H_*(X;  \mathcal{M}_d(R)_\gamma).$$
  The proof for cohomology is analogous.
  This concludes the proof of the lemma.
\end{proof}


Next, we deal with non-commutative Blanchfield pairings.
Let $N$ be a closed oriented $3$-manifold.
Let $\mathcal{R}$ be an Ore ring with involution and use~$\mathcal{Q}$ to denote its ring of quotients.
Given a homomorphism $\gamma \colon \Z[\pi_1(N)] \to \mathcal{R}$, we can view $\mathcal{R},\mathcal{Q}$ and~$\mathcal{Q}/\mathcal{R}$ as right $\Z[\pi_1(N)]$-modules.
In particular, as in Subsections~\ref{sub:TwistedHomology} and~\ref{sub:Blanchfield}, there are Poincar\'e duality isomorphisms and a Bockstein homomorphism on cohomology. Thus, if $H_1(N;\mathcal{Q}_\gamma)=0$, then one can consider the composition
$$ \Phi \colon H_1(N;\mathcal{R}_\gamma) \xrightarrow{\operatorname{PD}}  H^2(N;\mathcal{R}_\gamma) \xrightarrow{BS^{-1}} H^1(N;\mathcal{Q}/\mathcal{R}_\gamma) \xrightarrow{\operatorname{ev} \circ \kappa} \operatorname{Hom}_\mathcal{R}(H_1(N;\mathcal{R}_\gamma);\mathcal{Q}/\mathcal{R})^\#.$$
We refer to~\cite[Section 2]{FriedlLeidyNagelPowell} for further details, but the upshot is that there is a Blanchfield pairing in the non-commutative setting which is defined exactly as in Subsection~\ref{sub:Blanchfield}.

\begin{definition}
  \label{def:NonCommutativeBlanchfield}
  Let $N$ be a closed oriented $3$-manifold together with a map $\gamma \colon \Z[\pi_1(N)] \to \mathcal{R}$ to an Ore ring with involution.
  Assuming that $H_1(N;\mathcal{Q}_\gamma)=0$, there is a \emph{non-commutative Blanchfield pairing}
  $$ \Bl_{\mathcal{R},\gamma}(N) \colon H_1(N;\mathcal{R}_\gamma) \times H_1(N;\mathcal{R}_\gamma) \to \mathcal{Q}/\mathcal{R}$$
  which is defined by setting $\Bl_{\mathcal{R},\gamma}(N)(x,y)=\Phi(y)(x)$.
\end{definition}

Note that in the literature $\mathcal{R}$, is usually assumed to be an Ore \emph{domain}, however the definition extends \emph{verbatim} to the case of Ore rings.
The next proposition relates the non-commutative Blanchfield pairing to the twisted Blanchfield pairing.




\begin{proposition}
  \label{prop:BlanchfieldOreTwisted}
  Let \(N\) be a closed oriented \(3\)-manifold, and let \(\gamma \colon \pi_{1}(N) \to GL_d(\LF)\) be a representation.
  The following assertions hold:
  \begin{enumerate}
  \item We have $H_*(N;\F(t)^d_\gamma)=0$ if and only if \(H_*(N;\mathcal{M}_{d}(\F(t))_\gamma)=0\).
  \item The following $\LF$-linking forms are canonically isometric:
    \[\F[t^{\pm1}]^{d}_{\mathcal{M}} \otimes_{\mathcal{M}_{d}(\F[t^{\pm1}])} \Bl_{\mathcal{M}_{d}(\F[t^{\pm1}]),\gamma}(N) \cong \Bl_{\gamma}(N).\]
  \end{enumerate}
\end{proposition}
\begin{proof}
  The first assertion follows immediately from the first assertion of Proposition~\ref{prop:OreTwisted}.
  Next, we deal with the second assertion.
  Using once again the first assertion of Proposition~\ref{prop:OreTwisted}, the linking form $\Bl_{\mathcal{M}_d(\LF),\gamma}(N)$ can canonically be thought of as a linking form
  $$\Bl_{\mathcal{M}_d(\LF),\gamma}(N) \colon \bigoplus_{i=1}^d H_1(N;\LF_\gamma^d) \times \bigoplus_{i=1}^d H_1(N;\LF_\gamma^d)  \to \mathcal{M}_d(\F(t))/\mathcal{M}_d(\LF),$$
  where the above decomposition is obtained by using the matrices \(E_{i,i}\), for \(1 \leq i \leq d\).
  The result now follows by tensoring with $\LF_\mathcal{M}^d$ and using the second assertion of Proposition~\ref{prop:OreTwisted}.
  Indeed, both Blanchfield pairings are defined by composing Poincar\'e duality isomorphisms, inverse Bockstein isomorphisms and evaluation maps.
  The first two maps preserve the direct sum decomposition displayed above. After applying the adjoint of evaluation map, and using the identification from Proposition~\ref{prop:OreTwisted}, $\Bl_{\mathcal{M}_d(\LF)}$ is given by
  \[(E_{i,i} \cdot x, E_{j,j} \cdot y) \mapsto\Bl_{\gamma}(x,y) E_{i,j} \in \mathcal{M}_{d}(\F(t) / \LF),\]
  for \(x,y \in H_{1}(N;\LF^{d}_{\gamma})\).
  This concludes the proof of the proposition.
\end{proof}

Informally, Proposition~\ref{prop:BlanchfieldOreTwisted} implies that several statements that hold for non-commutative Blanchfield pairings also hold for twisted Blanchfield pairings.
This will be used in Theorem~\ref{thm:CablingTheorem}.

\section{Twisted signatures of 3-manifolds and knots}
\label{sec:TwistedSign3ManifKnot}

This section is organized as follows.
In Subsection~\ref{sub:TwistedSignatures}, we introduce our twisted signature function and describe its basic properties, in Subsection~\ref{sub:BordismSignature}, we provide conditions for its bordism invariance and in Subsection~\ref{sub:Satellite}, we investigate its behavior under satellite operations.

\subsection{Definition of the twisted signatures }
\label{sub:TwistedSignatures}
In this short subsection, we define our twisted signature invariants for knots and 3-manifolds.
To this effect, recall that Proposition~\ref{prop:SignatureFunctionAlgebra} associated to a linking form $(M,\lambda)$ a locally constant functions
$\sigma_{(M,\lambda)} \colon S^1 \to \Z $ and $\sigma_{(M,\lambda)}^{av} \colon S^1 \to \Z$ with desirable properties.

\begin{definition}
  \label{def:TwistedSignature}
  Let $N$ be a closed connected oriented $3$-manifold and let $\beta \colon \pi_1(N) \to GL_d(\F[t^{\pm 1}])$ be a unitary $H_1$-null representation.
  The \emph{twisted signature function} $\sigma_{N,\beta} \colon S^1 \to \Z$ of the pair~$(N,\beta)$ is the signature function of the non-singular linking form $\Bl_\beta(N)$.
  The \emph{averaged signature function}~$\sigma_{N,\beta}^{\operatorname{av}}$ is the averaged signature function of $\Bl_\beta(N)$.
\end{definition}

The corresponding invariants for a knot $K$ are obtained by applying Definition~\ref{def:TwistedSignature} to the 0-framed surgery $M_K$.

\begin{definition}
  \label{def:TwistedKnotSignature}
  Let $K$ be an oriented knot and let $\beta \colon \pi_1(M_K) \to GL_d(\F[t^{\pm 1}])$ be a unitary~$H_1$-null representation.
  The \emph{twisted signature function} $\sigma_{K,\beta}\colon S^1\to\Z$ of $K$ is the twisted signature function of~$M_K$ and similarly for the \emph{averaged signature function} $ \sigma_{K,\beta}^{\operatorname{av}}$.
\end{definition}

Before proving properties, we show how Definition~\ref{def:TwistedKnotSignature} recovers the Levine-Tristram signature.
This establishes the fourth property of Theorem~\ref{thm:Recap} from the introduction.

\begin{remark}
  \label{rem:LevineTristram}
  If $\beta$ is the abelianization map $\phi_K$, then $\Bl_\beta(K)$ coincides with the untwisted Blanchfield pairing $\operatorname{Bl}(K)$ and the corresponding signature function recovers the Levine-Tristram signature $\sigma_K$ thanks to~\cite[Lemma 3.2]{BorodzikFriedl}. Note that $\operatorname{Bl}(K)$ is usually defined on the Alexander module~$H_1(X_K;\LF)$, where $X_K=S^3 \setminus \mathcal{N}(K)$ is the complement of a tubular neighborhood~$\mathcal{N}(K)$ of~$K$ in $S^3$.
  It is however well known that the inclusion $X_K \hookrightarrow M_K$ induces a Blanchfield form preserving isomorphism $H_1(X_K;\LF) \cong H_1(M_K;\LF)$. 
\end{remark}


Contrarily to $\sigma_K$, the twisted signature function~$\sigma_{K,\beta}$ is not known to be symmetric (i.e.~$\sigma_{K,\beta}(\overline{\omega})$ is not known to equal~$\sigma_{K,\beta}(\omega)$).
More precisely,  in~\cite{BCP_Alg} we show that signature functions associated to abstract linking forms over \(\LC\) are not symmetric but it remains unknown whether such linking forms can be realized as twisted Blanchfield forms.
Despite this difference, the next proposition shows that the twisted signature shares several properties with its classical counterpart.
This establishes the first item of Theorem~\ref{thm:Recap} from the introduction. 

\begin{proposition}
  \label{prop:LocallyConstant}
  Let $K$ be an oriented knot and let $\beta \colon \pi_1(M_K) \to GL_d(\F[t^{\pm 1}])$ be a unitary $H_1$-null representation.
  \begin{enumerate}
  \item The twisted signature $\sigma_{K,\beta}$ is constant on the complement in~$S^1$ of the zero set of the twisted Alexander polynomial $\Delta_{M_K}^\beta(t):=\operatorname{Ord}(H_1(M_K;\F[t^{\pm 1}]_\beta^d))$.
  \item Let $\overline{K}$ denote the mirror image of $K$. The representation $\beta$ canonically defines a representation on $\pi_1(M_{\overline{K}})$ (which we also denote by $\beta$) and, for all $\omega \in S^1$, we have
    \begin{align*}
      &\Bl_\beta(\overline{K})=-\Bl_\beta(K), \\
      &\sigma_{\overline{K},\beta}=-\sigma_{K,\beta}.
    \end{align*}
  \item Let $K^r$ denote $K$ with its orientation reversed. Assume that $\beta(g)=\phi_K(g)\rho(g)$, where $\rho \colon \pi_1(M_K) \to GL_n(\F)$ is a unitary representation. The representation $\beta$ canonically defines a representation on $\pi_1(M_{K^r})$ (which we also denote by $\beta$) and, for all $\omega \in S^1$, we~have
    \begin{align*}
      & \Bl_{\beta}(K^r)(t) =\Bl_{\beta}(K^r)(t^{-1}) , \\
      & \sigma_{K^r,\beta}(\omega) =\sigma_{K,\beta}(\overline{\omega}).
    \end{align*}
  \end{enumerate}
\end{proposition}
\begin{proof}
  The proof of the first property follows by applying
  item \ref{item:S1} of Proposition~\ref{prop:SignatureFunctionAlgebra}
  to the twisted Blanchfield pairing. To prove the second assertion, let $h \colon S^3 \to S^3$ be an orientation reversing homeomorphism such that $h(K)=\overline{K}$. Note that $h$ induces an orientation reversion homeomorphism $h \colon M_K \to M_{\overline{K}}$. A diagram chase involving the definition of the Blanchfield pairing now shows that $h$ induces an isometry $\Bl_\beta(\overline{K})\cong -\Bl_\beta(K)$, the key point being that the square involving Poincar\'e duality anti-commutes thanks to the identity $h_*(h^*(\xi) \cap [M_K])=\xi \cap h_*([M_K])=-\xi \cap [M_{\overline{K}}]$ for any~$\xi \in H^2(M_K;\LF^n_\beta)$. The result on the signature now follows from the definition. The third assertion is a consequence of the identity $\phi_K(\mu_{K^r})=\phi_K(\mu_K)^{-1}$.
\end{proof}

The order of $H_1(M_K;\F[t^{\pm 1}]_\beta^d)$ can be related to the twisted Alexander polynomial of $K$ which is usually defined on the exterior $X_K$ of $K$; see~\cite[Lemma 6.3]{KirkLivingston} and~\cite[Lemma 3]{FriedlVidussiSurvey}. In the abelian case, since $H_1(M_K;\LF) \cong H_1(X_K;\LF)$, the first item of Proposition~\ref{prop:LocallyConstant} (and Remark~\ref{rem:LevineTristram}) recovers the well known fact that the Levine-Tristram signature function $\sigma_K$ is constant on the complement of the zero set of the Alexander polynomial $\Delta_K(t)$. The second item of Proposition~\ref{prop:LocallyConstant} is also a direct generalization of the corresponding property for the Levine-Tristram signature. 

\begin{remark}
  \label{rem:NotSymmetric}
  Since the Levine-Tristram is symmetric, it satisfies $\sigma_{K^r}(\omega)=\sigma_{K}(\omega)$. This does not appear to be the case in general since the signature function of a general complex linking form need not be symmetric.
\end{remark}

Motivated by Remark~\ref{rem:LevineTristram} and Proposition~\ref{prop:LocallyConstant}, the next subsections investigate to what extent the properties of the Levine-Tristram signature generalize to our twisted signatures.

\subsection{Bordism and concordance invariance}
\label{sub:BordismSignature}
In this subsection, we give conditions under which the signatures of 3-manifolds (resp. knots) are bordism (resp. concordance) invariants.
\medbreak

The proof of the next lemma can be found in~\cite[Proposition 2.8]{Letsche} or~\cite[Theorem 4.4]{CochranOrrTeichner}.

\begin{lemma}\label{lem:metabolizer}
  Let \(N_{1},\ldots,N_{k}\) be closed connected oriented \(3\)-manifolds.
  Let \(\beta_{i} \colon \pi_{1}(N_{i}) \to GL_d(\LF)\) be a unitary acyclic representation for $i=1,\ldots, k$.
  Assume that \(W\) is a \(4\)-manifold such that \(\partial W = N_{1} \sqcup \ldots \sqcup N_{k}\) and the two following conditions are satisfied:
  \begin{itemize}
  \item[(a)] for $i=1,\ldots, k, $ there exists a unitary representation \(\gamma \colon \pi_{1}(W) \to GL_d(\LF)\) extending the representations $\beta_{i}$;
  \item[(b)] the following sequence is exact:
    \begin{equation}
      \label{eq:necessary-cond-metabolic-form}
      TH_{2}(W, \partial W; \F[t^{\pm 1}]_{\gamma}^{d}) \xrightarrow{\partial} H_{1}(\partial W;\F[t^{\pm 1}]_{\gamma}^{d}) \xrightarrow{\iota_{\ast}} H_{1}(W; \F[t^{\pm 1}]_{\gamma}^{d}).
    \end{equation}
  \end{itemize}
  Then the linking form \(\oplus_{i=1}^{k} \Bl_{\beta_{i}}(N_{i})\) is metabolic.
\end{lemma}

The following result uses Lemma~\ref{lem:metabolizer} to provide (slightly less restrictive) assumptions under which the twisted Blanchfield pairing is metabolic.

\begin{proposition}
  \label{prop:MetabolicAcyclic}
  Let \(N = N_{1} \sqcup \ldots \sqcup N_{k}\) be a disjoint union of closed \(3\)-manifolds and let \(\beta_{i} \colon \pi_{1}(N_{i}) \to~GL_d(\LF)\) be a unitary acyclic representation for $i=1,\ldots, k$.
  If $N$ bounds a compact \(4\)-manifold~\(W\) and if the \(\beta_{i}\) extend to a unitary acyclic representation \(\gamma \colon \pi_{1}(W) \to~GL_d(\F[t^{\pm 1}])\), then \(\oplus_{i=1}^{k}\Bl_{\beta_{i}}(N_{i})\) is metabolic and, in particular, the sum $\sum_{i=1}^{k}\sigma^{\text{av}}_{N_i,\beta_{i}}$ of averaged signature functions is zero.
\end{proposition}
\begin{proof}
  Since $\gamma$ is acyclic,  $H_*(W;\F(t)_\gamma^d)$ vanishes.
  Poincar\'e duality and the universal coefficient theorem (where the underlying ring is the PID $R=\LF$) now imply that $H_*(W,N;\F(t)_\gamma^d)$ also vanishes.
  Since the $\beta_i$ are acyclic, the homology~$\F[t^{\pm 1}]$-modules of $N,W$ and $(W,N)$ are $\F[t^{\pm 1}]$-torsion and therefore the sequence~(\ref{eq:necessary-cond-metabolic-form}) is exact.
  The proposition now follows by applying Lemma~\ref{lem:metabolizer} (to deduce that $\oplus_{i=1}^{k}\Bl_{\beta_{i}}(N_{i})$ is metabolic) and 
  item~\ref{item:S5} of Proposition~\ref{prop:SignatureFunctionAlgebra} (to obtain that the sum of the signature functions is zero).
\end{proof}

Next, suppose that $K$ and $K'$ are two concordant knots in $S^3$ and let $M_K$ and $M_{K'}$ denote their respective $0$-framed surgeries.
Given a locally flat concordance \(\Sigma \subset S^{3} \times I\) from \(K\) to \(K'\), we can construct a cobordism \(W_{\Sigma}\) from \(M_{K}\) to \(M_{K'}\).
Indeed, let \(\mathcal{N}(\Sigma)\) denote a tubular neighborhood of~\(\Sigma\) (which can be identified with \(\mathcal{N}(\Sigma) \cong S^{1} \times D^{2} \times I\)). %
\footnote{The existence of such neigbhorhoods is due to Freedman-Quinn~\cite[Section 9]{FreedmanQuinn}; see also~\cite[Theorem 6.8]{FriedlNagelOrsonPowell}.}
  The desired cobordism is then obtained by setting
\[W_{\Sigma} = ((S^{3} \times I) \setminus \mathcal{N}(\Sigma)) \cup_{h} S^{1} \times D^{2} \times I,\]
where \(h \colon S^{1} \times S^{1} \times I \to S^{1} \times S^{1} \times I\) is the homeomorphism defined by \(h(x,y,z) = (y,x,z)\).
The next result shows that the averaged twisted signature is a concordance invariant as long as the corresponding representations extend through the complement of the concordance to an acyclic representation.

\begin{proposition}\label{prop:witt-equiv-concordance}
  Let \(K_{1}, K_{2}\) be two concordant knots in $S^3$, and let \(\beta_{i} \colon \pi_{1}(M_{K_{i}}) \to GL_d(\F[t^{\pm 1}])\) be a unitary acyclic representation for \(i=1,2\).
  Assume there is a concordance \(\Sigma \subset S^{3} \times I\)  from~\(K_{1}\) to~\(K_{2}\) and a unitary representation \(\gamma \colon \pi_{1}(W_{\Sigma}) \to GL_d(\F[t^{\pm 1}])\) such that
  \begin{itemize}
  \item[(a)] the representations \(\beta_{i}\) factor as \(\gamma \circ (\iota_{i})_{\ast}\), where \(\iota_{i} \colon M_{K_{i}} \hookrightarrow W_{\Sigma}\) is the inclusion for \(i=1,2\);
  \item[(b)] the representation $\gamma$ is acyclic.
  \end{itemize}
  Then the twisted Blanchfield pairings \(\Bl_{\beta_{1}}(K_1)\) and \(\Bl_{\beta_{2}}(K_2)\) are Witt equivalent and, in particular, the averaged signature functions \(\sigma^{av}_{K_{1},\beta_{1}}\) and $\sigma^{av}_{K_{2},\beta_{2}}$ agree.
\end{proposition}
\begin{proof}
  Set \(N = M_{K_{1}} \sqcup -M_{K_{2}}\).
  By our assumptions, \(W_{\Sigma}\) is a null-bordism for \(N\) and satisfies the conditions of Proposition~\ref{prop:MetabolicAcyclic}.
  Therefore, \(\Bl_{\beta_{1}}(M_{K_{1}}) \oplus \Bl_{\beta_{2}}(-M_{K_{2}})\) is metabolic.
  Since \(\Bl_{\beta_{2}}(-M_{K_{2}}) = -\Bl_{\beta_{2}}(M_{K_{2}})\) (use the same reasoning as in the proof of Proposition~\ref{prop:LocallyConstant}), it follows that \(\Bl_{\beta_{1}}(K_1)\) and \(\Bl_{\beta_{2}}(K_2)\) are Witt equivalent and, consequently the averaged signature functions agree: \(\sigma^{\operatorname{av}}_{K_{1},\beta_{1}}=\sigma^{\operatorname{av}}_{K_{2},\beta_{2}}\).
\end{proof}

The same proof yields a result for slice knots with a representation which extends to an acylic representation over the slice disk exterior. We record this result for completeness, as it was mentioned in the second item of Theorem~\ref{thm:Recap} from the introduction.

\begin{proposition}\label{prop:SignatureSlice}
  Let \(K\) be a slice knot and let \(\beta \colon \pi_{1}(M_K) \to GL_d(\F[t^{\pm 1}])\) be a unitary acyclic representation. If $D\subset D^4 $ is a slice disk for $K$ and there is a unitary acyclic representation \(\gamma \colon \pi_{1}(D^4 \setminus \mathcal{N}(D)) \to GL_d(\F[t^{\pm 1}])\) extending $\beta$,  then the twisted Blanchfield pairings \(\Bl_{\beta}(K)\) is metabolic and, in particular, the averaged signature function \(\sigma^{av}_{K,\beta}\) is zero.
\end{proposition}

\subsection{Satellite formulas for twisted signatures}
\label{sub:Satellite}

In this subsection, we use a result of Friedl-Leidy-Nagel-Powell~\cite{FriedlLeidyNagelPowell} to understand the behavior of the twisted Blanchfield pairings and signatures under various satellite operations. In particular, these results provide twisted generalizations of well known theorems for the classical Blanchfield pairing and Levine-Tristram signature.
\medbreak

Let \(Y\) be a \(3\)-manifold with empty or toroidal boundary.
Suppose that we are given a simple closed curve \(\eta\) in \(Y\).
If \(K \subset S^{3}\) is a knot, then an \emph{infection} of~\(Y\) by \(K\) is a \(3\)-manifold \(Y_{K}\) obtained by gluing \(Y_{K} = (Y \setminus \mathcal{N}(\eta)) \cup_{\partial} S^{3} \setminus \mathcal{N}(K)\), where we glue a zero-framed longitude of \(K\) to the meridian of \(\eta\) and some longitude of \(\eta\) to the meridian of \(K\).
Notice that the ambiguity in the choice of the longitude of \(\eta\) may lead to different outcomes, however, as pointed out in~\cite{FriedlLeidyNagelPowell}, the resulting twisted Blanchfield pairing does not depend on this choice.
Note that if $Y$ has a boundary, then so does $Y_K$.  In this case, we use Remark~\ref{rem:Boundary} which describes twisted Blanchfield pairings for $3$-manifolds with boundary.

Let \(P \subset S^{3}\) be a knot and let \(M_{P}\) be the zero-surgery on \(P\).
If \(K \subset S^{3}\) is another knot and \(\eta\) is a simple closed curve in the complement of \(P\), we can form the \emph{satellite knot} \(P(K,\eta)\) with \emph{pattern}~\(P\), \emph{companion} \(C\) and \emph{infection curve} \(\eta\) by looking at the image of \(P\) under the diffeomorphism \((S^{3}\setminus \mathcal{N}(\eta)) \cup_{\partial} (S^{3}  \setminus \mathcal{N}(K)) \cong S^{3}\), where the gluing of the exteriors of \(\eta\) and \(K\) identifies the meridian of \(\eta\) with the zero-framed longitude of \(K\) and vice-versa.

From this description, one can see that the zero-surgery on the satellite knot \(P(K,\eta)\) can be obtained by infection of~\(M_{P}\) by \(K\) along \(\eta \subset S^{3} \setminus P \subset M_{P}\):
\begin{equation}
  \label{eq:DecompoSurgerySatellite}
  M_{P(K,\eta)}=M_P \setminus \mathcal{N}(\eta) \cup_\partial S^3 \setminus \mathcal{N}(K).
\end{equation}
Let \(\mu_{\eta}\) denote a meridian of \(\eta\).
Our goal in this section is to express the Blanchfield form~\(\Bl_\gamma(M_{P(K,\eta)})\) in terms of the twisted Blanchfield forms on \(M_{P}\) and \(M_{K}\).
Using~\eqref{eq:DecompoSurgerySatellite}, notice that given a representation \(\gamma \colon \pi_{1}(M_{P(K,\eta)}) \to GL_d(\F[t^{\pm1}])\), there are induced representations
\[\gamma_{P} \colon \pi_{1}(M_{P} \setminus \mathcal{N}(\eta)) \to GL_d(\F[t^{\pm1}]), \quad \gamma_{K} \colon \pi_{1}(S^{3} \setminus \mathcal{N}(K)) \to GL_d(\F[t^{\pm1}]).\]
In general however, there is no guarantee that these representations can be extended to representations of the fundamental groups of \(M_{P}\) and \(M_{K}\).

\begin{remark}
  \label{rem:EtaRegular}
  In fact, \(\gamma_{P}\) and \(\gamma_{K}\) extend to representations of the respective zero-surgeries if and only if \(\gamma(\mu_{\eta}) = 1\).
  This follows from the following two observations.
  Firstly, \(\mu_{\eta}\) is identified with the zero-framed longitude of \(K\).
  Secondly, by van Kampen's theorem, the inclusion \(M_{P} \setminus \mathcal{N}(\eta) \hookrightarrow M_{P}\) induces a surjection $\pi_{1}(M_{P} \setminus \mathcal{N}(\eta)) \twoheadrightarrow \pi_{1}(M_{P})$ with kernel normally generated by~\(\mu_{\eta}\).
\end{remark}

We say that a representation \(\gamma \colon \pi_{1}(M_{P(K,\eta)}) \to GL_d(\F[t^{\pm1}])\) is \(\eta\)-\emph{regular} if \(\gamma(\mu_{\eta}) = \id\) and \(\det(\id - \gamma(\eta)) \neq 0\).
As explained in Remark~\ref{rem:EtaRegular}, if $\gamma$ is $\eta$-regular, then it induces well-defined representations $\gamma_P$ and $\gamma_K$ on $\pi_1(M_P)$ and $\pi_1(M_K)$.

Before proceeding to the next theorem,  it is necessary to recall some terminology from~\cite[Section~4.1]{BCP_Alg}. %
  If \((M,\lambda)\) is a linking form, then, for any submodule \(L \subset M\), 
  the orthogonal complement of \(L\) is defined as \(L^{\perp} = \{x \in M  \mid  \ \lambda(x,y) = 0 \ \forall {y \in L}  \}\). %
  We say that \(L\) is a \emph{sublagrangian} (or \emph{isotropic}) submodule if \(L \subset L^{\perp}\). %
  For a sublagrangian submodule \(L \subset M\),  the \emph{sublagrangian reduction} of \(M \) by \( L \) is defined as \((L^{\perp}/L, \lambda_{L})\), where \(\lambda_{L}(x+L,y+L) := \lambda(x,y)\). %
Note that~$\lambda$ and $\lambda_L$ are Witt equivalent and  therefore by Proposition~\ref{prop:SignatureFunctionAlgebra}
  the averaged signature functions of~\((M,\lambda)\) and its sublagrangian reduction are equal.
The following theorem describes the twisted Blanchfield pairing of satellite knots and proves Theorem~\ref{thm:SatelliteFormulaIntro} from the introduction.
\begin{theorem}
  \label{thm:CablingTheorem}
  Let $K,P$ be knots and let $\eta$ be a simple closed curve in $S^3 \setminus P \subset M_P$.
  Let \(\gamma \colon \pi_{1}(M_{P(K,\eta)}) \to GL_d(\F[t^{\pm1}])\) be a unitary acyclic \(\eta\)-regular representation.
  Let \(T \subset M_{P(K,\eta)}\) denote the common boundary of \(M_{P}\setminus \mathcal{N}(\eta)\) and \(S^{3} \setminus \mathcal{N}(K)\).
  Define
  \[L = \operatorname{im}(H_{1}(T;\LF_{\gamma}^{d}) \to H_{1}(M_{P(K,\eta)};\LF^{d}_{\gamma})).\]
  The following assertions hold:
  \begin{enumerate}
  \item The representations $\gamma_P$ and $\gamma_K$ are unitary and acyclic.
  \item The module \(L\) is isotropic and \(\ord L\) divides \(\det(1-\gamma(\eta))\).
  \item The sublagrangian reduction of \(\Bl_\gamma(P(K,\eta))\) with respect to \(L\) is isometric to the direct sum \(\Bl_{\gamma_{P}}(P) \oplus \Bl_{\gamma_K}(K)\). In particular \(\Bl_\gamma(P(K,\eta))\) is Witt equivalent to $\Bl_{\gamma_{P}}(P) \oplus \Bl_{\gamma_K}(K)$.
  \end{enumerate}
\end{theorem}

The proof relies on the following result from~\cite[Theorem~4.13]{BCP_Alg} whose statement we recall for the readers' convenience.
  \begin{theorem}\label{thm:isoprojection}
    Let $(M',\lambda')$ and $(M'',\lambda'')$ be two non-singular linking forms over $\LF$, let $M$ be a $\LF$-module and let 
    $\iota\colon M\to M'$ be a monomorphism. Write $\lambda$ for the form on~$M$ induced by~$\iota$.
    Suppose that $\pi\colon(M,\lambda)\to(M'',\lambda'')$ is a surjective morphism of linking forms and
    set $L:=\ker(\pi)$ and $L':=\iota(L)$. Then 
    $\ord(M)\ord(M)^\#$ divides $\ord(M')\ord(M'')$. Moreover, if
    \begin{equation}\label{eq:order}\ord(M)\ord(M)^\#\doteq \ord(M'')\ord(M'),
    \end{equation}
    then the following statements hold:
    \begin{enumerate}
    \item $\iota(M)={L'}^{\perp}$;
    \item The linking form $(M'',\lambda'')$ is isometric to the sublagrangian reduction of $M'$ with respect to~$L'$. In particular $(M',\lambda')$ and~$(M'',\lambda'')$ are Witt equivalent.
    \end{enumerate}
  \end{theorem}
\color{black}

\begin{proof}[Proof of Theorem~\ref{thm:CablingTheorem}]
  During this proof, in order to simplify the notation, we will omit coefficients. 
  As we mentioned in the statement, the torus \(T\) can be described either as \(\partial (M_{P} \setminus \mathcal{N}(\eta))$ or as $ \partial (S^{3} \setminus \mathcal{N}(K))\). As a consequence, we have~$M_{P(K,\eta)} = (M_{P} \setminus \mathcal{N}(\eta))\cup_{ T}(S^{3} \setminus \mathcal{N}(K))$.
  Applying~\cite[Theorem~1.1]{FriedlLeidyNagelPowell}, we know that if $H_*(T)$ and~$H_*(M_{P(K,\eta)})$ are $\LF$-torsion modules, then $H_*(M_{P} \setminus \mathcal{N}(\eta))$ and $H_*(S^{3} \setminus \mathcal{N}(K))$ are also
  $\LF$-torsion modules and the direct sum of the canonical inclusions induces the following morphism of linking forms:
  \begin{equation}\label{eq:morphism-inking forms}
    \psi \colon \Bl_{\gamma_P}(M_{P} \setminus \mathcal{N}(\eta)) \oplus \Bl_{\gamma_K}(S^{3} \setminus \mathcal{N}(K)) \to \Bl_\gamma(M_{P(K,\eta)}).
  \end{equation}
  Note that although~\cite[Theorem 1.1]{FriedlLeidyNagelPowell} is stated for Ore domain coefficients, Proposition~\ref{prop:BlanchfieldOreTwisted} implies that it also holds in the twisted setting.
  The next paragraph is devoted to checking that~$H_*(M_{P(K,\eta)})$ and $H_*(T)$ are indeed $\LF$-torsion modules.
  Since the representation $\gamma$ was assumed to be acyclic, $H_*(M_{P(K,\eta)})$ is indeed torsion. We must only show that $H_*(T)$ is torsion.
  We assert that
  \begin{equation}
    \label{eq:HomologyTorus}
    H_{i}(T) =
    \begin{cases}
      \F[t^{\pm1}]^{d} / (1-\gamma(\eta)) & i = 0,1, \\
      0                     & \text{otherwise}.
    \end{cases}
  \end{equation}
  Indeed, arguing as in~\cite[Proposition 3.11]{PowellThesis}, we see that the twisted chain complex \(C_{\ast}(T)\) is chain homotopy equivalent to
$$ \F[t^{\pm1}]^{d} \xrightarrow{\left(1-\gamma(\eta) \quad\gamma(\mu_{\eta})-1\right)^{T}} \F[t^{\pm1}]^{2d} \xrightarrow{\left(1-\gamma(\mu_{\eta}) \quad 1-\gamma(\eta)\right)} \F[t^{\pm1}]^{d}$$
  with a basis of \(C_{1}(T) \cong \F[t^{\pm1}]^{2d} \) given by the \(\mu_{\eta} \otimes e_i \) and \(\eta \otimes e_i\) for $i=1,\ldots,d$.
  Since $\gamma$ is $\eta$-regular, \(\gamma(\mu_{\eta})=1\) and the assertion follows.
  For later use, note that this argument also proves that the~$\F[t^{\pm 1}]$-module \(H_{1}(T)\) is generated by the homology classes~\([e_i \otimes \mu_{\eta}]\) for $i=1,\ldots,d$.
  Since~$\gamma$ is $\eta$-regular,~\eqref{eq:HomologyTorus} shows that $H_*(T)$ is torsion, concluding our verification of the assumptions of~\cite[Theorem 1.1]{FriedlLeidyNagelPowell} and thus establishing the existence of the morphism $\psi$ in~\eqref{eq:morphism-inking forms}.

  Our goal is now to relate the Blanchfield pairings on $M_P,M_K$ and $M_{P(K,\eta)}$.
  Observe that the $\LF$-modules $H_1(M_{P(K,\eta)}), H_1(M_{P} \setminus \mathcal{N}(\eta),T)$ and $H_1(S^{3} \setminus \mathcal{N}(K),T)$ are torsion: for~$H_1(M_{P(K,\eta)})$, this follows from the $\eta$-regularity of $\gamma$.
  For the two latter modules, this can be seen by looking at the long exact sequence of the pairs $(S^{3} \setminus \mathcal{N}(K),T)$ and $(M_{P} \setminus \mathcal{N}(\eta),T)$ together with the fact that~$H_1(M_{P} \setminus \mathcal{N}(\eta)),H_1(S^{3} \setminus \mathcal{N}(K))$ and $H_*(T)$ are all $\LF$-torsion.

  Combining this observation with $ H_{0}(M_{P} \setminus \mathcal{N}(\eta),T) = 0=H_{0}(M_{P} \setminus \mathcal{N}(\eta),T)$, Poincar\'e duality and the universal coefficient theorem (where the underlying ring is the PID $\LF$),
 we obtain the following isomorphisms:
  \begin{align}
    H_{2}(M_{P(K,\eta)}) &\cong H^{1}(M_{P(K,\eta)}) \cong \operatorname{Ext}^{1}_{\text{left-}\LF}(H_{0}(M_{P(K,\eta)}),\LF)^{\#} \cong H_{0}(M_{P(K,\eta)})^{\#}, \label{eq:H-2-1}  \\
    H_{2}(M_{P} \setminus \mathcal{N}(\eta)) &\cong H^{1}(M_{P}\setminus \mathcal{N}(\eta),T) \cong \operatorname{Ext}^{1}_{\LF}(H_{0}(M_{P} \setminus \mathcal{N}(\eta),T),\LF)  = 0,  \label{eq:H-2-2} \\
    H_{2}(S^{3} \setminus \mathcal{N}(K)) &\cong H^{1}(S^{3}\setminus \mathcal{N}(K),T) \cong \operatorname{Ext}^{1}_{\LF}(H_{0}(S^{3}\setminus \mathcal{N}(K),T),\LF) = 0. \label{eq:H-2-3}
  \end{align}
  In more details, in each of the three lines, in our application of the universal coefficient theorem, we used our observation that the relevant first homology modules are $\LF$-torsion.
  In the third isomorphism of~\eqref{eq:H-2-1}, we also used the fact that for any torsion \(\LF\)-module \(T\), we have~$\operatorname{Ext}^{1}_{\LF}(T,\LF) \cong~T$. 

  Our strategy to relate the various Blanchfield pairings is to apply Theorem~\ref{thm:isoprojection}.
  We now introduce the notation necessary to its application as well as verify one of its key assumptions.
  \begin{lemma}
    \label{claim:InclusionIsom}
    The inclusions $\iota_{P} \colon M_{P} \setminus \mathcal{N}(\eta) \hookrightarrow M_{P}$ and $\iota_{K} \colon S^{3} \setminus \mathcal{N}(K) \hookrightarrow M_{K}$ induce  isomorphisms on twisted homology in degree $j \neq 1,2$ and
    epimorphisms $ (\iota_{P})_{\ast},  (\iota_{K})_{\ast}$ in degree one.
    Moreover, \(\ker (\iota_{P})_{\ast}\) and \(\ker(\iota_{K})_{\ast}\) are generated by the~\([e_i \otimes \mu_{\eta}]\) for $i=1,\ldots, d$.
    Additionally, the modules $H_1(M_K)$ and $H_1(M_P)$ are $\LF$-torsion.
    Finally, we have the following equalities
    \begin{align*}
      \ord H_{1}(M_{P} \setminus \mathcal{N}(\eta)) &\doteq \frac{\ord(H_{1}(M_{P})) \cdot \ord(H_{1}(T))}{\ord(H_{0}(M_{P}))^\#}, \\
      \ord H_{1}(S^{3} \setminus \mathcal{N}(K))    &\doteq \frac{\ord(H_{1}(M_{K})) \cdot \ord(H_{1}(T))}{\ord(H_{0}(M_{K}))^\# }.
    \end{align*}
  \end{lemma}
  \begin{proof}
    We only prove the lemma for \(M_{P} \setminus \mathcal{N}(\eta)\): the proof for \(S^{3} \setminus \mathcal{N}(K)\) is completely analogous.
    Using~\eqref{eq:HomologyTorus} and the Mayer-Vietoris sequence of \(M_{P} = (M_{P} \setminus \mathcal{N}(\eta)) \cup_T \mathcal{N}(\eta)\), we deduce that~$i_P$ does indeed induce an isomorphism on twisted homology in degree $j > 2$.
    In degree zero, notice that the inclusion \(T \hookrightarrow \mathcal{N}(\eta)\) induces an isomorphism on \(H_{0}\) and therefore, with the aid of the Mayer-Vietoris sequence, we obtain an isomorphism \(H_{0}(M_{P} \setminus \mathcal{N}(\eta)) \cong H_{0}(M_{P})\).
    
    To get the result in degrees $j=1,2$, consider the relevant part of the Mayer-Vietoris sequence:
    \begin{equation}\label{eq:MV-sequence}
      0 \xrightarrow{} H_{2}(M_{P}) \xrightarrow{} H_1(T) \xrightarrow{j_{P}} H_1(M_{P} \setminus \mathcal{N}(\eta)) \xrightarrow{(\iota_{P})_{\ast}} H_1(M_P) \to 0.
    \end{equation}
    Notice that the leftmost part of~\eqref{eq:MV-sequence} is exact by~\eqref{eq:H-2-2}.
    Since $H_1(M_P \setminus \mathcal{N}(\eta))$ is $\LF$-torsion,~\eqref{eq:MV-sequence} implies that this is also the case for $H_1(M_P)$.
    Next,~\eqref{eq:MV-sequence} also implies that the map~\((\iota_{P})_{\ast}\) is an epimorphism with kernel generated by the~\([e_i \otimes \mu_{\eta}]\) (or technically by~\( j_P([e_i \otimes \mu_{\eta}])\), because~\(H_{1}(T)\) is generated by these elements.
    The last assertion of the lemma also follows from~\eqref{eq:MV-sequence}: we use the multiplicativity of the orders and the same isomorphisms as in~\eqref{eq:H-2-1} to establish the isomorphism $H_{2}(M_{P}) \cong H_0(M_{P})^{\#}$ (this is possible since we now know that $H_1(M_P)$ is torsion).
    This concludes the proof of the lemma.
  \end{proof}

  We now prove the first assertion of the theorem, namely that $\gamma_P$ and $\gamma_K$ are acyclic. We only check this for $\gamma_P$: the proof for $\gamma_K$ is identical. Since $H_*(M_P \setminus \mathcal{N}(\eta))$ is $\LF$-torsion, Lemma~\ref{claim:InclusionIsom} implies that $H_i(M_P)$ is $\LF$-torsion for $i \neq 2$. For $i=2$, we use the isomorphism $H_2(M_P)\cong H_0(M_P)^\#$ (that we established in the proof of Lemma~\ref{claim:InclusionIsom}) and the conclusion follows.
The first statement of Theorem~\ref{thm:CablingTheorem} is now proved.

  As a next step, consider the decomposition \(M_{P(K,\eta)} = \left(M_{P}\setminus \mathcal{N}(\eta)\right) \cup_T \left(S^{3} \setminus \mathcal{N}(K)\right)\).
  The associated Mayer-Vietoris sequence can be decomposed into the following three exact sequences:
  \begin{align}
    &0 \xrightarrow{} H_{2}(M_{P(K,\eta)}) \xrightarrow{} H_{1}(T) \xrightarrow{} \mathcal{K} \xrightarrow{} 0, \label{eq:ex-seq-1}\\
    &0 \xrightarrow{} \mathcal{K} \xrightarrow{} H_{1}(M_{P}\setminus \mathcal{N}(\eta)) \oplus H_{1}(S^{3}\setminus \mathcal{N}(K)) \xrightarrow{\psi} H_{1}(M_{P(K,\eta)}) \xrightarrow{} \mathcal{C} \xrightarrow{} 0, \label{eq:ex-seq-2}\\
    &0 \xrightarrow{} \mathcal{C} \xrightarrow{} H_{0}(T) \xrightarrow{} H_{0}(M_{P}\setminus \mathcal{N}(\eta)) \oplus H_{0}(S^{3}\setminus \mathcal{N}(K)) \xrightarrow{} H_{0}(M_{P(K,\eta)}) \xrightarrow{} 0 \label{eq:ex-seq-3},
  \end{align}
  where $\mathcal{K}=\ker(\psi)$ and $\mathcal{C}=\coker(\psi)$.
  Therefore, applying the multiplicativity of the orders, the use of~\eqref{eq:ex-seq-1} and~\eqref{eq:ex-seq-3} (as well as~\eqref{eq:H-2-1}) implies that
  \begin{align*}
    \ord (\mathcal{C}) &\doteq \frac{\ord(H_{0}(T)) \cdot \ord(H_{0}(M_{P(K,\eta)}))}{\ord(H_{0}(M_{P}\setminus \mathcal{N}(\eta))) \cdot \ord(H_{0}(S^{3} \setminus \mathcal{N}(K)))} \doteq \frac{\ord(H_{0}(T)) \cdot \ord(H_{0}(M_{P(K,\eta)}))}{\ord(H_{0}(M_{P})) \cdot \ord(H_{0}(M_{K}))}, \\
    \ord (\mathcal{K} )&\doteq \frac{\ord(H_{1}(T))}{\ord(H_{2}(M_{P(K,\eta)}))} \doteq \frac{\ord(H_{1}(T))}{\ord(H_{0}(M_{P(K,\eta)}))^\#}.
  \end{align*}
  The key $\LF$-module required in this proof is defined as
  \[A := \frac{H_{1}(M_{P}\setminus \mathcal{N}(\eta)) \oplus H_{1}(S^{3}\setminus \mathcal{N}(K))}{\mathcal{K}}.\]
  The module $A$ will be used both in the proof of the second and third statements of Theorem~\ref{thm:CablingTheorem}.
As a preliminary however, we first compute $    \ord (A)  \ord (A)^\# $.
  Combining~\eqref{eq:ex-seq-2} with our computation of $\ord (\mathcal{C})$, we obtain
  \[\ord (A) \doteq \frac{\ord(H_{1}(M_{P(K,\eta)}))}{\ord (\mathcal{C})} \doteq \frac{\ord(H_{1}(M_{P(K,\eta)})) \cdot \ord(H_{0}(M_{P})) \cdot \ord H_{0}(M_{K})}{\ord (H_{0}(T)) \cdot \ord(H_{0}(M_{P(K,\eta)}))}.\]
  On the other hand, Lemma~\ref{claim:InclusionIsom} and our computation of $\ord (\mathcal{K})$ imply that
  \begin{align*}
    \ord (A) &\doteq \frac{\ord(H_{1}(M_{P} \setminus \mathcal{N}(\eta))) \cdot \ord(H_{1}(S^{3}\setminus \mathcal{N}(K)))}{\ord (\mathcal{K})}  \\
             &\doteq \frac{\ord(H_{1}(M_{P})) \cdot \ord(H_{1}(M_{K})) \cdot \ord(H_{1}(T))^{2} \cdot \ord(H_{0}(M_{P(K,\eta)}))^\# }{\ord(H_{0}(M_{P}))^\#  \cdot \ord(H_{0}(M_{K}))^\#  \cdot \ord(H_{1}(T))}  \\
             &\doteq \frac{\ord(H_{1}(M_{P})) \cdot \ord(H_{1}(M_{K})) \cdot \ord(H_{1}(T)) \cdot \ord(H_{0}(M_{P(K,\eta)}))^\# }{\ord(H_{0}(M_{P}))^\#  \cdot \ord(H_{0}(M_{K}))^\# }.
  \end{align*}
  Additionally, observe that $\ord(H_1(T))=\ord(H_0(T))^\#$ (this follows by duality and the universal coefficient theorem) and $\ord(H_1(M_{P(K,\eta)})=\ord(H_1(M_{P(K,\eta)})^\#$ (since $H_1(M_{P(K,\eta)})$ supports a non-singular linking form).
  Combining these observations with the above formulas for $\ord (A)$, we obtain
  \begin{equation}
    \label{eq:ord-A}
    \ord (A)  \ord (A)^\# \doteq \ord(H_{1}(M_{P})) \cdot \ord(H_{1}(M_{K})) \cdot \ord(H_{1}(M_{P(K,\eta)})).
  \end{equation}
We now prove the second statement of Theorem~\ref{thm:CablingTheorem} which asserts that $L$ is isotropic and that \(\ord (L)\) divides \(\det(1-\gamma(\eta))\).
  Observe that \((\iota_{P})_{\ast} \oplus (\iota_{K})_{\ast}\) and \(\psi\) vanish on $\mathcal{K}$.
  As a consequence, we let \((\iota_{P})_{\ast} \oplus (\iota_{K})_{\ast} / \mathcal{K}\) and~\(\psi / \mathcal{K}\) denote the maps induced on \(A\).
  Additionally, using \(\widetilde{L}\) to denote the kernel of \((\iota_{P})_{\ast} \oplus (\iota_{K})_{\ast} / \mathcal{K}\) and~\(\operatorname{Cok}\) to denote the cokernel of \(\psi / \mathcal{K}\), we obtain the following diagram of exact sequences:

  \begin{equation}
    \label{eq:ApplyAlgebra}
    \xymatrix@R0.5cm{
      \widetilde{L} \ar@{^{(}->}[d] \\
      A \ar@{^{(}->}[r]^{\psi/\mathcal{K}} \ar@{->>}[d]_{(\iota_{P})_{\ast} \oplus (\iota_{K})_{\ast} / \mathcal{K}} & H_1(M_{P(K,\eta)}) \ar[r] & \operatorname{Cok}. \\
      H_1(M_P) \oplus H_1(M_K)
    }
  \end{equation}
  To check the second assertion of the theorem, we will show that $(\psi/\mathcal{K})(\widetilde{L}) = L$, where $L$ was defined as $\im (H_1(T) \to H_1(M_{P(K,\eta)}))$.
  Use $j_P$ (resp. $j_K$) to denote the map induced by the inclusion $T \to M_P \setminus \mathcal{N}(\eta)$ (resp. $T \to S^3 \setminus \mathcal{N}(K)$).
  The exact sequence in~\eqref{eq:ex-seq-2} implies that~\(\mathcal{K}\), treated as a submodule of \(H_{1}(M_{P}\setminus \mathcal{N}(\eta)) \oplus H_{1}(S^{3}\setminus \mathcal{N}(K))\), is generated by \((j_{P}([e_{i} \otimes \mu_{\eta}]), -j_{K}([e_{i}\otimes\mu_{\eta}]))\), for~$i=1,\ldots, d$.
  Therefore, for any $i$, we have \((j_{P}([e_{i} \otimes \mu_{\eta}]),0) = (0,j_{K}([e_{i} \otimes \mu_{\eta}]))\) in \(A\).
  Moreover,
  \[\widetilde{L} =\ker \left((i_P)_{\ast} \oplus (i_K)_{\ast}\right) = \langle (j_{P}[e_{i}\otimes \mu_{\eta}], j_{K}[e_{i}\otimes \mu_{\eta}]) \mid 1 \leq i \leq d\rangle / \langle j_{P}[e_{i} \otimes \mu_{\eta}] = j_{K}[e_{i} \otimes \mu_{\eta}] \rangle,\]
  and thus \((\psi/\mathcal{K})(\widetilde{L}) = L\). Since \(\psi/\mathcal{K}\) is injective, it follows that \(L\) and \(\widetilde{L}\) are isomorphic.
  In order to prove that \(L\) is isotropic, observe that the homology classes \([e_{i}\otimes \mu_{\eta}]\) belong to the radical of~\(\Bl_{\gamma_{P}}(M_{P} \setminus \mathcal{N}(\eta))\) and \(\Bl_{\gamma_{K}}(S^{3} \setminus \mathcal{N}(K))\), because they come from the boundary of the respective manifolds.
  Since \(\psi\) and \(\psi / \mathcal{K}\) are morphisms of linking forms, it follows that \(L\) is isotropic.
  The constraint on the order of \(L\) comes from the fact that \(L\) is a quotient of \(H_{1}(T)\).
This establishes the second statement of Theorem~\ref{thm:CablingTheorem}.

We now move on to the third statement which describes the relation between $\Bl_\gamma(P(K,\eta))$ and $\Bl_{\gamma_P}(P) \oplus \Bl_{\gamma_K}(K)$.
  First observe that $\Bl_{\gamma_P}(M_P \setminus \mathcal{N}(\eta)) \oplus \Bl_{\gamma_K}(S^3 \setminus \mathcal{N}(K)) $ vanishes on $\mathcal{K}$: indeed~$\mathcal{K}$ only contains homology classes which come from the boundary torus $T$.
  As a consequence, $\Bl_{\gamma_P}(M_P \setminus \mathcal{N}(\eta)) \oplus \Bl_{\gamma_K}(S^3 \setminus \mathcal{N}(K))$ descends to $A$.


  We wish to apply Theorem~\ref{thm:isoprojection} to $(M'',\lambda'')=\Bl_\gamma(P(K,\eta)),(M',\lambda')=\Bl_{\gamma_P}(P) \oplus \Bl_{\gamma_K}(K)$ with~$M=A,\iota=\psi/\mathcal{K}$ and $\pi=(\iota_{P})_{\ast} \oplus (\iota_{K})_{\ast} / \mathcal{K}$.
  First, it is clear that the maps \((\iota_{P})_{\ast} \oplus (\iota_{K})_{\ast} / \mathcal{K}\) and \(\psi / \mathcal{K}\) are morphisms of linking forms.
  Next, by definition of $\mathcal{K}$, the map $\psi / \mathcal{K}$ is injective, while \((\iota_{P})_{\ast} \oplus (\iota_{K})_{\ast} / \mathcal{K}\) is surjective thanks to~\eqref{eq:MV-sequence}.
  Since $M_P, M_K$ and $M_{P(K,\eta)}$ are closed manifolds, the corresponding Blanchfield forms are non-singular.
  Finally, we already showed in~\eqref{eq:ord-A} that $\ord (A) \cdot (\ord (A))^\# \doteq \ord(H_{1}(M_{P})) \cdot \ord(H_{1}(M_{K})) \cdot \ord(H_{1}(M_{P(K,\eta)}))$.
  We can therefore apply Theorem~\ref{thm:isoprojection}  to the above diagram to obtain the third assertion of the theorem, concluding the proof.
\end{proof}

The next result shows that Blanchfield pairings need not be additive under satellite operations.

\begin{corollary}\label{cor:cabling-formula}
  Using the same assumptions as in Theorem~\ref{thm:CablingTheorem}, if \(\ord (L) \) is relatively prime to \(\ord(H_{1}(M_{P};\LF^{d}_{\gamma_{P}})) \cdot \ord(H_{1}(M_{K};\LF^{d}_{\gamma_{K}}))\), then there is an isometry of linking forms
  \[\Bl_\gamma(P(K,\eta)) \cong \Bl_{\gamma_P}(P) \oplus \Bl_{\gamma_K}(K) \oplus (X,\lambda_{X}),\]
  where \((X,\lambda_{X})\) is a metabolic linking form and \(L \subset X\) is a metabolizer for \(\lambda_{X}\).

  In particular, if the map \(H_{1}(T;\LF^{d}_{\gamma}) \to H_{1}(M_{P(K,\eta)};\LF^{d}_{\gamma})\) vanishes, then we obtain the following isometry of linking forms:
  \[\Bl_{\gamma}(P(K,\eta)) \cong \Bl_{\gamma_P}(P) \oplus \Bl_{\gamma_K}(K).\]
\end{corollary}
\begin{proof}
  We first recall some notations from the proof of Theorem~\ref{thm:CablingTheorem}. Recall that we set $\mathcal{K}=~\ker(\psi)$, where $\psi \colon \Bl_{\gamma_P}(M_{P} \setminus \mathcal{N}(\eta)) \oplus \Bl_{\gamma_K}(S^{3} \setminus \mathcal{N}(K)) \to \Bl_\gamma(M_{P(K,\eta)})$ is the morphism of linking forms described in~\eqref{eq:morphism-inking forms}. We once again omit coefficients and recall that $A:=(H_1(M_P \setminus \mathcal{N}(\eta) )\oplus H_1(S^3 \setminus \mathcal{N}(K)))/\mathcal{K}$. Since we saw that $\Bl_{\gamma_P}(P) \oplus \Bl_{\gamma_K}(K)$ descends to $A$; we let $\Bl_A$ denote the resulting pairing. 

  \begin{claim*}
    There is a linking form $(X,\lambda_X)$ such that $ \Bl_\gamma(P(K,\eta)) \cong \Bl_{\gamma_{P}}(P) \oplus \Bl_{\gamma_{K}}(K) \oplus (X,\lambda_X)$.
  \end{claim*}
  \begin{proof}
    Using the multiplicativity of orders on the vertical exact sequence in~\eqref{eq:ApplyAlgebra}, we see that $\ord (A) = \ord(\widetilde{L}) \cdot \ord(H_{1}(M_{P})) \cdot \ord(H_{1}(M_{K})).$
    Using the primary decomposition theorem (since~$\LF$ is a PID), the Chinese remainder theorem and \(\gcd(\ord (\widetilde{L}), \ord(H_{1}(M_{P})) \cdot \ord(H_{1}(M_{K}))) = 1\), we obtain an isomorphism
    $A \cong \widetilde{L} \oplus B,$
    where $B$ is isomorphic to $H_{1}(M_{P}) \oplus H_{1}(M_{K})$.
    Observe that the surjectivity of \((\iota_{P})_{\ast} \oplus (\iota_{K})_{\ast}/\mathcal{K}\) implies that 
    \begin{equation}
      \label{eq:BNonSingular} 
      \Bl_{A}|_{B} \cong \Bl_{\gamma_{P}}(P) \oplus \Bl_{\gamma_{K}}(K).
    \end{equation}
Let \(B' = (\psi/\mathcal{K})(B) \subset H_{1}(M_{P(K,\eta)})\).
      Since \(\psi/\mathcal{K}\) is injective (recall~\eqref{eq:ApplyAlgebra}), \(B \cong B'\).
      By our assumptions, \(\ord(L)\) is relatively prime with \(\ord(B')\).
      Furthermore, since \(\ord(B')\) is symmetric, it follows that \(\gcd(\ord(L)^{\#},\ord(B')) = 1\).
      Consequently,
      \[\gcd(\ord (H_{1}(M_{P(K,\eta)})/B'), \ord (B')) = \gcd(\ord(L) \cdot \ord(L)^{\#}, \ord(B')) = 1,\]
      where the first equality follows from the fact that \(\ord(H_{1}(M_{P(K,\eta)}) / B') = \ord(L) \cdot \ord(L)^{\#}\).
      Hence there exists a submodule \(X \subset H_{1}(M_{P(K,\eta)})\), such that \(H_{1}(M_{P(K,\eta)}) = X \oplus B'\).
      By~\cite[Proposition 2.25]{BCP_Alg}, there is a corresponding splitting of linking forms \(\Bl_{\gamma}(P(K,\eta)) = \Bl_{\gamma}(P(K,\eta))|_{X} \oplus \Bl_{\gamma}(P(K,\eta))|_{B'}\).
       Since, $\psi/\mathcal{K} \colon A \to H_1(M_{P(K,\eta)})$ is an injective map of linking forms, we obtain identifications
       \[\Bl_{\gamma}(P(K,\eta))|_{B'} \cong \Bl_{A}|_{B} \cong \Bl_{\gamma_{P}}(P) \oplus \Bl_{\gamma_{K}}(K),\]
       where the first isomorphism is given by \((\psi / \mathcal{K})|_{B}\) and the last isomorphism follows from equation~\eqref{eq:BNonSingular}.
      Hence, we can take \(\lambda_{X} = \Bl(P(K,\eta))|_{X}\), which is non-singular by~\cite[Proposition 2.25]{BCP_Alg}.
    This concludes the proof of the claim.
  \end{proof}
  Since Theorem~\ref{thm:CablingTheorem} shows that \(\Bl_\gamma(M_{P(K,\eta)})\) is Witt equivalent to $\Bl_{\gamma_P}(P) \oplus~\Bl_{\gamma_K}(K)$, the claim implies that  
  $(X,\lambda_X)$ is metabolic. It remains to show that a metabolizer for $(X,\lambda_X)$ is in fact given by $L:= (\psi/\mathcal{K})(\widetilde{L})$.
  As $(\psi/\mathcal{K}) \colon A \to H_1(M_{P(K,\eta)}) $ is an injective morphism of linking forms, the isomorphisms $A \cong \widetilde{L} \oplus B$ and $H_1(M_{P(K,\eta)})~\cong \psi/\mathcal{K}(B) \oplus~X$ imply that \(L = (\psi/\mathcal{K})(\widetilde{L})\) is an isotropic submodule of $X$.
  Since Theorem~\ref{thm:CablingTheorem} shows that $\Bl_{\gamma_P}(P) \oplus \Bl_{\gamma_K}(K)$ is isometric to the isotropic reduction of~\(\Bl_\gamma(M_{P(K,\eta)})\) with respect to $L$, the multiplicativity of orders implies that~$\ord (X)=\ord(L)\ord(L)^{\#}$. As $\LF$ is a PID, we deduce that~\(L\) is a metabolizer in~\(X\), concluding the proof of the first assertion.

  In order to prove the second assertion, notice that our hypothesis implies that the submodule~$L$ is trivial. Since $L$ is a metabolizer of $X$, we deduce that $X$ is itself trivial. The conclusion now follows from the first assertion.
\end{proof}

The situation simplifies considerably when \(\gamma_{K}\) is abelian, i.e.\ factors through the abelianization of \(\pi_{1}(M_{K})\).
Let \(\F[t_{K}^{\pm1}]\) be the group algebra of the abelianization of \(\pi_{1}(M_{K})\) with \(t_{K}\) denoting the generator corresponding to the meridian identified with \(\eta\).
In this case, \(\LF^{d}_{\gamma_{K}}\) becomes a module over \(\F[t_{K}^{\pm1}]\) via the homomorphism \(\gamma_{K} \colon \F[t_K^{\pm 1}] \to \LF^d_{\gamma_{K}} \), i.e.\ via the action $t_{K} \cdot v = v\gamma_K(\eta) $ with~$v \in \LF^{d}_{\gamma_{K}}$.

In this setting, Theorem~\ref{thm:CablingTheorem} takes the following form (compare with~\cite[Theorem 1.3]{FriedlLeidyNagelPowell}).

\begin{corollary}\label{cor:cabling-formula-abelian}
  Using the same assumptions as in Theorem~\ref{thm:CablingTheorem}, if \(\gamma_{K}\) is an abelian representation, then there exists an isometry of linking forms
  \[\Bl_{\gamma}(P(K,\eta)) \cong \Bl_{\gamma_P}(P) \oplus \Bl_{\gamma_K}(K).\]
  If we denote by \(\F[\eta] \subset \F[\pi(M_{P})]\) the subring generated by \(\eta\)
  and we assume that \(\det(\gamma_{P}(x)) \neq 0\), for any \(x \in \F[\eta] \setminus \{0\}\), then $H_{1}(M_{K};\LF^{d}_{\gamma})$ is isomorphic to $H_{1}(M_{K};\F[t_{K}^{\pm1}]) \otimes_{\F[t_{K}^{\pm1}]} \LF^{d}_{\gamma_{K}},$ and there is an isometry
  $$\Bl_{\gamma_K}(K) \cong \Bl(K) \otimes_{\F[t_{K}^{\pm1}]} \LF^{d}_{\gamma_{K}}.$$
\end{corollary}
\begin{proof}
  Let \(\lambda\) denote the zero-framed longitude of \(K\).
  If \(F \subset S^{3} \setminus \mathcal{N}(K)\) is a Seifert surface for~\(K\), then it follows that the image of \(\pi_{1}(F)\) is contained in the commutator subgroup of \(\pi_{1}(S^{3} \setminus \mathcal{N}(K))\), therefore the inclusion of \(F\) lifts to a map to the infinite cyclic cover of \(S^{3} \setminus \mathcal{N}(K)\).
  This means that for any \(\F[t_{K}^{\pm1}]\)-module \(M\), the homology class of \(\lambda\) is trivial in \(H_{1}(S^{3} \setminus \mathcal{N}(K);M)\) and~\(H_{1}(M_{K};M)\).
  Thus, the hypothesis of Corollary~\ref{cor:cabling-formula} is satisfied   and we therefore obtain the desired isometry of linking forms.

  If we have \(\det(\gamma(x))\neq 0\) for any \(x \in \F[\eta] \setminus \{0\}\), then \(\LF^{d}_{\gamma_{K}}\) is a torsion-free \(\F[t_{K}^{\pm1}]\)-module and hence (since~$\LF$ is a PID) flat~\(\F[t_{K}^{\pm1}]\)-module.
  The last two assertions follow easily from the flatness of~\(\LF^{d}_{\gamma_{K}}\).
\end{proof}

Next, we discuss the consequences of Theorem~\ref{thm:CablingTheorem} on signatures of satellite knots.
\begin{corollary}
  \label{cor:SignatureSatellite}
  Let $P(K,\eta)$ be a satellite knot with pattern $P$, companion $K$ and infection curve~$\eta$.
  Let $\gamma \colon \pi_1(M_{P(K,\eta)}) \to GL_d(\LF)$ be a unitary acyclic $\eta$-regular representation.
  If the assumptions of Theorem~\ref{thm:CablingTheorem} hold, then for any $\omega \in S^1$, the averaged twisted signatures of~$P(K,\eta),P$ and $K$ satisfy
  $$ \sigma_{P(K,\eta),\gamma}^{\operatorname{av}}(\omega)=\sigma_{P,\gamma_P}^{\operatorname{av}}(\omega)+\sigma_{K,\gamma_K}^{\operatorname{av}}(\omega).$$
  In particular, under these assumptions, the averaged twisted signature is additive on connected~sums:
  $$ \sigma_{P \# K,\gamma}^{\operatorname{av}}(\omega)=\sigma_{P,\gamma_P}^{\operatorname{av}}(\omega)+\sigma_{K,\gamma_K}^{\operatorname{av}}(\omega).$$
\end{corollary}
\begin{proof}
  Applying Theorem~\ref{thm:CablingTheorem},
  \(\Bl_\gamma(P(K,\eta))\)  is Witt equivalent to \(\Bl_{\gamma_{P}}(P) \oplus \Bl_{\gamma_K}(K)\).
  The first assertion now follows from the Witt invariance of the averaged signature, recall item~\ref{item:S4} of Proposition~\ref{prop:SignatureFunctionAlgebra}; compare also
 ~\cite[Proposition~5.7]{BCP_Alg}.
  The second assertion is immediate since the connected sum is obtained by taking the infection curve $\eta$ to be the meridian of $P$.
\end{proof}

Finally, in the abelian case, we recover results concerning the behavior of the Blanchfield pairing and Levine-Tristram signature under satellite operations (see~\cite{LivingstonMelvin} and~\cite{Shinohara, LitherlandIterated}).

\begin{corollary}
  \label{cor:SignatureSatellite_abelian}
  Let $P(K,\eta)$ be a satellite knot with pattern $P$, companion $K$, infection curve~$\eta$ and winding number \(w = \lk(\eta,P)\).
  The Blanchfield pairing and Levine-Tristram signature sastify
  \begin{align*}
    \Bl(P(K,\eta))(t)&=\Bl(P)(t) \oplus  \Bl(K)(t^w), \\
    \sigma_{P(K,\eta)}(\omega)&=\sigma_{P}(\omega)+\sigma_{K}(\omega^w).
  \end{align*}
\end{corollary}
\begin{proof}
  The representation $\gamma$ is abelianization and the composition $ H_1(M_K;\Z) \to H_1(M_{P(K,\eta)};\Z) \cong~\Z$ is given by multiplication by $w$. 
  The result now follows from Remark~\ref{rem:LevineTristram} and Corollary~\ref{cor:cabling-formula-abelian}.
\end{proof}

\section{Metabelian Representations and Casson-Gordon invariants}
\label{sec:CassonGordon}

The aim of this section is to study signatures associated to metabelian representations.
In Subsection~\ref{sub:Metabelian}, we recall two equivalent constructions of a ``metabelian Blanchfield pairing", in Subsection~\ref{sub:MillerPowell}, we review a result due to Miller and Powell, in Subsection~\ref{sub:CassonGordon}, we recall the definition of the Casson-Gordon $\tau$-invariant and in Subsection~\ref{sec:Proof}, we relate it to our twisted signatures.
Throughout this section, we use $\Sigma_n(K)$ to denote the $n$-fold cover of $S^3$ branched along a knot $K$ and fix a character $\chi \colon H_1(\Sigma_n(K);\Z) \to \Z_m$. We also set \(\xi_{m} = e^{\frac{2 \pi i}{m}}\) and~$M_K$ will always denote the 0-framed surgery along $K$.

\subsection{The metabelian Blanchfield pairing}
\label{sub:Metabelian}

In this subsection, we review two equivalent definitions of a particular twisted Blanchfield pairing associated to $M_K$.
The first makes use of a metabelian representation $\alpha(n,\chi) \colon \pi_1(M_K) \to GL_n(\LC)$, while the second uses the $n$-fold cover $M_n \to M_K$.
References for this subsection include~\cite{HeraldKirkLivingston, FriedlEta, MillerPowell}.
\medbreak

Use $\phi_K \colon \pi_1(M_K) \to  H_1(M_K;\Z)\cong \Z=\langle t_K \rangle$  to denote the abelianization homomorphism.
Since~$\phi_K$ endows $\Z[t_K^{\pm 1}]$ with a right $\pi_1(M_K)$-module structure, it gives rise to the twisted homology $\Z[t_K^{\pm1}]$-module $H_1(M_K;\Z[t_{K}^{\pm1}])$.
This module is known to coincide with the Alexander module of $K$. In what follows, we shall frequently identify $H_1(\Sigma_n(K);\Z)$ with $H_1(M_K;\Z[t_K^{\pm 1}])/(t_K^n-1)$ as in~\cite[Corollary 2.4]{FriedlPowell}. We now consider the semidirect product $\Z \ltimes H_1(\Sigma_n(K);\Z)$, where the group law is given by $ (t_K^i,v)\cdot (t_K^j,w)=(t_K^{i+j},t_K^{-j}v+w)$. Next, consider the representation  
\begin{align}
  \label{eq:Matrix}
  \gamma_K(n,\chi) \colon \Z \ltimes H_1(\Sigma_n(K);\Z) &\to ~\operatorname{GL}_n(\LC)  \nonumber \\
  (t_K^j,v) &\mapsto \begin{pmatrix}
    0& 1 & \cdots &0 \\
    \vdots & \vdots & \ddots & \vdots  \\
    0 & 0 & \cdots & 1 \\
    t & 0 & \cdots & 0
  \end{pmatrix}^j
                \begin{pmatrix}
                  \xi_{m}^{\chi(v)} & 0 & \cdots &0 \\
                  0 & \xi_{m}^{\chi(t_K \cdot v)} & \cdots &0 \\
                  \vdots & \vdots & \ddots & \vdots \\
                  0 & 0 & \cdots & \xi_{m}^{\chi(t_K^{n-1} \cdot v)}
                \end{pmatrix}.
\end{align}
We will now use $\gamma_K(n,\chi)$ to obtain a representation of $\pi_1(M_K)$. To achieve this, we identify the Alexander module $H_1(M_K;\Z[t_{K}^{\pm1}])$ with the derived quotient $\pi_1(M_K)^{(1)}/\pi_1(M_K)^{(2)}$ and consider the following composition of canonical projections
\[q_K \colon \pi_{1}(M_{K})^{(1)}  \to  H_1(M_K;\Z[t_K^{\pm 1}]) \to H_1(\Sigma_n(K);\Z).\]
Fix an element~$\mu_{K}$ in $\pi_1(M_K)$ such that $\phi_K(\mu_{K})=t_K$.
Note that for every $g \in \pi_1(M_K)$, we have~$\phi_K(\mu_K^{-\phi_K(g)}g)=1$.
Since $\phi_K$ is the abelianization map, we deduce that $\mu_K^{-\phi_K(g)}g$ belongs to~$\pi_1(M_K)^{(1)}$.
As a consequence, we obtain the following map:
\begin{align}
  \label{eq:MapToSemiDirect}
  \widetilde{\rho}_K \colon  \pi_1(M_K) &\to \Z \ltimes H_1(\Sigma_n(K);\Z) \\
  g &\mapsto (\phi_K(g),q_K(\mu_{K}^{-\phi_K(g)}g)). \nonumber
\end{align}
Precomposing the representation $\gamma_K(n,\chi)$ with $\widetilde{\rho}_K$ provides the unitary representation 
\begin{equation}
  \label{eq:ActionForMK}
  \alpha_K(n,\chi) \colon \pi_1(M_K) \stackrel{\widetilde{\rho}}{\to}  \Z \ltimes H_1(\Sigma_n(K);\Z) \stackrel{\gamma_K(n,\chi)}{\longrightarrow} \operatorname{GL}_n(\LC).
\end{equation}
If $m$ is a prime power and $\chi \colon H_1(\Sigma_n(K);\Z) \to \Z_m$ is nontrivial, it is known that the unitary representation $\alpha(n,\chi)$ is acyclic~\cite{FriedlPowellInjectivity} (see also~\cite[Lemma 6.6]{MillerPowell} and~\cite[Corollary following Lemma~4]{CassonGordon2}). When the knot $K$ is clear from the context, we drop it from the notation and simply write~$\alpha(n,\chi)$ instead of $\alpha_K(n,\chi)$.

The next lemma illustrates these notations and will be used to ensure that the hypotheses of Theorem~\ref{thm:CablingTheorem} hold.

\begin{lemma}
  \label{lem:H0CG}
  If $\chi \colon H_1(\Sigma_n(K);\Z) \to \Z_m$ is nontrivial, then $H_0(M_K;\LC_{\alpha(n,\chi)}^n)=0$.
\end{lemma}
\begin{proof}
  Let \(e_{i} \in \C[t^{\pm1}]^{n}\), for \(i=1,2,\ldots,n\), denote the
  vectors from the standard basis.
  By definition~$H_0(M_K;\LC_{\alpha(n,\chi)}^n)=\LC^n/V$, with 
  \[V=\langle (\alpha(n,\chi)(g)-\id)v \ | \ v \in \LC^n, g \in ~\pi_1(M_K) \rangle.\] 
  Therefore, it is sufficient to prove that for any \(1 \leq i \leq n\),
  there exists \(g_{i} \in \pi_{1}(M_{K})\) and~$w_{i} \in \C[t^{\pm1}]^{n}\) so that
  ~$e_{i} = \left(\alpha(n,\chi)(g_{i}) - \id\right) \cdot w_{i}\).
  Fix~$1 \leq i \leq n\).
  Choose any~$h_{i} \in [\pi_{1}(M_{K}),\pi_{1}(M_{K})]\) so that~$\chi(\widetilde{\rho}_K(h_{i})) \neq 1\); such an $h_i$ exists because $\chi$ is non-trivial and $\widetilde{\rho}_K$ is surjective.
  The definition of $\alpha(n,\chi)$ then gives
  \[\left(\alpha(n,\chi)(\mu_K^{-i+1} \cdot h_{i} \cdot \mu_K^{i-1}) - \id \right) \cdot e_{i} =
    \left(\chi(\widetilde{\rho}_K(h_{i})) - 1\right) e_{i}.\]
  Therefore, for~$g_{i} := \mu_K^{-i+1} \cdot h_{i} \cdot \mu_K^{i-1}\) and~$w_{i} := \frac{1}{\chi(\widetilde{\rho}_K(h_{i}))-1} e_{i}\), we have~$e_{i} = \left(\alpha(n,\chi)(g_{i}) - \id\right) \cdot w_{i}\).
\end{proof}

We now use $\alpha(n,\chi)$ to endow $\C[t^{\pm 1}]^n$ with a right $\Z[\pi_1(M_K)]$-module structure and apply Definition~\ref{def:Blanchfield} to define a twisted Blanchfield pairing on $M_K$.

\begin{definition}
  \label{def:BlanchfieldMetabelian}
  Let $K$ be an oriented knot, $n$ be an integer, $m$ be a prime power and choose a nontrivial character $\chi \colon H_1(\Sigma_n(K);\Z) \to~\Z_m$.
  The \emph{metabelian Blanchfield pairing} is the Blanchfield pairing twisted by the unitary acyclic representation $\alpha(n,\chi)$:
  $$\Bl_{\alpha(n,\chi)}(K) \colon H_1(M_K;\LC^n_{\alpha(n,\chi)}) \times H_1(M_K;\LC^n_{\alpha(n,\chi)}) \to \OC/ \LC.$$
\end{definition}

Next, we provide a description of $\Bl_{\alpha(n,\chi)}(K)$ in terms of the $n$-fold cyclic cover $p \colon M_n \to M_K$.  The first step is to consider the abelian subgroup $n\Z \times H_1(\Sigma_n(K);\Z)$ of $\Z \ltimes H_1(\Sigma_n(K);\Z)$ and the following $1$-dimensional representation: 
\begin{align}
  \label{eq:Reprho(n,chi)}
  \rho_K(n,\chi) \colon n \Z \times H_1(\Sigma_n(K);\Z) &\to \LC \\
  (t_K^{nk},v) &\mapsto t^k \xi_{m}^{\chi(v)}. \nonumber
\end{align}
We now use $\rho_K(n,\chi)$ to obtain a representation of $\pi_1(M_n)$. Consider the composition $\pi_1(M_n) \stackrel{p_*}{\rightarrow} \pi_1(M_K) \to H_1(M_K;\Z) \cong \mathbb{Z}$.
Since the image of this map is isomorphic to $n \mathbb{Z}$, mapping to it produces a surjective map $\alpha \colon \pi_1(M_n) \to n \mathbb{Z}$.
Using the decomposition $H_1(M_n;\Z) \cong H_1(\Sigma_n(K);\Z) \oplus~\Z$ additionally gives rise to a homomorphism $\rho \colon  \pi_1(M_n) \to H_1(\Sigma_n(K);\Z)$. Consider the $\Z[\pi_1(M_n)]$-module structure on~$\C[t^{\pm 1}]$ given by the following composition:
\begin{align}
  \label{eq:ActionForMn}
  \alpha \times \chi \colon \pi_1(M_n) \stackrel{\alpha \times \rho}{\to} n \Z \times H_1(\Sigma_n(K);\Z) \stackrel{\rho_K(n,\chi)}{\to} \LC.
\end{align}

Under the assumption that $m$ is a prime power and that $\chi$ is a nontrivial character, Casson and Gordon show that $H_*(M_n;\C(t))$ vanishes~\cite[Corollary after Lemma 4]{CassonGordon2}.
In other words, under these conditions, the unitary representation $\alpha \times \chi$ is acyclic.

\begin{definition}
  \label{def:BlanchfieldCover}
  Let $n$ be an integer, let $m$ be a prime power and let $\chi \colon H_1(\Sigma_n(K);\Z) \to \Z_m$ be a nontrivial character.
  The \emph{Blanchfield pairing of the cover} $\Bl_{\alpha \times \chi}(M_n)$ is the Blanchfield pairing corresponding to the unitary acyclic representation~$\alpha \times \chi:$
  \begin{equation}
    \label{eq:BlanchfieldCover}
    \Bl_{\alpha \times \chi}(M_n) \colon H_1(M_n;\LC) \times H_1(M_n;\LC) \to \C(t) / \LC.
  \end{equation}
\end{definition}

Up to the end of this subsection, we shall use $\LC_{\alpha \times \chi}$ to denote the $\Z[\pi_1(M_n)]$-module structure on $\LC$ given by~(\ref{eq:ActionForMn}) and by $\LC^n_{\alpha(n,\chi)}$ the $\Z[\pi_1(M_K)]$-module structure on $\LC^n$ given by~(\ref{eq:ActionForMK}).
The following result is due to Herald, Kirk and Livingston~\cite[Theorem 7.1]{HeraldKirkLivingston}.

\begin{proposition}
  \label{prop:HeraldKirkLivingston}
  Let $n$ be a positive integer.
  There is a canonical chain isomorphism of left~$\LC$-modules between $C_*(M_n;\C[t^{\pm 1}]_{\alpha \times \chi})$ and $C_*(M_K;\C[t^{\pm 1}]^n_{\alpha(n,\chi)})$.
  In particular,  one obtains a canonical isomorphism between the following homology left $\LC$-modules:
  \begin{equation}
    \label{eq:IsomorphismHKL}
    H_1(M_n;\C[t^{\pm 1}]_{\alpha \times \chi}) \cong H_1(M_K;\C[t^{\pm 1}]^n_{\alpha(n,\chi)}).
  \end{equation}
\end{proposition}

The proof of Proposition~\ref{prop:HeraldKirkLivingston} relies on the relation between the representation $\gamma_K(n,\chi)$ of~\eqref{eq:Matrix} and the representation $\rho_K(n,\chi)$ of~\eqref{eq:Reprho(n,chi)}. Since we shall use this relation in Subsection~\ref{sub:SatelliteMetabelian}, we give some further details.

\begin{remark}
  \label{rem:InductionFunctor}
  Let $H$ be a subgroup of a group $G$ and let $\rho$ be a representation of $H$. 
  It is known that there is an~\emph{induced representation} $\operatorname{ind}_H^G\rho$ of $G$, we refer to \cite{HeraldKirkLivingston} for a construction and a discussion of its properties.
  Taking $G=\Z \ltimes H_1(\Sigma_n(K);\Z)$ and~$H=n\Z \times H_1(\Sigma_n(K);\Z)$, the proof of Proposition~\ref{prop:HeraldKirkLivingston} relies heavily on the isomorphism
  $$ \gamma_K(n,\chi)\cong\operatorname{ind}_H^G \rho_K(n,\chi). $$
\end{remark}

As the isomorphism described in~\eqref{eq:IsomorphismHKL} arises from a canonical chain isomorphism (essentially Shapiro's lemma), we obtain the following result which appears to be implicit in~\cite{MillerPowell}.
However, since no proof has appeared in the literature, we will give a detailed argument in appendix and, in particular, Corollary~\ref{cor:BlanchfieldDownstairs} will follow from the more general Proposition~\ref{prop:Shapiros-lemma-bl-forms}.

\begin{corollary}
  \label{cor:BlanchfieldDownstairs}
  The isomorphism described in~\eqref{eq:IsomorphismHKL} gives rise to an isometry between the twisted Blanchfield pairings $\Bl_{\alpha \times \chi}(M_n)$ and $ \Bl_{\alpha(n,\chi)}(K)$.
\end{corollary}


%

Summarizing the content of this section, we can choose to work alternatively with $\Bl_{\alpha(n,\chi)}(K)$ or with $\Bl_{\alpha \times \chi}(M_n)$.
The former is used in the paper of Miller-Powell~\cite{MillerPowell} and will be used in Subsections~\ref{sub:MillerPowell},~\ref{sub:SatelliteMetabelian} and in \cite{BCP_Compu}.
The latter is closer to the definition of the Casson-Gordon $\tau$ invariant with which we shall work in Subsections~\ref{sub:CassonGordon} and~\ref{sec:Proof}.

\subsection{A sliceness obstruction due to Miller and Powell}
\label{sub:MillerPowell}

The aim of this subsection is to state an obstruction to sliceness which is due to Miller and Powell~\cite{MillerPowell}.
\medbreak

Given an integer $n$,  we use $\lambda_n$ to denote the non-singular $\Q/\Z$-valued linking form on the finite abelian group $H_1(\Sigma_n(K);\Z)$.
Note that $H_1(\Sigma_n(K);\Z)$ is also a $\Z[\Z_n]$-module.
A \emph{metabolizer} of~$\lambda_n$ is a $\Z[\Z_n]$-submodule $P$ of $H_1(\Sigma_n(K);\Z)$ such that $P=P^\perp$, where the orthogonal complement is defined as~$P^\perp=\{y\in H_1(\Sigma_n(K);\Z)\ | \ \forall x\in P,\ \lambda_n(x,y)=~0\}$.
Finally, given a prime $q$, integers $b \geq a$ and a character $\chi \colon H_1(\Sigma_n(K);\Z) \to~\Z_{q^a}$, we use $\chi_b$ to denote the composition of $\chi$ with the natural inclusion $\Z_{q^a} \to \Z_{q^b}$.

The following obstruction to sliceness is due to Miller and Powell~\cite[Theorem 6.9]{MillerPowell}.

\begin{theorem}
  \label{thm:MillerPowell}
  Let $K$ be a slice knot.
  Then, for any prime power $n$, there exists a metabolizer~$P$ of~$\lambda_n$ such that for any prime power $q^a$, and any nontrivial character $\chi \colon  H_1(\Sigma_n(K);\Z) \to \Z_{q^a}$ vanishing on $P$, we have some $b \geq a$ such that the metabelian Blanchfield pairing $Bl_{\alpha(n,\chi_b)}(K)$ is metabolic.
\end{theorem}

Miller and Powell actually prove Theorem~\ref{thm:MillerPowell} under the weaker assumption that $K$ is 2-solvable.
We refer to Cochran-Orr-Teichner's landmark paper~\cite{CochranOrrTeichner} for the definition of $n$-solvability but do not delve deeper into this issue.

Combining Theorem~\ref{thm:MillerPowell} with the earlier sections, we obtain the following result.

\begin{theorem}
  \label{thm:MillerPowellSignature}
  Let $K$ be a slice knot.
  Then, for any prime power $n$, there exists a metabolizer~$P$ of~$\lambda_n$ such that for any prime power $q^a$, and any nontrivial character $\chi \colon  H_1(\Sigma_n(K);\Z) \to \Z_{q^a}$ vanishing on $P$, we have some $b \geq a$ such that the signature function $\sigma^{av}_{K,\alpha(n,\chi_b)}$ is zero.
\end{theorem}
\begin{proof}
  Theorem~\ref{thm:MillerPowell} implies that $\Bl_{\alpha(k,\chi_b)}(K)$ is metabolic.
  Consequently, item~\ref{item:S4} of Proposition~\ref{prop:SignatureFunctionAlgebra} implies that the averaged signature function of $\Bl_{\alpha(n,\chi_b)}(K)$ vanishes.
\end{proof}

The obstruction of Theorem~\ref{thm:MillerPowell} will be used in \cite{BCP_Compu} to prove that a certain algebraic knot is not slice, thus recovering a result of Hedden-Kirk-Livingston~\cite{HeddenKirkLivingston}.

\subsection{The Casson-Gordon $\tau$-invariant}
\label{sub:CassonGordon}

Given an oriented knot $K$ and a prime power order character~$\chi \colon H_1(\Sigma_n(K);\Z) \to \Z_m$, we recall Casson-Gordon's construction of the Witt class~$\tau(K,\chi)$~\cite{CassonGordon2}.
\medbreak

Using the isomorphism $ H_1(M_n;\Z) \cong H_1(\Sigma_n(K);\Z) \oplus \Z$, the character $\chi$ induces a character on $H_1(M_n;\Z)$ for which we use the same notation. Recalling that the projection $p \colon M_n \to M_K$ gives rise to a surjection $\alpha \colon \pi_1(M_n) \to n\Z \cong \Z$, we obtain a homomorphism
\begin{equation}
  \label{eq:CGHomom}
  \alpha \times \chi \colon \pi_1(M_n) \to \mathbb{Z} \times \mathbb{Z}_m.
\end{equation}
Since the bordism group $\Omega_3(\mathbb{Z} \times \mathbb{Z}_m)$ is finite (this can be seen using the Atiyah-Hirzbruch spectral sequence; see e.g.~\cite[Chapter 1, Section 7]{ConnerFloyd}), there is an integer $r$ such that $r$ copies of $(M_n,\alpha \times \chi)$ bound~$(V_n,\psi)$ for some $4$-manifold $V_n$ and for some homomorphism $\psi \colon \pi_1(V_n) \to \Z \times \Z_m$.
Consider the map $\C[\Z \times \Z_m] \to \C[t^{\pm 1}]$ which sends the generator of $\Z_m$ to $\xi_{m} = e^{2\pi i/m}$ and the generator of~$\Z$ to $t$.
Endow $\mathbb{C}(t)$ with the right $\mathbb{Z}[\pi_1(V_n)]$-module structure which arises from the composition 
\begin{equation}
  \label{eq:ModuleStructure}
  \mathbb{Z}[\pi_1(V_n)] \xrightarrow{\psi} \mathbb{C}[\mathbb{Z} \times \mathbb{Z}_m] \to \LC \to \mathbb{C}(t).
\end{equation}
We will denote by \(\mathbb{C}(t)_{\alpha \times \chi}\)
the module \(\mathbb{C}(t)\) equipped with the action of the group ring $\Z[\pi_1(V_n)]$
as described above.
Observe that~(\ref{eq:ModuleStructure}) also provides a right $\Z[\pi_1(V_n)]$-module structure on $\LC$ (we write $\LC_{\alpha \times \chi}$) which restricts to the one described in~(\ref{eq:ActionForMn}) on each boundary component of~$V_n$.
We will denote the related \(\mathbb{C}[\Z\times\Z_{m}]\)-module by~\(\LC_{\Z \times \Z_m}\).
Next, consider the $\C(t)$-valued intersection form~$\lambda_{\mathbb{C}(t)_{\alpha \times \chi},V_n}$ on $H_2(V_n;\mathbb{C}(t)_{\alpha \times \chi})$.
Casson and Gordon show that if $\chi$ is a character of prime power order, then this Hermitian form is non-singular~\cite[Corollary following Lemma 4]{CassonGordon2}.
We record this result for later reference.

\begin{lemma}
  \label{lem:CassonGordon}
  Let $n$ be an integer and let $m$ be a prime power.
  If $\chi$ is a nontrivial character, then~$H_*(M_n;\C(t)_{\alpha \times \chi})=0$ and the intersection pairing~$\lambda_{\C(t)_{\alpha \times \chi},V_n}$ is non-singular.
\end{lemma}

Since $\lambda_{\C(t)_{\alpha \times \chi},V_n}$ is non-singular, it gives rise to a Witt class $[\lambda_{\C(t)_{\alpha \times \chi},V_n}]$ in $W(\C(t))$.
On the other hand, the standard intersection pairing $\lambda_{\Q,V_n}$ may very well be singular.
Consequently, we use~$\lambda_{\Q,V_n}^{\text{nonsing}}$ to denote the form obtained by moding out the radical.
The Witt class of $\lambda_{\Q,V_n}^{\text{nonsing}}$ provides an element in~$W(\mathbb{Q})$.
Since the inclusion map $\Q \to \C(t)$ induces a group homomorphism $i \colon W(\mathbb{Q}) \to W(\C(t))$, we therefore obtain an element $i([\lambda_{\Q,V_n}^{\text{nonsing}}])$ in $W(\C(t))$.

\begin{definition}
  \label{def:CassonGordonInvariant}
  Let $K$ be an oriented knot, let $n$ be an integer, let $m$ be a prime power and let $\chi \colon H_1(\Sigma_n(K);\Z) \to \Z_m$ be a non-trivial character.
  The \emph{Casson-Gordon $\tau$-invariant} is the Witt~class
  $$\tau(K,\chi):=([\lambda_{\C(t)_{\alpha \times \chi},V_n}]-i([\lambda_{\Q,V_n}^{\text{nonsing}}]) \otimes \frac{1}{r} \in  W(\C(t))\otimes \mathbb{Q}.$$
\end{definition}

Casson and Gordon show that $\tau(K,\chi)$ is independent both of the choice of the 4-manifold $V_n$ and of the extension $\psi$ of $\alpha \times \chi$ to $\pi_1(V_n)$~\cite{CassonGordon2}.
Furthermore, as we shall recall later on, $\tau(K,\chi)$ provides an obstruction to sliceness.

\begin{remark}
  \label{rem:Cyclotomic}
  Casson and Gordon define their Witt class as an element of $W(\Q(\xi_m)(t))$~\cite{CassonGordon2}.
  Since the natural map $\Q(\xi_m)(t) \to \C(t)$ induces a map on Witt groups, we also obtain a Witt class in~$W(\C(t))$.
  Our reason for working with~$W(\C(t))$ instead of $W(\Q(\xi_m)(t))$ is the following: the former Witt group is much simpler than the latter and, in particular, it is more amenable to the machinery we developed in~\cite{BCP_Alg}.
\end{remark}

\subsection{Casson-Gordon signatures and averaged Blanchfield signatures}
\label{sec:Proof}
In this section, we investigate the relation between the Casson-Gordon invariant $\tau(K,\chi)$ and the twisted Blanchfield pairing $\Bl_{\alpha(n,\chi)}(K)$.
In order to make both these invariants more tractable, we use averaged signatures.
\medbreak

Let $A(t)$ be a matrix over $\C(t)$.
The function $\sign_\omega(A(t)):=\sign(A(\omega))$ is a step-function with finitely many discontinuities, and at each discontinuity $\omega$, we can take the average of the one-sided limits in order to obtain a rational number $\sign_\omega^{\text{av}}(A(t))$.
As explained in~\cite[discussion preceding Theorem 3]{CassonGordon2}, 
one now obtains a well-defined homomorphism $\sign^{\text{av}}_\omega \colon W(\C(t)) \to \Q$ by setting
$$\sign_\omega^{\text{av}}([A(t)]):=\text{sign}^{\text{av}}_\omega(A(t)).$$
This map also gives rise to a well-defined map on $W(\C(t)) \otimes \Q$.
We now apply this discussion to extract signatures from the Casson-Gordon invariant $\tau(K,\chi)$ which is an element of $W(\C(t)) \otimes \Q$.

\begin{definition}
  \label{def:CassonGordonSignature}
  The \emph{Casson-Gordon signature} associated to $\omega \in S^1$ is $\sign_\omega^{\text{av}}(\tau(K,\chi)) \in \Q$.
\end{definition}

Before returning to Blanchfield pairings, let us make Definition~\ref{def:CassonGordonSignature} somewhat more concrete.

\begin{remark}
  \label{rem:CassonGordonSignature}
  Suppose $V_n$ is a $4$-manifold whose boundary consists of a disjoint union of $r$ copies of~$M_n$ and over which $\alpha \times \chi$ extends.
  In other words, we have 
  \[\tau(K,\chi)= ([\lambda_{\C(t)_{\alpha\times \chi},V_n}]-i_*([\lambda_{\Q,V_n}^{\text{nonsing}}])) \otimes \frac{1}{r}.\] 
  Since $\sign_\omega^{\text{av}}$ is a homomorphism, we obtain
  \[\sign_\omega^{\text{av}}(\tau(K,\chi))=\frac{1}{r}(\sign_\omega^{\text{av}}([\lambda_{\C(t)_{\alpha\times \chi},V_n}])-\sign_\omega^{\text{av}}([\lambda_{\Q,V_n}^{\text{nonsing}}])).\] 
  Clearly the latter term is simply equal to $\sign(V_n)$, the (untwisted) signature of $V_n$.
  Consequently, if $A(t)$ represents $\lambda_{\C(t)_{\alpha\times \chi},V_n}$, then we deduce that
  $$\sign_\omega^{\text{av}}(\tau(K,\chi))=\frac{1}{r}(\sign^{\text{av}}_\omega(A(t))-\sign(V_n)).$$
  Note that $A(1)$ does \emph{not} represent the standard intersection form $\lambda_{\Q,V_n}$.
  It does however represent the twisted intersection form $\lambda_{\C_{\xi_{m}},V_n}$ which arises from the coefficient system $\pi_1(V_n) \to \Z_m \to \C$, where the latter map sends the generator of $\Z_m$ to $\xi_m=\text{exp}(\frac{2\pi i}{m})$.
\end{remark}

Returning to Blanchfield pairings, recall from Definition~\ref{def:BlanchfieldMetabelian} that $\Bl_{\alpha(n,\chi)}(K)$ denotes the twisted Blanchfield pairing associated to the representation $\alpha(n,\chi) \colon \pi_1(M_K) \to GL_n(\LC)$.
Recall furthermore that $\sigma_{K,\alpha(n,\chi)} \colon S^1 \to \Z$ denotes the associated twisted signature~function.

The following theorem relates our Blanchfield signatures to the Casson-Gordon signatures.

\begin{theorem}
  \label{thm:BlanchfieldCG}
  Let $K$ be an oriented knot and $n$ be an integer.
  Let $\Sigma_n(K)$ be the $n$-fold cyclic cover of $S^3$ branched along $K$ and let $M_n$ be the $n$-fold cyclic cover of $M_K$.
  If $\chi \colon H_1(\Sigma_n(K);\Z) \to~\Z_m$ is a nontrivial character of prime power order $m$, then the following statements hold:
  \begin{itemize}
  \item[(a)] The twisted Blanchfield pairing $\Bl_{\alpha(n,\chi)}(K)$ is representable.
  \item[(b)] For each $\omega$ in $S^1$, we have the equality 
    $$  -\sigma^{\text{av}}_{K,\alpha(n,\chi)}(\omega)=\sign_\omega^{\text{av}}(\tau(K,\chi))-\sign_1^{\text{av}}(\tau(K,\chi)).$$
  \end{itemize}
\end{theorem}

Before moving towards the proof, we start with a remark and a corollary to Theorem~\ref{thm:BlanchfieldCG}.

\begin{remark}
  \label{rem:RemarkOnCGTheorem}
  Although we stated Theorem~\ref{thm:BlanchfieldCG} using the metabelian Blanchfield pairing $\Bl_{\alpha(n,\chi)}(K)$, the proof will use the Blanchfield pairing of the cover $\operatorname{Bl}_{\alpha \times \chi}(M_n)$.
  Indeed, since Corollary~\ref{cor:BlanchfieldDownstairs} states that $\operatorname{Bl}_{\alpha(n,\chi)}(K)$ and $\operatorname{Bl}_{\alpha \times \chi}(M_n)$ are isometric, one linking form is representable if and only if the other one is and, for every $\omega \in S^1$, the second item of Theorem~\ref{thm:BlanchfieldCG} could have been written using the signature function of $\operatorname{Bl}_{\alpha \times \chi}(M_n)$, since the following equality holds:
  \begin{equation}
    \label{eq:SignatureEquality}
    \sigma^{\text{av}}_{K,\alpha(n,\chi)}(\omega)= \sign_\omega^{\text{av}}(\Bl_{\alpha \times \chi}(M_n)).
  \end{equation}
\end{remark}

Since $\tau(K,\chi)$ is known to provide an obstruction to sliceness~\cite{CassonGordon2}, a similar conclusion holds for~$\sigma_{K,\alpha(n,\chi)}$, thus yielding a second proof of (a variation on) Theorem~\ref{thm:MillerPowellSignature}.
Indeed, recalling that~$\lambda_n$ denotes the $\Q/\Z$-valued linking form on $H_1(\Sigma_n(K);\Z)$, the following result is an immediate corollary of Theorem~\ref{thm:BlanchfieldCG} and \cite[Theorem 2]{CassonGordon2}.

\begin{corollary}
  \label{cor:CG}
  If $K$ is slice, then there is a subgroup $G$ of $H_1(\Sigma_n(K))$ such that \(|G|^{2} = |H_{1}(\Sigma_{n}(K))|\) and $\lambda_n$ vanishes on $G$ and, for every character $\chi$ vanishing on $G$, the signature function $ \sigma^{\text{av}}_{K,\alpha(n,\chi)}(\omega)$ is zero.
\end{corollary}

We now move towards the proof of Theorem~\ref{thm:BlanchfieldCG}.
First of all, as we mentioned in Remark~\ref{rem:RemarkOnCGTheorem}, we can work with the pairing $\Bl_{\alpha \times \chi}(M_n)$ instead of $\Bl_{\alpha(n,\chi)}(K)$.
Next, as in Subsection~\ref{sub:CassonGordon}, we choose a smooth
$4$-manifold~$V_n$ such that the boundary of $V_n$ consists of $r$ copies of $M_n$ and such that the representation extends.
Without loss of generality, we can assume that $\pi_1(V_n)=\Z_m \times \Z$: indeed since~$\Z_m \times \Z$ is finitely normally presented,
one can perform finitely many surgeries to obtain the desired fundamental group, while leaving the boundary fixed.

The following lemma collects some algebraic statements which we shall need later on.

\begin{lemma}
  \label{lem:ev}
  The following statements hold:
  \begin{itemize}
  \item[(a)] The $\LC$-module $H_1(V_n;\LC_{\alpha \times \chi})$ vanishes.
  \item[(b)] The map  $\ev \circ \kappa \colon H^2(V_n;\LC_{\alpha \times \chi}) \rightarrow \makeithash{\Hom_{\LC}(H_2(V_n;\LC_{\alpha \times \chi}),\LC)}$ is an isomorphism.
  \item[(c)] The $\LC$-module $H_2(V_n,\partial V_n;\LC_{\alpha \times \chi})$ is free.
  \item[(d)] The inclusion induced map $ H^2(V_n,\partial V_n;\C(t)_{\alpha\times \chi}) \to H^2(V_n;\C(t)_{\alpha\times \chi})$ is an isomorphism.
  \end{itemize}
\end{lemma}
\begin{proof}
  We start by proving (a).
  Set $G:=\Z_m \times \Z$ for brevity. Recalling the notation introduced below~\eqref{eq:ModuleStructure}, we claim that \(\LC_{G}\) is a projective \(\C[G]\)-module.
  To see this, first note that \(\C[G] = \C[t^{\pm 1}] \otimes_{\C} \C[\Z_m]\) and \(\C[\Z_m] = \bigoplus_{j = 0}^{m-1}~\C_{\xi_m^j}\), where~\(\C_{\xi_m^j}\) denotes the irreducible complex representation of \(\Z_m\) with the action of the cyclic group given by multiplication by the root of unity \(\xi_m^j\), for \(\xi_m = \exp\left(\frac{2 \pi i}{m}\right)\).
  Therefore, we have \(\C[G] = \bigoplus_{j=0}^{m-1} \C[t^{\pm 1}] \otimes_{\C}~\C_{\xi_m^j}\) and \(\LC_{G} = \C[t^{\pm 1}] \otimes_{\C}~\C_{\xi_m}\), concluding the proof of the claim.
  The claim now implies that $H_1(V_n; \LC_{\alpha \times \chi})$ is a summand of~$H_1(V_n; \C[G])$.
  Since $\pi_1(V_n)=G$, the corresponding $G$-cover of $V_n$ is simply-connected and thus~$H_1(V_n;\C[G])$ vanishes.
  Consequently \(H_1(V_n; \LC_{\alpha \times \chi})\) also vanishes since it is a direct summand of \(H_1(V_n; \C[G])\).
  This concludes the proof of (a).

  We now prove (b) according to which the evaluation map
  \[\ev \circ \kappa \colon H^2(V_n;\LC_{\alpha \times \chi}) \rightarrow \makeithash{\Hom_{\LC}(H_2(V_n;\LC_{\alpha \times \chi}),\LC)}\]
  is an isomorphism.
  Since $H_1(V_n;\LC_{\alpha \times \chi})=0$ by the first statement and \(H_0(V_n;\LC_{\alpha \times \chi})=~0\) (recall Lemma~\ref{lem:H0CG}), this follows immediately from the universal coefficient theorem.

  We move on to (c) which asserts that $H_2(V_n,\partial V_n;\LC_{\alpha \times \chi})$ is a free $\LC$-module.
  Using Poincar\'e duality and the second statement, we deduce that $H_2(V_n,\partial V_n;\LC_{\alpha \times \chi})$ is isomorphic to $\makeithash{\Hom_{\LC}(H_2(V_n;\LC_{\alpha \times \chi}),\LC)}$ which is free because $\LC$ is a PID.

  Finally, we deal with (d), that is we show that the inclusion induced map~$H^2(V_n,\partial V_n;\C(t)_{\alpha\times \chi}) \to H^2(V_n;\C(t)_{\alpha\times \chi})$ is an isomorphism.
  Since $H_*(M_n;\C(t)_{\alpha\times \chi})$ vanishes by Lemma~\ref{lem:CassonGordon} and $\partial V_n$ consists of~$r$ disjoint copies of $M_n$, we deduce that~$H_*(\partial V_n;\C(t)_{\alpha\times \chi})$ also vanishes.
  The long exact sequence of the pair~$(V_n,\partial V_n)$ implies that $H_2(V_n;\C(t)_{\alpha\times \chi}) \rightarrow H_2(V_n,\partial V_n;\C(t)_{\alpha\times \chi})$ is an isomorphism and the result follows by Poincar\'e duality.
  This concludes the proof of (d) and thus the proof of the lemma.
\end{proof}

Since $\LC$ is a PID, we can decompose any $\LC$-module $H$ as the direct sum of its free part~$FH$
  and its torsion part $TH$.
While $FH$ is typically defined as $H/TH$, here it is convenient to think of $FH$ as a submodule of $H$ such that $H  \cong FH \oplus TH$.
While such a submodule is not unique, different choices will not affect the remainder of the proof. 
In particular, it follows that
\[\makeithash{\Hom_{\LC}(H_2(V_n;\LC_{\alpha \times \chi}),\LC)}=\makeithash{\Hom_{\LC}(FH_2(V_n;\LC_{\alpha \times \chi}),\LC)}.\] 
Consequently, point (b) of Lemma~\ref{lem:ev} provides the following well-defined isomorphism:
\begin{equation}
  \label{eq:EvaluationFree}
  ev \circ \kappa \colon H^2(V_n;\LC_{\alpha \times \chi}) \rightarrow \makeithash{\Hom_{\LC}(FH_2(V_n;\LC_{\alpha \times \chi}),\LC)}.
\end{equation}
Composing the inclusion $FH_2(V_n;\LC_{\alpha \times \chi}) \to H_2(V_n;\LC_{\alpha \times \chi})$ with the inclusion induced map $H_2(V_n;\LC_{\alpha \times \chi}) \to H_2(V_n,\partial V_n;\LC_{\alpha \times \chi})$ gives rise to a well-defined map 
\[FH_2(V_n;\LC_{\alpha \times \chi}) \to H_2(V_n,\partial V_n;\LC_{\alpha \times \chi}).\] 
Combining these remarks, there is a well-defined intersection form
$$\lambda_{\LC_{\alpha \times \chi},V_n} \colon FH_2(V_n;\LC_{\alpha \times \chi}) \times FH_2(V_n;\LC_{\alpha \times \chi}) \to \LC.$$
Here,  note that the choice of the submodule of $FH$ does not  affect the isometry type of $\lambda_{\LC_{\alpha \times \chi},V_n}$.

Next, we describe the Blanchfield $\Bl_{\alpha \times \chi}(\partial V_n)$ on $H_1(\partial V_n;\LC_{\alpha \times \chi})$.
As $\partial V_n$ consists of $r$ copies of~$M_n$, we deduce that $H_1(\partial V_n;\LC_{\alpha \times \chi})$ is isomorphic to the direct sum of~$r$ copies of $H_1(M_n;\LC_{\alpha \times \chi})$.
Since the latter $\LC$-module is torsion, so is the former.
We deduce that $ \Bl_{\alpha\times \chi}(\partial V_n)$ decomposes as the direct sum $\Bl_{\alpha \times \chi}(M_n) \oplus \cdots \oplus \Bl_{\alpha \times \chi}(M_n)$, where there are $r$ direct summands.

The proof of the next result is postponed until the end of the proof of Theorem~\ref{thm:BlanchfieldCG}.
\begin{proposition}
  \label{prop:Boundary}
  Any matrix $A(t)$ representing $\lambda_{\LC_{\alpha \times \chi},V_n}$ also represents~$\lambda_{\C(t)_{\alpha\times \chi},V_n}$ and~$-\Bl_{\alpha \times \chi}(\partial V_n)$.
\end{proposition}
\begin{proof}[Proof of Theorem~\ref{thm:BlanchfieldCG} assuming Proposition~\ref{prop:Boundary}]
  We start by proving the second statement, namely the equality of signatures.
  Let $A(t)$ be a matrix representing $\lambda_{\LC_{\alpha \times \chi},V_n}$.
  Since $A(t)$ represents the twisted intersection form $\lambda_{\C(t)_{\alpha\times \chi},V_n}$, the definition of the Casson-Gordon signature and Remark~\ref{rem:CassonGordonSignature} imply that 
  \begin{equation}
    \label{eq:CassonGordonForProof}
    \sign_\omega^{\text{av}}(\tau(K,\chi))=\frac{1}{r}(\sign^{\text{av}}_\omega(A(t))-\sign(V_n)).
  \end{equation}
  Since Proposition~\ref{prop:Boundary} tells us that $A(t)$ also represents $-\Bl_{\alpha \times \chi}(\partial V_n)$, item~\ref{item:S5} of Proposition~\ref{prop:SignatureFunctionAlgebra} implies that
  $-\sign_\omega^{\text{av}}(\Bl_{\alpha \times \chi}(\partial V_n))=\sign^{\text{av}}_\omega(A(t))-\sign^{\text{av}}_1(A(t))$.
  On the other hand, since $\partial V_n$ consists of $r$ disjoint copies of $M_n$, we observed previously that $ \Bl_{\alpha \times \chi}(\partial V_n)=\Bl_{\alpha \times \chi}(M_n) \oplus \ldots \oplus \Bl_{\alpha \times \chi}(M_n)$ and consequently $\sign_\omega^{\text{av}}(\Bl_{\alpha \times \chi}(\partial V_n))=r \cdot  \sign_\omega^{\text{av}}(\Bl_{\alpha \times \chi}(M_n))$.
  Combining these two observations, we obtain the following equality:
  \begin{equation}
    \label{eq:BlanchfieldSignatureProof2}
    -\sign^{\text{av}}_\omega(\Bl_{\alpha \times \chi}( M_n))=\frac{1}{r}(\sign^{\text{av}}_\omega(A(t))-\sign^{\text{av}}_1(A(t))).
  \end{equation}
  Finally, adding and subtracting a $\text{sign}(V_n)$ term in~(\ref{eq:BlanchfieldSignatureProof2}) and combining the result with~(\ref{eq:CassonGordonForProof}), we can conclude the proof of the first assertion: 
  \begin{align*}
    -\sign^{\text{av}}_\omega(\Bl_{\alpha \times \chi}( M_n))
    &=\frac{1}{r}\left(\sign^{\text{av}}_\omega(A(t))-\sign(V_n)\right)-\frac{1}{r}\left(\sign^{\text{av}}_1(A(t))-\sign(V_n) \right) \\
    &=\sign_\omega^{\text{av}}(\tau(K,\chi))-\sign_1^{\text{av}}(\tau(K,\chi)).
  \end{align*}

  It remains to prove the first statement, namely that the twisted Blanchfield pairing $\Bl_{\alpha \times \chi}(M_n)$ is representable.
  In the previous paragraph, we showed that $ \Bl_{\alpha \times \chi}(\partial V_n)=\Bl_{\alpha \times \chi}(M_n) \oplus \ldots \oplus \Bl_{\alpha \times \chi}(M_n)$.
 Since we know from Proposition~\ref{prop:Boundary} that $\Bl_{\alpha \times \chi}(\partial V_n)$ is representable,
    $\Bl_{\alpha \times \chi}(M_n) $ is itself representable
     ($\Bl_{\alpha\times\chi}(\partial V_n)$ satisfies the signature condition of item~\ref{item:S6} of Proposition~\ref{prop:SignatureFunctionAlgebra} and thus so does~$\Bl_{\alpha\times\chi}(M_n)$),
  and so the theorem is proved.
\end{proof}

We now prove Proposition~\ref{prop:Boundary}; the strategy follows very closely~\cite[Section 5.2]{ConwayFriedlToffoli} which itself is based on the proof of~\cite[Theorem 2.6]{BorodzikFriedl}.

\begin{proof}[Proof of  Proposition~\ref{prop:Boundary}]
  Consider the following diagram in which all homomorphisms are understood to be homomorphisms of $\LC$-modules:
  \begin{equation}
    \label{eq:LargeDiagram}
    \xymatrix@R0.6cm@C-1cm{
      FH_2 \ar[rr] \ar[dd]^{\Theta} && H\ar[r]^-\partial \ar[d]^{PD} & H_1(\partial V_n;\Lambda_{\alpha \times \chi})\ar[d]^{PD}\ar[rr] \ar@/^/[dddrr]^{\Omega} &&0\\
      && H^2(V_n;\Lambda_{\alpha \times \chi}) \ar[r]\ar[d]\ar[lld]^{\ev \circ \kappa}& H^2(\partial V_n;\Lambda_{\alpha \times \chi})\ar[d]^{BS^{-1}} && \\
      \makeithash{\Hom_\Lambda(FH_2,\Lambda)} \ar[d] && H^2(V_n;\Omega_{\alpha \times \chi})\ar[lld]^{\ev \circ \kappa}\ar[d]^{\cong} & H^1(\partial V_n;\Omega_{\alpha \times \chi}/\Lambda_{\alpha \times \chi})\ar[rrd]^{\ev \circ \kappa} \ar[d]&&   \\ 
      \makeithash{ \Hom_\Lambda(FH_2,\Omega) } \ar[d]^{\cong }&& H^2(V_n,\partial V_n;\Omega_{\alpha \times \chi})\ar@<-2pt>`d[r] `[r] [r]\ar[lld]^{\ev \circ \kappa}\ \ & H^2(V_n,\partial V_n;\Omega_{\alpha \times \chi}/\Lambda_{\alpha \times \chi}) \ar[rrd]^{\ev \circ \kappa}&& \makeithash{ \Hom_\Lambda(H_1(\partial V_n;\Lambda_{\alpha \times \chi}),\Omega/\Lambda) } \ar[d]_{\partial^{\#}} \\
      \makeithash{\Hom_\Lambda(H,\Omega)} \ar[rrrr]^{} &&&&&  \makeithash{ \Hom_\Lambda(H,\Omega/\Lambda) }.
    }
  \end{equation}
  To make the diagram more concise we have used the following shorthands:
  \begin{itemize}
  \item $\Lambda=\LC$;
  \item $FH_2=FH_2(V_n;\LC_{\alpha \times \chi})$;
  \item $H=H_2(V_n,\partial V_n;\LC_{\alpha \times \chi})$;
  \item $\Omega=\C(t)$.
  \end{itemize}

  The top horizontal line of \eqref{eq:LargeDiagram} is exact thanks to the long exact sequence of the pair $(V_n,\partial V_n)$ together with Lemma~\ref{lem:ev}: exactness at the rightmost end is guaranteed by the first point of Lemma~\ref{lem:ev},  while exactness at the middle follows from the third point.
  In more details, since~$H:=H_2(V_n,\partial V_n;\LC_{\alpha \times \chi})$ is free, the image of $H_2(V_n;\LC_{\alpha \times \chi})$ in $H$ is equal to the image of~$FH_2(V_n;\LC_{\alpha \times \chi})$ in $H$. 

The commutativity of the top middle square is a consequence of the definition of the Poincar\'e duality isomorphism.
  All the squares involving evaluation maps clearly commute, while the upper left (resp. right) square commutes by definition of the intersection (resp. Blanchfield) pairing.
  Finally, the middle rectangle \textit{anti}-commutes thanks to~\cite[Lemma~5.4]{ConwayFriedlToffoli}, see~\cite[Appendix A]{ConwayBlanchfield} for a proof.

  Inspired by~(\ref{eq:LargeDiagram}), we define a pairing $\theta$ on $H_2(V_n,\partial V_n;\LC_{\alpha \times \chi})$ by the composition
  \begin{align*}
    H_2(V_n,\partial V_n;\LC_{\alpha \times \chi}) &\xrightarrow{\PD} H^2(V_n;\LC_{\alpha \times \chi}) \rightarrow H^2(V_n;\C(t)_{\alpha\times \chi}) \cong H^2(V_n,\partial V_n;\C(t)_{\alpha\times \chi}) \\
                              &\rightarrow \makeithash{\Hom_{\LC}(H_2(V_n,\partial V_n;\LC_{\alpha \times \chi}),\C(t))},
  \end{align*} 
  where the third map is an isomorphism thanks to the fourth point of Lemma~\ref{lem:ev}.
  The commutativity of the diagram in equation~\eqref{eq:LargeDiagram} immediately implies the commutativity of 
  \begin{equation}
    \label{eq:DiagramThreeLines}
    \xymatrix@C1.4cm@R0.5cm{ FH_2(V_n;\LC_{\alpha \times \chi}) \times FH_2(V_n;\LC_{\alpha \times \chi}) \ar[r]^{\hskip 1cm \ \ \ \  \ \ \ \ \ \ \ \ -\lambda_{\LC_{\alpha \times \chi},V_n}}\ar[d]^{} & \LC\ar[d]  \\
      H_2(V_n,\partial V_n;\LC_{\alpha \times \chi}) \times H_2(V_n,\partial V_n;\LC_{\alpha \times \chi}) \ar[r]^{\ \ \ \ \ \ \ \ \ \ \ \ \ \ \ \ \ \ \ \  -\theta}\ar[d]^{\partial \times \partial} & \C(t) \ar[d]   \\
      H_1(\partial V_n;\LC_{\alpha \times \chi}) \times H_1(\partial V_n;\LC_{\alpha \times \chi}) \ar[r]^{\hskip 1cm \ \ \ \ \ \ \ \ \  \ \ \Bl_{\alpha \times \chi}(\partial V_n)}& \C(t)/\LC.
    }
  \end{equation} 

  We now pick bases in order to obtain matrices for the intersection form.
  Namely, choose any basis~$\mathcal{C}$ of~$FH_2(V_n;\LC_{\alpha \times \chi})$ and endow $\makeithash{\Hom_{\LC}(FH_2(V_n;\LC_{\alpha \times \chi}),\LC)}$ with the corresponding dual basis $\mathcal{C}^*$.
  Let $A(t)$ denote the matrix of the $\LC$-intersection form $\lambda_{\LC_{\alpha \times \chi},V_n}$ with respect to these bases.

  Next, we use $\mathcal{C}$ and $\mathcal{C}^*$ to base $H_2(V_n;\C(t)_{\alpha\times \chi})$ and $\makeithash{\Hom_{\LC}(H_2(V_n;\C(t)_{\alpha\times \chi}),\C(t))}$.
  In other words, as claimed in the first part of Proposition~\ref{prop:Boundary}, there are bases with respect to which both~$\lambda_{\C(t)_{\alpha\times \chi},V_n}$ and $\lambda_{\LC_{\alpha \times \chi},V_n}$ are represented by $A(t)$.
  Therefore, to conclude the proof it only remains to show that $A(t)$ also represents the Blanchfield pairing on $\partial V_n$.

  With this aim in mind, we start by providing a basis for $H_2(V_n,\partial V_n;\LC_{\alpha \times \chi})$ (which is free by Lemma~\ref{lem:ev}(b)).
  This will allow us to represent the inclusion induced map $FH_2(V_n;\LC_{\alpha \times \chi}) \to H_2(V_n,\partial V_n;\LC_{\alpha \times \chi})$ by a matrix.
  First, consider the following commutative diagram of $\LC$-homomorphisms:
  \begin{equation}
    \label{eq:DiagramForBasis}
    \xymatrix@R0.5cm@C0.9cm{
      FH_2(V_n;\LC_{\alpha \times \chi})\ar[r]\ar[d]^-{PD}_{\cong} & H_2(V_n,\partial V_n;\LC_{\alpha \times \chi})\ar[d]^-{PD}_{\cong} \\
      FH^2(V_n,\partial V_n;\LC_{\alpha \times \chi})\ar[r] & H^2(V_n;\LC_{\alpha \times \chi})\ar[d]^-{\ev \circ \kappa}_{\cong}\\
      &\makeithash{\Hom_{\LC}(FH_2(V_n;\LC_{\alpha \times \chi}),\LC)}.}
  \end{equation}
  Here the bottom-right map is an isomorphism thanks to the second point of Lemma~\ref{lem:ev}.
  We now equip $H_2(V_n,\partial V_n;\LC_{\alpha \times \chi})$ with the basis induced from $\mathcal{C}^*$ and two isomorphisms in the right column of diagram in~(\ref{eq:DiagramForBasis}).
  Arguing as in~\cite[Claim in Section 5.3]{ConwayFriedlToffoli}, elementary linear algebra shows that with respect to these bases, the map $FH_2(V_n;\LC_{\alpha \times \chi}) \to H_2(V_n,\partial V_n;\LC_{\alpha \times \chi})$ is represented by the matrix $A(t)^T$.
  Rewriting the commutative diagram  of equation (\ref{eq:DiagramThreeLines}) in terms of these bases, we obtain
  \[ \xymatrix@C2.3cm@R0.5cm{\LC^n \times \LC^n \ar[r]^{(a,b) \mapsto -a^T A(t) \makeithash{b}}\ar[d]_{(a,b) \mapsto (A(t)^Ta,A(t)^Tb)} & \LC \ar[d]  \\
      \LC^n \times \LC^n\ar[d]\ar[r]^{(a,b) \mapsto -a^T A(t)^{-1} \makeithash{b}} & \C(t) \ar[d] \\
      H_1(\partial V_n;\LC_{\alpha \times \chi}) \times H_1(\partial V_n;\LC_{\alpha \times \chi})\ar[r]^{\hskip 1cm \ \ \ \ \ \ \ \ \ \ \ \Bl_{\alpha \times \chi}(\partial V_n)} & \C(t)/\LC.
    } \]
  Here the middle horizontal map is determined by the top horizontal map, the vertical maps and the commutativity.
  Since $H_1(V_n;\LC_{\alpha \times \chi})$ vanishes by the first point of Lemma~\ref{lem:ev}, we deduce that the Blanchfield pairing on~$H_1(\partial V_n; \LC_{\alpha \times \chi})$ is represented by $-A(t)$.
  This concludes the proof of Proposition~\ref{prop:Boundary} and thus the proof of Theorem~\ref{thm:BlanchfieldCG}.
\end{proof}

\subsection{Satellite formulas for metabelian Blanchfield forms}
\label{sub:SatelliteMetabelian}

Given two knots $K,P$ and an unknotted curve $\eta$ in the complement of~$P$, we use $P(K,\eta)$ to denote the resulting satellite knot.
As described in Subsection~\ref{sub:Metabelian}, for a character 
$\chi \colon H_1(\Sigma_n(P(K,\eta));\Z) \to \Z_{m}$, there is an associated metabelian representation $\alpha(n,\chi) \colon \pi_1(M_{P(K,\eta)})  \to GL_n(\C[t^{\pm 1}])$.
The goal of this subsection is to apply the satellite formula of Theorem~\ref{thm:CablingTheorem} to~$\alpha(n,\chi)$. On the level of signatures, the result is reminiscent of Litherland's description of the behavior of the Casson-Gordon invariants of satellite knots~\cite[Theorem 2]{LitherlandCobordism}, but differs in the winding number zero case.
\medbreak
If the representation $\alpha(n,\chi)$ is $\eta$-regular, then it gives rise to representations $\alpha(n,\chi)_P$ on~$\pi_1(M_P)$ and~$\alpha(n,\chi)_K$ on~$\pi_1(M_K)$.
The representation $\alpha(n,\chi)_P$ can be shown to agree with~$\alpha(n,\chi_P)$, where $\chi_P$ is the character induced by $\chi$ on~$H_1(\Sigma_n(P);\Z)$, see~\cite[Section 4]{LitherlandCobordism}.
In order order to state and prove the metabelian satellite formula, we need to better understand $\alpha(n,\chi)_K$.
                                                                                                                  
The composition $H_1(M_K;\Z) \cong H_1(X_K;\Z) \to H_1(M_{P(K,\eta)};\Z) \cong \Z$ is given by multiplication by~$w=\ell k(\eta,P)$.
Thus, the corresponding cover of $M_K$ is disconnected and has $|w|$ components (if~$w=0$, then the covering is trivial, and so we temporarily disregard this case).
Using $t_Q$ (resp.~$t_K$) to denote the generator of the deck transformation group of the infinite cyclic cover of~$M_{P(K,\eta)}$ (resp. $M_K$), we note that~$t_Q^w=t_K$ and consider the following inclusion induced map
\begin{equation}
  \label{eq:InclusionSatelliteAlexanderModule}
  \iota_* \colon \bigoplus_{i=1}^{|w|} t_Q^{i-1} H_1(M_K;\Z[t_K^{\pm 1}]) \cong H_1(M_K;\Z[t_Q^{\pm 1}]) \to H_1(M_{P(K,\eta)};\Z[t_Q^{\pm 1}]).
\end{equation}
In order to obtain maps on the branched covers, we shall quotient both sides of~\eqref{eq:InclusionSatelliteAlexanderModule} by \(t_{Q}^{n}-1\). On the right-hand side, the result is $H_1(\Sigma_n(P(K,\eta));\Z)$ and so we focus on the left-hand side.
Write $h:=\gcd(w,n)$
and observe that \(t_{K}^{n/h}=t_Q^{wn/h}=t_Q^{\operatorname{lcm}(n,w)}=1 \pmod{t_{Q}^{n}-1}\). Thus, the map~$i_*$ of~\eqref{eq:InclusionSatelliteAlexanderModule} descends to a map 
$$ \bigoplus_{i=1}^{h} t_Q^{i-1} H_1(\Sigma_{n/h}(K);\Z) \to H_1(\Sigma_n(P(K,\eta));\Z),$$
where we are thinking of $H_{1}(\Sigma_{n/h}(K);\Z)$ as $H_1(M_K;\Z[t_K^{\pm 1}])/(t_K^{n/h}-1)$. From now on, we fix a copy of $H_1(\Sigma_{n/h}(K);\Z) $ in this direct sum once and for all. Using this copy, we obtain a map
\begin{equation}
  \label{eq:Defin}
  \iota_n \colon H_1(\Sigma_{n/h}(K);\Z) \to H_1(\Sigma_n(P(K,\eta));\Z).
\end{equation}
Since $t_K^{n/h}-1=0$ mod $t_Q^n-1$, the character $\chi$ descends to a character on each $t_Q^{i-1} H_1(\Sigma_{n/h}(K);\Z)$. Thus, for $i=1,\ldots,h$, the character \(\chi\) gives rise to the following characters:
\begin{align*}
  \chi_{i} \colon H_{1}(\Sigma_{n/h}(K);\Z) &\to \Z_m \\
  v &\mapsto \chi(t_{Q}^{i-1} \iota_n(v)).
\end{align*}
We use $\mu_Q$ (resp. $\mu_K$) to denote the meridian of~$M_{P(K,\eta)}$ (resp. $M_K$). Then, just as in Subsection~\ref{sub:Metabelian}, we consider the following composition of canonical projections:
$$q_{Q} \colon \pi_{1}(M_{P(K,\eta)})^{(1)} \to H_{1}(M_{P(K,\eta)};\Z[t_Q^{\pm1}]) \to H_{1}(\Sigma_{n}(P(K,\eta));\Z).$$
Observe that~$\mu_Q^{-w}\eta$ has trivial abelianization and therefore belongs to $\pi_{1}(M_{P(K,\eta)})^{(1)}$.
Finally, given $\omega \in S^1$ and $m \geq 0$, we write $\Bl(J)(\omega t^m)$ for the twisted Blanchfield pairing associated to the representation $\pi_1(M_J) \to GL_1(\LC), \gamma \mapsto \omega t^{m \ell k(\gamma,\mu_J)}$. When $m=1$ and $\omega=1$, this reduces to the usual Blanchfield form $\Bl(J)$.
The main result of this section is the following.

\begin{theorem}\label{thm:metabelian-cabling-formula}
  Let \(K,P\) be two knots in \(S^{3}\), let \(\eta\) be an unknotted curve in the complement of~\(P\) with meridian $\mu_\eta$, let \(w=\lk(\eta,P)\), let \(n>1\) and set \(h = \gcd(n,w)\).
  \begin{itemize}
  \item If \(w\neq 0\), then for any character \(\chi \colon H_1(\Sigma_n(P(K,\eta));\Z) \to \Z_{m}\) of prime power order, the metabelian representation~\(\alpha(n,\chi)\) is \(\eta\)-regular.
    Moreover,
    \begin{enumerate}
    \item if \(w\) is divisible by \(n\), then there exists an isometry of linking forms
      \[\Bl_{\alpha(n,\chi)}(P(K,\eta)) \cong \Bl_{\alpha(n,\chi_P)}(P) \oplus \bigoplus_{i=1}^{n} \Bl(K)(\xi_{m}^{\chi_{i}(q_Q(\mu_Q^{-w}\eta))}t^{w/n});\]
    \item if \(w\) is not divisible by \(n\), then \(\Bl_{\alpha(n,\chi)}(P(K,\eta))\)  is Witt equivalent to
      \[\Bl_{\alpha(n,\chi_{P})} (P)\oplus \bigoplus_{i=1}^{h} \Bl_{\alpha(n/h,\chi_{i})}(K)(\xi_{m}^{\chi_{i}(q_Q(\mu_Q^{-w}\eta))}t^{w/h}).\]
    \end{enumerate}
  \item If $w=0$, then the representation \(\alpha(n,\chi)\) is 
    $\eta$-regular if and only if $\chi_i(q_Q(\mu_K^{-w}\eta)) \neq 0$ for each $i=1,\ldots,n$. In this case, there exists an isometry of linking forms  
    \[\Bl_{\alpha(n,\chi)}(P(K,\eta)) \cong \Bl_{\alpha(n,\chi_P)}(P).\]
  \end{itemize}
\end{theorem}

On the level of averaged signatures, Theorem~\ref{thm:metabelian-cabling-formula} is reminiscent of Litherland's description of the behavior of the Casson-Gordon invariants of satellite knots~\cite[Theorem 2]{LitherlandCobordism}.

\begin{remark}
  \label{rem:ComparisonWithLitherland}
  Taking averaged signatures in Theorem~\ref{thm:metabelian-cabling-formula} (and applying item~\ref{item:S4} of Proposition~\ref{prop:SignatureFunctionAlgebra}),
  we see that when the winding number $w$ is non-zero, our metabelian signatures behave in the same way as the Casson-Gordon signatures do~\cite[Theorem 2]{LitherlandCobordism}. On the other hand, when $w=0$,
  the behaviors differ: namely we obtain $\sigma^{\text{av}}_{P(K,\eta),\alpha(n,\chi)}(\omega)=\sigma^{\text{av}}_{P,\alpha(n,\chi_P)}(\omega)$; the Levine-Tristram signatures of the companion knot do not contribute, contrarily do Litherland's
  formula for $\sign^{av}_{\omega} \tau(P(K,\eta),\chi)$. This can also be seen by combining Theorem~\ref{thm:BlanchfieldCG} with~\cite[Corollary~2]{LitherlandCobordism}.
  
We explain why, in the winding number zero case, it should not come as a surprise that Litherland's formula for $\tau(P(K,\eta),\chi)$ has a contribution from the companion knot $K$, while our formula for $\Bl_{\alpha(n,\chi)}(P(K,\eta))$ does not.
Litherland's result~\cite[Corollary~2]{LitherlandCobordism} implies that the difference \(f(\omega) = \sign^{av}_{\omega} \tau(P(K,\eta),\chi) - \sign^{av}_{\omega} \tau(P,\chi_P)\) is a constant function of \(\omega \in S^{1}\).
 In fact, the value of \(f\) is expressible in terms of a value of the Levine-Tristram signature of the companion knot~$K$.
    In particular, this show that $K$ does not contribute any signature jumps to \(\sign_{\omega}^{av}(\tau(P(K,\eta),\chi))\).
 Since the Witt class of $\Bl_{\alpha(n,\chi)}(P(K,\eta))$ precisely captures these signature jumps, this gives another explanation of why a contribution of~\(K\) should not be expected in our formula.
\end{remark}

Notice furthermore that Theorem~\ref{thm:metabelian-cabling-formula} takes a particularly simple form for connected sums.
Indeed, in this case, we have~$w=1$ (so~$h=1$) as well as $\eta=\mu_P$ so that we obtain the following corollary.
\begin{corollary}
  \label{cor:MetabelianConnectedSum}
  Let $K, P$ be two knots. If $\chi \colon H_1(\Sigma_n(K \# P);\Z) \to \Z_{m}$ is a character of prime power order, then  \(\Bl_{\alpha(n,\chi)}(K \# P)\) is Witt equivalent to $\Bl_{\alpha(n,\chi_P)}(P) \oplus \Bl_{\alpha(n,\chi_K)}(K)$.
\end{corollary}

The remainder of this subsection is devoted to the proof of Theorem~\ref{thm:metabelian-cabling-formula}. We start by proving the $\eta$-regularity of $\alpha(n,\chi)$.

\begin{lemma}
  \label{lemma:metabelian-eta-regularity}
  Let $K,P$ be knots and let $\eta \subset S^3 \setminus P \subset M_P$ be a simple closed curve.
  Choose a character \(\chi \colon H_1(\Sigma_n(P(K,\eta));\Z) \to \Z_{m} \) of finite order and let \(w = \lk(\eta,P)\).
  \begin{enumerate}
  \item If \(w \neq 0\), then \(\alpha(n,\chi)\) is \(\eta\)-regular.
  \item  For \(w=0\),  \(\alpha(n,\chi)\) is \(\eta\)-regular if and only if $\chi_i(q_Q(\mu_K^{-w}\eta)) \neq 0$ for each $i=1,\ldots,n$.
  \end{enumerate}
\end{lemma}
\begin{proof}
  Notice that in $M_{P(K,\eta)}$, the meridian \(\mu_{\eta}\) of $\eta$ is identified with the zero-framed longitude of \(K\). In particular, since the longitude of~\(K\) is contained in \(\pi_{1}(S^{3} \setminus \mathcal{N}(K))^{(2)}\), it follows that~\(\mu_{\eta} \in \pi_{1}(M_{P(K,\eta)})^{(2)}\).
  Therefore, since the representation $\alpha(n,\chi)$ is metabelian, we deduce that \(\alpha(n,\chi)(\mu_{\eta}) =~1\).

  It remains to show that \(\det\left(\alpha(n,\chi)(\eta) - \id\right) \neq 0\). For \(a_{1},\ldots,a_{n}$ in~$\LC\), define \(\operatorname{diag}(a_{1},\ldots,a_{n})\) to be the diagonal matrix with the $a_i$ on its main diagonal. The second assertion will quickly follow from the following observation:
  \begin{claim*}
    Set \(D = \operatorname{diag}(a_{1},a_{2},\ldots,a_{n})\) and~$A_n(t)= \bsm
    0& 1 & \cdots &0 \\
    \vdots & \vdots & \ddots & \vdots  \\
    0 & 0 & \cdots & 1 \\
    t & 0 & \cdots & 0
    \esm$. For \(k\in \Z\) and \(h = \gcd(n,k)\), the following equality holds: 
    \[\det \left(A_{n}(t)^{k} D-\id\right) =
      \begin{cases}
        (a_1-1)\cdots(a_n-1)
        \quad  &\text{if } k = 0, \\
        \pm \prod_{i=1}^{h}(1-t^{k/h} a_{i}  a_{i+h} \cdots  a_{i+k-h})
        \quad &\text{if } k \neq 0.
      \end{cases}
    \]
  \end{claim*}
  \begin{proof}
    If $k = 0$, then the claim is clear. We therefore assume that $k \neq 0$.
    If \(h\neq 1\), then a computation using the definition of $A_n(t)$ implies that there is a permutation matrix \(B_{n,k}\)~such~that
    \[B_{n,k} \left(A_{n}(t)^{k}  D-\id\right) B_{n,k}^{-1} = \bigoplus_{i=1}^{h} \left(A_{n/h}(t)^{k/h}  \operatorname{diag}(a_{i},a_{i+h},\ldots)-\id\right).\]
    We can therefore restrict our attention to the case \(h=1\).
    In this case, elementary row operations can be used to kill the entries below the diagonal.
    Eventually, we end up with a matrix with zeros below the main diagonal and the following entries on the main diagonal:
    \[\underbrace{-1,-1,\ldots,-1}_{n-1}, -1 + a_{1}  a_{2} \cdots a_{n} t^{k}.\]
    Thus, the determinant is equal to \((-1)^{n-1}(a_{1}  a_{2} \cdots a_{n}t^{k}-1)\), concluding the proof of the claim.
  \end{proof}
  Use $\phi \colon \pi_1(M_{P(K,\eta)}) \to \Z$ to denote the abelianization map. We will apply the claim to~$k=\phi(\eta)=w$ and $a_i=\chi_i(q_Q(\mu_Q^{-w}\eta))$. If $w \neq 0$, then the claim implies that \(\det\left(\alpha(n,\chi)(\eta) - \id\right) \neq 0\), proving the first assertion. If $w=0$, then the claim implies that \(\det\left(\alpha(n,\chi)(\eta) - \id\right) \neq 0\) if and only if none of the $a_i$ is equal to $1$. This concludes the proof of the second assertion and therefore the proof of the lemma.
\end{proof}

We now suppose that $\alpha(n,\chi)$ is $\eta$-regular (this is automatic for $w \neq 0$). Thus, $\alpha(n,\chi)$ restricts to representations $\alpha(n,\chi)_K$ and $\alpha(n,\chi)_P=\alpha(n,\chi_P)$ on the fundamental groups of $M_P$ and~$M_K$.
Before applying Theorem~\ref{thm:CablingTheorem} to obtain a decomposition of~$\operatorname{Bl}_{\alpha(n,\chi)}(M_{P(K,\eta)})$, we start by studying~$\alpha(n,\chi)_K$. The next remark leads to the main idea in this process.

\begin{remark}
  \label{rem:MackeyMotivation}
  Set $G:=\Z \ltimes H_1(\Sigma_n(P(K,\eta));\Z)$. Recall from Subsection~\ref{sub:Metabelian} that the metabelian representation $\alpha(n,\chi) \colon \pi_1(M_{P(K,\eta)}) \to GL_n(\LC)$ is defined as the composition of the map $\widetilde{\rho}_{Q} \colon \pi_1(M_{P(K,\eta)}) \to~G$ of~\eqref{eq:MapToSemiDirect} with the representation $\gamma_Q(n,\chi) \colon G \to GL_n(\LC)$ described in~\eqref{eq:Matrix}. Recalling~the definition of $\iota_*$ from \eqref{eq:InclusionSatelliteAlexanderModule} as well as the definition of $\iota_n$ from~\eqref{eq:Defin}, the inclusion map $ S^3 \setminus \mathcal{N}(K) \to~M_{P(K,\eta)}$ gives rise to the following commutative diagram:
  \begin{equation}
    \label{eq:DiagramForCover}
    \xymatrix@R0.5cm{
      \pi_1(S^3\setminus \mathcal{N}(K)) \ar[r]^{\iota}\ar[d] \ar@/_7pc/[ddd]_{\widetilde{\rho}_{K,Q}}& \pi_1(M_{P(K,\eta)}) \ar[d]  \ar@/^5pc/[dd]^{\widetilde{\rho}_Q} \\
      \Z \ltimes H_1(M_K;\Z[t_K^{\pm 1}]) \ar[r]^{}\ar[d]& \Z \ltimes H_1(M_{P(K,\eta)};\Z[t_Q^{\pm 1}]) \ar[d] \\
      \Z \ltimes \bigoplus_{i=1}^{h} t_Q^{i-1}H_1(\Sigma_{n/h}(K);\Z)\ar[r]^{} 
      \ar@/^0.5pc/[d]^{\operatorname{proj}_0}& \Z \ltimes H_1(\Sigma_{n}(Q);\Z)  \\
      \Z \ltimes H_1(\Sigma_{n/h}(K);\Z),\ar[ru]_{\iota_n}  \ar@/^0.5pc/[u]^{\operatorname{incl}_0}&
    }
  \end{equation}
  where $\operatorname{proj}_{0}(t_{Q}^{\ell}, (x_{1}, t_{Q}x_{2},\ldots,t_{Q}^{h-1}x_{h})) = (t_{Q}^{\ell},x_{1})$ and $ \operatorname{incl}_{0}(t_{Q}^{\ell}, x) = (t_{Q}^{\ell}, (x, 0, \ldots, 0)).$
  Since 
  \[\alpha(n,\chi)_K=\alpha(n,\chi) \circ \iota=\gamma_Q(n,\chi) \circ \widetilde{\rho}_Q \circ \iota =\gamma_Q(n,\chi) \circ \iota_n \circ \widetilde{\rho}_{K,Q},\] 
  we must study the restriction of $\gamma_Q(n,\chi)$ to $H_1:=\im(\iota_n)$. We shall denote this restriction by $\res^{G}_{H_1} \gamma_Q(n,\chi)$. Next, set $H_2:=n\Z \times H_1(\Sigma_n(Q);\Z)$ and recall from Remark~\ref{rem:InductionFunctor} that $\gamma_Q(n,\chi)$ is isomorphic to the induced representation $\ind^{G}_{H} \rho_Q(n,\chi).$ Thus, in order to understand $\alpha(n,\chi)_K$, we must study
  $$\res^{G}_{H_1} \ind^{G}_{H_2} \rho_Q(n,\chi).$$
\end{remark}

Remark~\ref{rem:MackeyMotivation} calls for an application of Mackey's induction formula. In order to recall this result, we start with a group $G$ and two subgroups \(H_{1},H_{2} \subset G\).
For any \(a \in G\), there is a \emph{double \((H_{1},H_{2})\)-coset} \(H_{1} a H_{2} := \{h_{1}  a  h_{2} \mid h_{1} \in H_{1}, h_{2} \in H_{2}\} \subset G\).
It is easy to see that \(H_{1} a H_{2} = H_{1} b H_{2}\) if and only if there are \(h_{1} \in H_{1}\) and \(h_{2} \in H_{2}\) such that \(a = h_{1}  b  h_{2}\).
Denote the set of double \((H_{1},H_{2})\)-cosets by  \(H_{1} \backslash G / H_{2}\).
If \(a \in G\) and \(\rho\) is a representation of \(H_{2}\), then we will denote by \({}^{a}\rho\) the representation of \(aH_{2}a^{-1}\) given by the formula ${}^{a}\rho(x)= \rho(a^{-1}xa).$

With these notations, Mackey's formula reads as follows; see \cite[Theorem 10.13]{CurtisReiner}.

\begin{theorem}\label{thm:mackey-thm}
  Let $G$ be a group and let \(H_{1},H_{2} \subset G\) be finite index subgroups. If \(\rho\) is a representation of \(H_{2}\), then we have
  \[\res^{G}_{H_{1}} \ind^{G}_{H_{2}}(\rho) = \bigoplus_{H_{1} a H_{2}} \ind^{H_{1}}_{aH_{2}a^{-1} \cap H_{1}} \res^{aH_{2}a^{-1}}_{aH_{2}a^{-1} \cap H_{1}} ({}^{a} \rho),\]
  where the sum is taken over all \(H_{1} a H_{2} \in H_{1} \backslash G / H_{2}\).
\end{theorem}

We will now apply Mackey's formula to the group \(G = \Z \ltimes H_{1}(\Sigma_{n}(P(K,\eta));\Z)\) and to the subgroups  $H_2 = n\Z \times H_{1}(\Sigma_{n}(P(K,\eta));\Z)$ and $H_1 = \operatorname{Im}(\iota_{n}).$ In order to carry this out, we need slight generalizations of our metabelian representations. Let $J$ be an arbitrary knot. Given a complex number \(\theta \in S^{1}\) and an integer~\(\ell~\in~\Z\), we consider the representation
\begin{align}
  \label{eq:representation-beta}
  \gamma_J(n,\chi,\theta,\ell) \colon  &\Z \ltimes H_{1}(\Sigma_{n}(K);\Z) \to GL_n(\LC)  \nonumber \\
                            &(t_Q^{a},x)  \mapsto 
                              \begin{pmatrix}
                                0              & 1      & 0      & \cdots & 0 \\
                                0              & 0      & 1      & \cdots & 0 \\
                                \vdots         & \vdots & \vdots &        & \vdots \\
                                0              & 0      & 0      & \cdots & 1 \\
                                \theta t^{\ell} & 0      & 0      & \cdots & 0
                              \end{pmatrix}^{a}
                                                                     \begin{pmatrix}
                                                                       \xi_{m}^{\chi(x)} & 0                          & 0     & \cdots & 0 \\
                                                                       0               & \xi_{m}^{\chi(t_J \cdot x)} & 0      & \cdots & 0 \\
                                                                       \vdots          & \vdots                     & \vdots &        & \vdots \\
                                                                       0               & 0                          & 0      & 0      & \xi_{m}^{\chi(t_J^{n-1}\cdot x)} 
                                                                     \end{pmatrix}.
\end{align}
By definition of this representation, we have the equality $\gamma_J(n,\chi,1,1) = \gamma_J(n,\chi)$. Next, we generalize the representation $\rho_J(n,\chi)$ by considering the representation
\begin{align}
  \rho_J(n,\chi,\theta,\ell) \colon  n\Z \times H_1(\Sigma_n(J);\Z) &\to \LC \\
  (t_J^{nk},x) &\mapsto  (\theta t^{\ell})^{k} \cdot \xi_{m}^{\chi(x)}.    \nonumber
\end{align}
Just as for the representation $\gamma_J(n,\chi,\theta,\ell)$, we observe that $\rho_J(n,\chi,1,1)=\rho_J(n,\chi)$. Furthermore, generalizing Remark~\ref{rem:InductionFunctor}, we also have the induction formula 
$$\ind_H^G \rho_J(n,\chi,\theta,\ell)=\gamma_J(n,\chi,\theta,\ell).$$ 
The following lemma gathers some results for our application of Mackey's formula. 


\begin{lemma}
  \label{lem:DoubleCosets}
  The following assertions hold:
  \begin{enumerate}
  \item Every double \((H_1,H_2)\)-coset is equal to~\(H_1(t_{Q}^{i},0)H_2\), for some \(1 \leq i \leq h\).
  \item $H_2$ is a normal subgroup of $G$.
  \item For any $(t_Q^k,x) \in G$ and any  \(1 \leq i \leq h\), we have $${}^{t_{Q}^{i}} \rho_{Q}(n,\chi,\theta,\ell)(t_Q^k,x) = \rho_{Q}(n,\chi_{i},\theta,\ell)(t_Q^k,x).$$
  \item For any $(t_K^{n/h},x) \in H_1 \cap H_2$, we have 
    $$\res_{H_2 \cap H_1}^{H_2} \rho_{Q}(n,\chi,\theta,\ell)(t_K^{n/h},x) = \rho_{K}(n/h,\chi,\theta^{w/h} \xi_{m}^{\chi(q_Q(\mu_K^{-w}\eta))},\ell w/h)(t_K^{n/h},x).$$ 
  \end{enumerate}
\end{lemma}
\begin{proof}
  Given \(g = (t_{Q}^d,x) \) in $G$, we study the $(H_1,H_2)$-coset $H_1gH_2$. We write $h=\operatorname{gcd}(n,w)$ as \(h=an + bw \) as well as \(d = hs + r\), for some \(0 \leq r \leq h-1\). Furthermore, for \(x \in H_{1}(\Sigma_{n}(P(K,\eta));\Z)\), we set \(\nu_{p}(t_Q) = t_{Q}^{-p+1}+t_{Q}^{-p+2}+\cdots+t_{Q}^{-1}+1\). A computation now shows that
  \[(t_{Q}^{d},x) = (t_{Q}^{w},q_Q(\mu_K^{-w}\eta))^{bs} (t_{Q}^{r},0) (t_{Q}^{asn}, x-t_Q^{-r-asn}\nu_{wbs}(t_Q)q_Q(\mu_K^{-w}\eta)).\]
  To show that the right-hand side belongs to~\(H_1(t_{Q}^r,0)H_2\), we only need to prove that $(t_{Q}^{w},q_Q(\mu_K^{-w}\eta))$ belongs to $H_1=\im(i_n)$. This follows from the diagram in~\eqref{eq:DiagramForCover} which implies that
  $$ (t_{Q}^{w},q_Q(\mu_K^{-w}\eta))=\widetilde{\rho}_Q(\mu_K^{w} (\mu_K^{-w}\eta)=\widetilde{\rho}_Q(\iota (\mu_K))=\iota_n(\widetilde{\rho}_{K,Q}(\mu_K))=\iota_n(t_K,0). $$
  The second assertion is a consequence of the definition of the group law on $G$. We now prove the third assertion. Given $(t_Q^k,x) \in G$ and $0 \leq i \leq h-1$, we apply successively the definition of $ {}^{a} \rho$, the group law in $G$ and the definition of $\chi_i$ to obtain the desired equality:
  \begin{align*}
    {}^{t_{Q}^{i}} \rho_{Q}(n,\chi,\theta,\ell)(t_{Q}^{k},x)
    &= \rho_{Q}(n,\chi,\theta,\ell)((t_{Q}^{-i},0) (t_{Q}^{k},x)(t_{Q}^{i},0))  \\
    &= \rho_{Q}(n,\chi,\theta,\ell)(t_{Q}^{k},t_{Q}^{-i}\cdot x) \\
    &=\rho_{Q}(n,\chi_{i},\theta,\ell)(t_{Q}^{k},x) .
  \end{align*}
  To prove the fourth assertion, we first note that $H_2 \cap H_1=\im \left(\iota_{n}|_{(n/h)\Z \times H_{1}(\Sigma_{n/h}(K);\Z)}\right)$. Next, using the relation \(\iota_{n}(t_{K},0) = (t_{Q}^{w},q_Q(\mu_K^{-w}\eta))\) and the definition of $\rho(n,\chi,\theta,\ell)$, we obtain
  \begin{align*}
    \left(\res_{H_2 \cap H_1}^{H_2} \rho_{Q}(n,\chi,\theta,w)\right)(t_{K}^{n/h},0) 
    &= \rho_{Q}(n,\chi,\theta,\ell)(t_{Q}^{wn/h},q_Q(\mu_K^{-w}\eta))  \\
    &= (\theta t^{\ell})^{w/h} \xi_{m}^{\chi(q_Q(\mu_K^{-w}\eta))} \\
    &=\rho_{K}(n/h,\chi,\theta^{w/h} \xi_{m}^{\chi(q_Q(\mu_K^{-w}\eta))},\ell w/h))(t_{K}^{n/h},0) .
  \end{align*}
  To get the equality for arbitrary elements $(t_{K}^{n/h},x) \in H_1 \cap H_2$, note that the value of both homomorphisms on $(1,x)$ is $\xi_m^{\chi(\iota_n(x))}.$ This concludes the proof of the lemma. 
\end{proof}

Lemma~\ref{lem:DoubleCosets} is now used to apply Mackey's theorem to our setting.

\begin{proposition}
  \label{prop:MackeyApplication}
  Use $\gamma_Q(n,\chi)_{K}$ to denote the restriction of $\gamma_Q(n,\chi)$ to $\Z \ltimes H_1(\Sigma_{n/h}(K),\Z)$. There exists an isomorphism of representations 
  \[\gamma_{Q}(n,\chi)_{K} \cong \bigoplus_{i=1}^{h} \gamma_{K}(n/h,\chi_{i},\xi_{m}^{\chi_{i}(q_Q(\mu_K^{-w}\eta))},w/h).\]
\end{proposition}
\begin{proof}
  Since we saw in Remark~\ref{rem:MackeyMotivation} that $\gamma(n,\chi)_{K} = \res^{G}_{H_1} \ind^{G}_{H_2} \rho_{Q}(n,\chi),$ our goal is to apply Mackey's formula to $\rho_Q(n,\chi)=\rho_Q(n,\chi,1,1)$. For the remainder of the proof, we write $\rho$ as a shorthand for $\rho_Q(n,\chi,\theta,\chi)$.
  Using consecutively Theorem~\ref{thm:mackey-thm} (as well as the first item of Lemma~\ref{lem:DoubleCosets}), and then the second, third and fourth items of  Lemma~\ref{lem:DoubleCosets}, we obtain
  \begin{align*}
    \res^{G}_{H_{1}} \ind^{G}_{H_{2}}(\rho)
    &= \bigoplus_{i=1}^{h} \ind^{H_1}_{H_2 \cap H_1} \res^{H_2}_{H_2 \cap H_{1}} ({}^{t_Q^i} \rho)   \\
    &= \bigoplus_{i=1}^{h} \ind^{H_1}_{H_2 \cap H_1} \res^{H_2}_{H_2 \cap H_{1}} (\rho_{Q}(n,\chi_{i},\theta,\ell))   \\
    &= \bigoplus_{i=1}^{h}  \ind^{H_1}_{H_2 \cap H_1}\rho_{K}(n/h,\chi_i,\theta^{w/h} \xi_{m}^{\chi_i(q_Q(\mu_K^{-w}\eta))},\ell w/h)).
  \end{align*}
  Since $H_2 \cap H_1=\im \left(\iota_{n}|_{(n/h)\Z \times H_{1}(\Sigma_{n/h}(K);\Z)}\right)$, observe that   $\ind_{H_1 \cap H_2}^G \rho_K(n,\chi,\theta,\ell)=\gamma_K(n,\chi,\theta,\ell)$. Additionally, recall that $\gamma_Q(n,\chi)_{K} = \res^{G}_{H_1} \ind^{G}_{H_2} \rho_{Q}(n,\chi,1,1)$ and $\gamma_Q(n,\chi)=\gamma_Q(n,\chi,1,1)$. Therefore, taking $\theta=1$ and $\ell=1$ in the previous computation, we obtain
  \begin{align*}
    \gamma_Q(n,\chi)_{K} 
    =\res^{G}_{H_{1}} \ind^{G}_{H_{2}}(\rho_Q(n,\chi))
    =\bigoplus_{i=1}^{h} \gamma_{K}(n/h,\chi_{i},\xi_{m}^{\chi_{i}(q_Q(\mu_K^{-w}\eta))},w/h).
  \end{align*}
  This concludes the proof of the proposition.
\end{proof}


Next, we return to Blanchfield pairings. Namely, we discuss the effect of the representation~$\gamma_Q(n,\chi,\theta,\ell)$ on the variable $t$ of the metabelian Blanchfield pairing.

\begin{remark}
  \label{rem:BlanchfieldChangeVariable}
  Recall from Subsection~\ref{sub:Metabelian} that for any knot $J$, the representation $\alpha_J(n,\chi)$ is obtained by precomposing the representation $\gamma_J(n,\chi) \colon \Z \ltimes H_1(\Sigma_n(J);\Z) \to GL_n(\LC)$ with the map $\widetilde{\rho}_J \colon \pi_1(M_J) \to  \Z \ltimes H_1(\Sigma_n(J);\Z)$. We shall adopt the same convention for $\gamma_J(n,\chi,\theta,\ell)$, thus obtaining a representation $\alpha_J(n,\chi,\theta,\ell)$ of $\pi_1(M_J)$. Since $\gamma_J(n,\chi,\theta,\ell)$ can be obtained from~\(\gamma_J(n,\chi)\) via the substitution~\(t \mapsto \theta t^{\ell}\), it follows that $\Bl_{\alpha(n,\chi,\theta,\ell)}(J)(t)=\Bl_{\alpha(n,\chi)}(J)(\theta t^\ell)$.
\end{remark}

Returning to satellite knots, recall from Remark~\ref{rem:MackeyMotivation} that the Blanchfield pairing~$\Bl_{\alpha(n,\chi)_K}(K)$ is obtained by precomposing the representation $\gamma_Q(n,\chi)_K \colon \Z \ltimes H_1(\Sigma_{n/h}(K);\Z) \to GL_n(\LC)$ with the map $\widetilde{\rho}_{K,Q} \colon \pi_1(M_K) \to  \Z \ltimes H_1(\Sigma_{n/h}(K);\Z)$. Thus, applying Proposition~\ref{prop:MackeyApplication} and Remark~\ref{rem:BlanchfieldChangeVariable}, we obtain the following isometry of linking forms:
\begin{equation}
  \label{eq:MackeyBlanchfield}
  \Bl_{\alpha(n,\chi)_K}(K)(t)
  \cong  \bigoplus_{i=1}^{h} \Bl_{\alpha_K(n/h,\chi)}(K)(\chi_{i}(q_Q(\mu_K^{-w}\eta))t^{w/h}).
\end{equation}

Using this isometry,  we can now prove Theorem~\ref{thm:metabelian-cabling-formula}.

\begin{proof}[Proof of Theorem~\ref{thm:metabelian-cabling-formula}]
  First of all, note that Lemma~\ref{lemma:metabelian-eta-regularity} ensures that $\alpha(n,\chi)$ is $\eta$-regular. 
  We~can therefore apply the satellite formula of Theorem~\ref{thm:CablingTheorem} to each of the cases which we shall now~distinguish.

  We first assume that $w \neq~0$. As $\alpha(n,\chi)$ is $\eta$-regular, it restricts to representations $\alpha(n,\chi)_K$ and $\alpha(n,\chi)_P$ on $\pi_1(M_K)$ and~$\pi_1(M_P)$. We consider two cases.

  \emph{Case 1.}  If \(w\) is divisible by \(n\), then $\alpha(n,\chi)_K$ is abelian and is isomorphic to $\bigoplus_{i=1}^n \xi_m^{\chi_i} \otimes~\phi_K^{w/n}$, where $\xi_m^{\chi_i}$ is understood as the character mapping the meridian $\mu_K$ to $\xi_m^{\chi_i(q_Q(\mu_Q^{-w}\eta))}$. As a consequence,~\eqref{eq:MackeyBlanchfield} and Corollary~\ref{cor:cabling-formula-abelian} provide the desired isometry of linking forms:
  \[\Bl_{\alpha(n,\chi)}(P(K,\eta)) \cong \Bl_{\alpha(n,\chi_P)}(P) \oplus \bigoplus_{i=1}^{n} \Bl(K)(\xi_m^{\chi_{i}(q_Q(\mu_K^{-w}\eta))}t^{w/n}).\]
  
  \emph{Case 2.} If \(w\) is not divisible by \(n\), then~\eqref{eq:MackeyBlanchfield} and Theorem~\ref{thm:CablingTheorem} imply that 
  \(\Bl_{\alpha(n,\chi)}(P(K,\eta))\)
  is Witt equivalent to
  \[\Bl_{\alpha(n,\chi_{P})}(P) \oplus \bigoplus_{i=1}^{h} \Bl_{\alpha_K(n/h,\chi_{i})}(K)(\xi_m^{\chi_{i}(q_Q(\mu_K^{-w}\eta))}t^{w/h}).\]
  Next, we assume that \(w=0\).
  As the abelianization map $\phi \colon H_1(M_{P(K,\eta)};\Z) \to \Z$ restricts to the zero map on $H_1(M_K;\Z)$, it follows that $\alpha(n,\chi)_K$ is abelian.
  More precisely, \(\alpha(n,\chi)_{K} \cong \bigoplus_{i=1}^{n} \xi_m^{\chi_i}\), where $\xi_m^{\chi_i}$
  is understood as the character mapping $\mu_K$ to $\xi_m^{\chi_i(q_Q(\eta))}$.
  As a consequence, we obtain the isomorphism
  $$H_{\ast}(M_{K};\LC^{n}_{\alpha(n,\chi)}) \cong \bigoplus_{i=1}^n H_{1}(M_{K};\LC_{\xi_m^{\chi_i}})\cong  \bigoplus_{i=1}^n H_{1}(M_{K};\C_{\xi_m^{\chi_i}}) \otimes_\C\LC. $$
  The dimension of $H_{1}(M_{K};\C_{\xi_m^{\chi_i}}) $ is equal to $\eta_K(\xi_m^{\chi_i(q_Q(\eta))})$, the nullity of $K$ evaluated at the unit complex number $\xi_m^{\chi_i(q_Q(\eta))}=\xi_m^{\chi(t_Q^{i-1}q_Q(\eta))}$.
  This value is non-zero if and only if the Alexander polynomial satisfies $\Delta_K(\xi_m^{\chi(t_Q^{i-1}q_Q(\eta))})=0$.
  Since the character $\chi$ has prime power order; this can never happen~\cite[proof of Proposition 3.3]{FriedlEta} and we therefore deduce that  $H_{\ast}(M_{K};\LC^{n}_{\alpha(n,\chi)})=~0.$ Applying Corollary~\ref{cor:cabling-formula-abelian} now provides the desired isometry :
  $$\Bl_{\alpha(n,\chi)}(P(K,\eta)) \cong \Bl_{\alpha(n,\chi_{P})}(P).$$
  This concludes the proof of the theorem.
\end{proof}

\appendix
\section{Shapiro's lemma and twisted Blanchfield forms}
\label{sec:Appendix}

The goal of this appendix is to prove Corollary~\ref{cor:BlanchfieldDownstairs} which relates the Blanchfield form of the cover $\Bl_{\alpha \times \chi}(M_n)$ to the metabelian Blanchfield form~$\Bl_{\alpha(n,\chi)}(K)$.
Since the proof is somewhat technical, we will first need to collect a certain number of preliminaries.
Namely in Subsections~\ref{sub:BocksteinAppendix}, ~\ref{sub:EvaluationAppendix} and~\ref{sub:InducedCoinduced}, we respectively recall some facts about Bockstein homomorphisms, evaluation maps in twisted (co)homology,  and (co)induced modules.
Finally,  Subsection~\ref{sub:BlanchfieldAppendix} proves Proposition~\ref{prop:Shapiros-lemma-bl-forms} which is a generalisation of Corollary~\ref{cor:BlanchfieldDownstairs}.

\subsection{Twisted homology: functoriality and Bocksteins.}
\label{sub:BocksteinAppendix}

First, we briefly recall the sense in which twisted (co)homology is functorial. 
We then prove a lemma about the functoriality of Bockstein homomorphisms in twisted homology.

\begin{construction}
\label{cons:FunctorialityTwisted}
Fix a ring $R$,  path-connected CW complexes $X,Y$,  a $(R,\Z[\pi_1(X)])$-bimodule~$M$ and an $(R,\Z[\pi_1(Y)])$-bimodule~$N$. 
\begin{itemize}
\item Given a map $f \colon X \to Y$,  endow $N$ with the right $\Z[\pi_1(X)]$-module structure obtained by pulling back the right $\Z[\pi_1(Y)]$-module structure via the induced map \(f_{\ast} \colon \pi_{1}(X) \to \pi_{1}(Y)\).
Now $f$ and a $(R,\Z[\pi_1(X)])$-bimodule map $\alpha \colon M \to N$ induce an $R$-linear map
$$ (f,\alpha)_* \colon H_*(X;M) \to H_*(Y;N).$$ 
To see this one verifies that $\alpha$ and $f$ induce well defined chain maps
\[M \otimes_{\Z[\pi_{1}](X)} C_{\ast}(\widetilde{X}) \xrightarrow{\alpha \otimes \op{id}} N \otimes_{\Z[\pi_{1}(X)]} C_{\ast}(\widetilde{X}) \xrightarrow{\op{id} \otimes \widetilde{f}} N \otimes_{\Z[\pi_{1}(Y)]} C_{\ast}(\widetilde{Y}),\]
where \(\widetilde{f} \colon \widetilde{X} \to \widetilde{Y}\) is the lift of \(f\) to universal covers.
\item 
In a similar manner, if \(\alpha \colon N \to M\) is a homomorphism of \((R,\Z[\pi_{1}(X)])\)-bimodules, we obtain an \(R\)-linear map on cohomology
\[(f,\alpha)^{\ast} \colon H^{\ast}(Y;N) \to H^{\ast}(X,M)\]
induced by the composition of cochain maps
\begin{align*}
 \Hom_{\text{right-}\Z[\pi_1(Y)]}(C_*(\widetilde{Y})^\#,N) & \xrightarrow{\hom(\widetilde{f},\op{id}_{N})}  \Hom_{\text{right-}\Z[\pi_1(X)]}(C_*(\widetilde{X})^\#,N) \\ & \xrightarrow{\hom(\op{id},\alpha)} \Hom_{\text{right-}\Z[\pi_1(X)]}(C_*(\widetilde{X})^\#,M).
 \end{align*}
\end{itemize}
\end{construction}

The definition of the Blanchfield form involves a Bockstein homomorphism.
It should therefore not surprise the reader that to relate the Blanchfield pairing of the cover to the metabelian Blanchfield form, we  need to know how Bocksteins interact with induced maps.
To this effect,  recall that if we are given an exact sequence
\[0 \to M_{1} \to M_{2} \to M_{3} \to 0\]
 of \((R,\Z[\pi_{1}(X)])\)-bimodules, then, for any \(k \geq 0\), there are associated Bockstein homomorphisms
\[BS \colon H_{k+1}(X;M_{3}) \to H_{k}(X;M_{1}) \quad \text{and} \quad BS \colon H^{k}(X;M_{3}) \to H^{k+1}(X;M_{1}).\]
The following lemma shows that these Bockstein homomorphisms are natural.
\begin{lemma}\label{lemma:Bockstein-naturality}
  Let \(f \colon X \to Y\) be a 
   map of path-connected CW complexes.
  Suppose that we are given a commutative diagram
  \begin{equation}
    \label{eq:Bockstein-naturality}
    \begin{tikzcd}
      0 \ar[r] & M_{1} \ar[r] & M_{2} \ar[r] & M_{3} \ar[r] & 0 \\
      0 \ar[r] & N_{1} \ar[r] \ar[u, "{\alpha_{1}}"] & N_{2} \ar[r] \ar[u, "{\alpha_2}"] & N_{3} \ar[r] \ar[u, "\alpha_{3}"] & 0
    \end{tikzcd}
  \end{equation}
  where the first row is an exact sequence of \((R,\Z[\pi_{1}(X)])\)-bimodules, the second row is an exact sequence of \((R,\Z[\pi_{1}(Y)])\)-bimodules and the maps \(\alpha_{i}\), for \(i=1,2,3\), are homomorphisms of
  \((R,\Z[\pi_{1}(X)])\)-bimodules.
Recall that \(N_{i}\), for \(i=1,2,3\), is equipped with a right action of \(\Z[\pi_{1}(X)]\) by pulling back the right action of \(\Z[\pi_{1}(Y)]\) via the induced map \(f_{\ast} \colon \pi_{1}(X) \to \pi_{1}(Y)\).

  For any \(k \geq 0\), the following diagram is commutative
  \begin{center}
    \begin{tikzcd}
      H^{k}(X;M_{3}) \ar[r, "{BS}"] & H^{k+1}(X;M_{1}) \\
      H^{k}(Y;N_{3}) \ar[r, "{BS}"] \ar[u, "{(f,\alpha_3)^{\ast}}"] & H^{k+1}(Y,N_{1}), \ar[u, "{(f,\alpha_1)^{\ast}}"]
    \end{tikzcd}
  \end{center}
  where the maps denoted by \(BS\) are the respective Bockstein homomorphisms associated to the exact sequences in~\eqref{eq:Bockstein-naturality}.
\end{lemma}
\begin{proof}
  The reader can readily verify that the following diagram of chain complexes of left $R$-modules is commutative:
  \begin{center}
    \begin{tikzcd}
      0 \ar[r] & C^{\ast}(X;M_{1}) \ar[r] & C^{\ast}(X,M_{2}) \ar[r] & C^{\ast}(X;M_3) \ar[r] & 0 \\
      0 \ar[r] & C^{\ast}(Y;N_{1}) \ar[r] \ar[u, "{(f,\alpha_1)}"] & C^{\ast}(Y,N_{2}) \ar[r] \ar[u, "{(f,\alpha_2)}"] & C^{\ast}(Y;N_3) \ar[r] \ar[u, "{(f,\alpha_3)}"] & 0.
    \end{tikzcd}
  \end{center}
  The lemma now follows by taking induced maps on (co)homology.
\end{proof}
  
\subsection{The evaluation map: detailed definition and naturality}
\label{sub:EvaluationAppendix}

We recall some details regarding the evaluation map $\ev$ on twisted cohomology that was mentioned in Subsection~\ref{sub:TwistedHomology}.
We then prove that this evaluation map is natural.
\medbreak

Fix a CW pair $(X,A)$ and a ring $R$ that is endowed with an involution $x \mapsto x^\#$.
Let $M$ and $M'$ be $(R,\Z[\pi_1(X)])$-bimodules and let $S$ be an $(R,R)$-bimodule.
Furthermore, let $\langle -,-\rangle \colon M \times M' \to S$ be a non-singular
$\pi_1(X)$-invariant sesquilinear pairing.
\begin{construction}
\label{cons:Evaluation}
We construct the evaluation $\ev \colon H^i(X,A;M) \to \makeithash{\Hom_{\text{left-}R}(H_{i}(X,A;M'),S)}$.

The evaluation map can be defined in two steps. 
  Firstly,  the non-singularity of the pairing~$\langle-,-\rangle$ implies that the following map is an isomorphism of cochain complexes of left $R$-modules:
  \begin{align*}
    \kappa \colon \Hom_{\text{right-}\Z[\pi_{1}(X)]}(\makeithash{C_{\ast}(X,A)},M) &\to \makeithash{\Hom_{\text{left-}R}(M' \otimes C_{\ast}(X,A),S)}, \\
    f &\mapsto \left((m' \otimes \sigma) \mapsto \langle m',f(\sigma) \rangle \right).
  \end{align*}
  Secondly,  evaluating cocycles on homology classes yields a homomorphism of left $R$-modules
  \[E \colon H^{i}(\makeithash{\Hom_{\text{left-}R}(M' \otimes C_{\ast}(X,A),S)}) \xrightarrow{} \makeithash{\Hom_{\text{left-}R}(H_{i}(X,A;M'),S)}.\]
  Finally,  the evaluation map $\ev$, which is also left $R$-linear,  is obtained by composing $\kappa$ with $E$:
  \[ \ev \colon H^{i}(X,A;M) \xrightarrow{\kappa,\cong} H^{i}(\makeithash{\Hom_{\text{left-}R}(M' \otimes C_{\ast}(X,A),S)}) \xrightarrow{E} \makeithash{\Hom_{\text{left-}R}(H_{i}(X,A;M'),S)}.\]
Note that $E$ is the edge homomorphism in the universal coefficient spectral squence.
\end{construction}

We now prove a naturality statement for this evaluation map.
While this result also holds for pairs, we only state it for absolute homology since this is all we require in the sequel.

\begin{lemma}\label{lemma:eval-map-naturality}
  Let \(X,Y\) be path-connected CW-complexes.
  Let \(M\), \(M'\) be \((R,\Z[\pi_{1}(X)])\)-bimodules, let  \(N\), \(N'\) be
\((R,\Z[\pi_{1}(Y)])\)-bimodules and let $S,T$ be $(R,R)$-bimodules.
    Suppose also that we are given non-singular $\pi_1(X)$-equivariant (resp. $\pi_1(Y)$-equivariant) sesquilinear pairings
    \[\langle -,- \rangle_{X} \colon M' \times M \to S, \quad \langle -,- \rangle_{Y} \colon N' \times N \to T.\]
Recall that we can equip \(N\) and \(N'\) with a right action of \(\Z[\pi_{1}(X)]\) by pulling back the right \(\Z[\pi_{1}(Y)]\)-action by the induced homomorphism \(f_{\ast} \colon \pi_{1}(X) \to \pi_{1}(Y)\).
If there are \((R,\Z[\pi_{1}(X)])\)-linear maps $\alpha \colon N \to M$ and $ \alpha' \colon M' \to N' $
    as well as an \((R,R)\)-linear map \(\beta \colon T \to S\) such that for all \(m' \in M'\) and \(n \in N\)
    \[\langle m', \alpha(n)  \rangle_{X} = \beta(\langle\alpha'(m'),n  \rangle_{Y}),\]
    then, the following diagram of left $R$-modules commutes:
  \begin{center}
    \begin{tikzcd}
      H^{k}(X;M) \ar[r, "ev"] &\makeithash{ \Hom_{\text{left-}R}(H_{k}(X;M'),S)} \\
      H^{k}(Y;N) \ar[r, "ev"] \ar[u, "{(f,\alpha)^{\ast}}"] & \makeithash{\Hom_{\text{left-}R}(H_{k}(Y,N'),T)}. \ar[u, "{\Hom((f,\alpha')_{\ast},\beta)}"]
    \end{tikzcd}
  \end{center}
\end{lemma}
\begin{proof}
Naturality of the universal coefficient spectral sequence implies that it is sufficient to prove that the following diagram of cochain complexes of left $R$-modules is commutative:
  \begin{center}
    \begin{tikzcd}
      \Hom_{\text{right-}\Z[\pi_1(X)]}(\makeithash{C_{\ast}(\widetilde{X})},M) \ar[r, "\kappa"] &      \makeithash{ \Hom_{\text{left-}R}(M' \otimes_{\Z[\pi_1(X)]} C_{\ast}(\widetilde{X}), S)} \\
\Hom_{\text{right-}\Z[\pi_1(Y)]}(\makeithash{C_{\ast}(\widetilde{Y})},N) \ar[r, "\kappa"] \ar[u, "{(f,\alpha)}"] & \makeithash{\Hom_{\text{left-}R}(N' \otimes_{\Z[\pi_1(Y)]} C_{\ast}(\widetilde{Y}),T)}. \ar[u, "{\Hom((f,\alpha'),\beta)}"]
    \end{tikzcd}
  \end{center}
Starting with a \(\phi \in \Hom_{\text{right-}\Z[\pi_1(Y)]}(\makeithash{C_{\ast}(\widetilde{Y})},N)\) in the lower left corner of this diagram and going up,  we have \((f,\alpha)(\phi) = \alpha \circ \phi \circ f\).
    Similarly,  if we start with \(\phi \in \makeithash{\Hom_{\text{left-}R}(N' \otimes_{\Z[\pi_1(Y)]} C_{\ast}(\widetilde{Y}),T)}\) in the lower right corner and go up,  then we obtain
    \[\Hom((f,\alpha'),\beta)(\phi) = \beta \circ \phi \circ (f \otimes \alpha').\]
 With these observations at hand,  starting with \(\phi \in \Hom_{\text{right-}\Z[\pi_1(Y)]}(\makeithash{C_{\ast}(\widetilde{Y})},N)\) in the lower left corner, the up-right and right-up paths of the diagram respectively yield
    \begin{align*}
      (\kappa \circ (f,\alpha))(\phi)(m' \otimes \sigma) &= \langle m', (\alpha \circ \phi \circ f)(\sigma) \rangle_{X}, \\
      (\Hom((f,\alpha'),\beta) \circ \kappa)(\phi)(m' \otimes \sigma) &= \beta \left( \langle \alpha'(m'), (\phi \circ f)(\sigma)  \rangle_{Y} \right),
    \end{align*}
for all \(m' \in M'\) and all \(\sigma \in C_{\ast}(\widetilde{X})\).
Using the hypothesis of the lemma, both quantities are equal, and hence the required commutativity follows.
This concludes the proof of of the lemma.
\end{proof}

\subsection{Induced and coinduced representations}
\label{sub:InducedCoinduced}

We now need to introduce one last ingredient before delving into the proof of Corollary~\ref{cor:BlanchfieldDownstairs}.
Namely, we recall some basic facts about the induction and coinduction functors as these will play a role in expliciting the isometry between the Blanchfield form in the cover and the metabelian Blanchfield form.
As in the previous sections we fix a ring \(R\) with involution.
  
 \begin{definition}
 \label{def:Induced}
  Let \(G\) be a group and let \(H\) be a subgroup of \(G\).
  For any \((R,\Z[H])\)-bimodule~\(M\),  the \emph{induced \((R,\Z[G])\)-bimodule} refers to the $(R,\Z[G])$-bimodule
  \[\op{Ind}^{G}_{H}(M) := M \otimes_{\Z[H]} \Z[G].\]
  Similarly,  the \emph{coinduced \((R,\Z[G])\)-bimodule} refers to
  \[\op{Coind}^{G}_{H}(M) := \Hom_{\text{right-}\Z[H]}(\Z[G],M).\]
  The structure of a \((R,\Z[G])\)-bimodule on \(\op{Coind}^{G}_{H}(M)\) is given as follows:
for \(r \in R\), \(f \in \op{Coind}^{G}_{H}(M)\), and \(x,y \in \Z[G]\),  
  \begin{equation}
    \label{eq:action-on-coind}
    (r \cdot f \cdot x)(y) := r \cdot f(xy).
  \end{equation}

\end{definition}

\begin{example}\label{ex:induced-representation}
Consider a group \(G\) and a subgroup \(H \leq G\) of finite index \(n\).
Given a commutative ring $R$,
we use induction to explain how a representation $\beta \colon H \to GL_d(R)$ induces a representation 
$$\alpha:=\op{Ind}^{G}_{H}(\beta) \colon G \to GL_{dn}(R).$$
First some set-up.
Let \(g_{1}, g_{2}, \ldots, g_{n}\) be representatives of cosets \(H \backslash G\).
  As a left \(\Z[H]\)-module, \(\Z[G]\) is a free module of rank \(n\) and admits a decomposition
  \begin{equation}
    \label{eq:decomposition-ZG-as-ZH-mod}
    \Z[G] = \bigoplus_{i=1}^{n} \Z[Hg_{i}],
  \end{equation}
  hence, the set \( \lbrace 1_{G}=g_{1},g_{2},\ldots,g_{n} \rbrace  \) is a \(\Z[H]\)-basis of \(\Z[G]\) treated as a left \(\Z[H]\)-module.
  Observe that the above decomposition does not depend on the particular choice of coset representatives.

Next we describe the representation $\op{Ind}^{G}_{H}(\beta)$.
  The representation \(\beta\) equips the free \(R\)-module~\(R^{d}\) with the right action from \(H\), hence it becomes a \((R,\Z[H])\)-bimodule.
  Using~\eqref{eq:decomposition-ZG-as-ZH-mod},  the induced \((R,\Z[G])\)-bimodule \(\op{Ind}^{G}_{H}(R^{d}_{\beta})\) can be decomposed as
\begin{equation}
\label{eq:DirectSumInduced}
\op{Ind}^{G}_{H}(R^{d}_{\beta}) = R^{d}_{\beta} \otimes_{\Z[H]} \Z[G] = \bigoplus_{i=1}^{n} R^{d}_{\beta} \otimes_{\Z[H]} \Z[Hg_{i}].
  \end{equation}
  In order to shorten the notation,  we  set \(R^{d}_{\beta} \otimes g_{i} := R^{d}_{\beta} \otimes_{\Z[H]} \Z[Hg_{i}]\).
Note that as a left \(R\)-module, \(\op{Ind}^{G}_{H}(R^{d}_{\beta})\) is free of rank \(nd\).
  Since \(\op{Ind}^{G}_{H}(R^{d}_{\beta}) \cong R^{dn}\) is a right $\Z[G]$-module,  we therefore obtain the desired \emph{induced representation} $\alpha := \op{Ind}^{G}_{H}(\beta) \colon G \to GL_{dn}(R).$
\end{example}

\begin{construction}
\label{cons:Phibar}
Given a group \(G\),  a finite index subgroup \(H \leq G\) and a \((R,\Z[H])\)-bimodule~\(M\),  we construct a $(R,\Z[G])$-linear map $\overline{\phi} \colon \op{Ind}^{G}_{H}(M) \to \op{Coind}^{G}_{H}(M)$.

First, we consider the following homomorphism of \((R,\Z[H])\)-bimodules
  \[\phi \colon M \to \op{Coind}^{G}_{H}(M), \quad \phi(m)(g) =
    \begin{cases}
      m \cdot g, & g \in H, \\
      0, & g \not\in H,
    \end{cases}
  \]
where \(m \in M\) and \(g \in G\).
  We can uniquely extend \(\phi\) to a homomorphism of \((R,\Z[G])\)-bimodules 
  \begin{equation}
    \label{eq:nat-isomorphism-ind-coind}
    \overline{\phi} \colon \op{Ind}^{G}_{H}(M) \to \op{Coind}^{G}_{H}(M), \quad \overline{\phi}(m \otimes g)(g') = (\phi(m) \cdot g)(g') = \phi(m)(g \cdot g').
  \end{equation}
The reader can check that \(\overline{\phi}\) is natural in \(M\), i.e., it is a component of a natural transformation of the induction and coinduction functors.
\end{construction}

The map \(\overline{\phi}\) is known to be an isomorphism of \((R,\Z[G])\)-bimodules~\cite[Proposition III.5.9]{brownCohomologyGroups1982},  but we describe the inverse explicitly when $G$ is a finite group as we will need it in the
sequel.

\begin{lemma}
\label{lemma:ind-coind-isom}
Given a group \(G\),  a finite index subgroup \(H \leq G\) and a \((R,\Z[H])\)-bimodule \(M\),  the $(R,\Z[G])$-linear map $\overline{\phi}$ from Construction~\ref{cons:Phibar} defines a natural isomorphism of $(R,\Z[G])$-bimodules:
  \[ \overline{\phi} \colon \op{Ind}^{G}_{H}(M) \xrightarrow{\cong} \op{Coind}^{G}_{H}(M).\]
\end{lemma}
\begin{proof}
The reader can check that the inverse of \(\overline{\phi}\) is given by the formula
  \begin{equation}
    \label{eq:inv-nat-isomorphism-ind-coind}
    \overline{\psi} \colon \op{Coind}^{G}_{H}(M) \to \op{Ind}^{G}_{H}(M), \quad \overline{\psi}(f) = \sum_{g \in H \backslash G} f(g^{-1}) \otimes g.
  \end{equation}
  Here, the notation \(g \in H \backslash G\) means that we are summing over coset representatives of \(H \backslash G\).
  One readily verifies that \(\overline{\psi}\) does not depend on the specific choice of representatives.
\end{proof}

Next, we introduce the maps that we will use to construct the isometry between the Blanchfield form of the cover and the metabelian Blanchfield form.
Fix a \((R,\Z[H])\)-bimodule \(M\) and consider the following maps of \((R,\Z[H])\)-bimodules:
  \begin{align}
    \label{eq:structure-maps-ind-coind}
    i^{G}_{H}(M) &\colon M \to \op{Ind}^{G}_{H}(M), \quad M \owns m \mapsto m \otimes 1_{G} \in \op{Ind}^{G}_{H}(M), \\
    p^{G}_{H}(M) &\colon \op{Coind}^{G}_{H}(M) \to M, \quad f \in \op{Coind}^{G}_{H}(M) \mapsto f(1). \nonumber
  \end{align}
Using the decomposition~\eqref{eq:decomposition-ZG-as-ZH-mod}, we can see that \(M\) embedds in \(\op{Ind}^{G}_{H}(M)\) as the \((R,\Z[H])\)-subbimodule \(M \otimes 1_{G}\) and \(i^{G}_{H}(M)\) denotes the inclusion.
    Here, we used the same notation for the subbimodule \(M \otimes 1_{G}\) as in Example~\ref{ex:induced-representation}.
    Dually, if we apply the functor \(\hom_{\text{right}-\Z[H]}((-),M)\) to the inclusion \(\Z[H] \subset \Z[G]\) from Example~\ref{ex:induced-representation}, we obtain the map \(p^{G}_{H}(M)\).

The next lemma uses $\overline{\phi}$ and the maps from~\eqref{eq:structure-maps-ind-coind} to make more explicit the isomorphism from Proposition~\ref{prop:HeraldKirkLivingston}. 
\begin{lemma}[Shapiro's Lemma]\label{lemma:shapiros-lemma}
Suppose that \(X\) is a path-connected CW-complex and that \(q \colon Y \to X\) is a finite-sheeted covering with \(Y\) also path-connected.
For any \((R,\Z[\pi_{1}(Y)])\)-bimodule~\(M\), the following maps are isomorphisms of left $R$-modules:
\[H_{\ast}(Y;M) \xrightarrow{(q,i^{G}_{H}(M))} H_{\ast}(X;\op{Ind}^{G}_{H}(M)), \quad H^{\ast}(X;\op{Ind}^{G}_{H}(M)) \xrightarrow{(q,p^{G}_{H}(M) \circ \overline{\psi})} H^{\ast}(Y;M).\]
Here,  the maps \(i^{G}_{H}(M)\) and \(p^{G}_{H}(M)\) were defined in~\eqref{eq:structure-maps-ind-coind}, and the map \(\overline{\psi}\) was defined in the proof of Lemma~\ref{lemma:ind-coind-isom}.
\end{lemma}
\begin{proof}
First, observe that by~\cite[Proposition III.6.2]{brownCohomologyGroups1982} and~\cite[Exercise III.8.2]{brownCohomologyGroups1982}, the maps \((q,i^{G}_{H}(M))\) and \((q,p^{G}_{H}(M))\) induce isomorphisms of
  left \(R\)-modules, hence the first isomorphism follows.
  To obtain the second isomorphism, observe that the map \((q,p^{G}_{H}(M) \circ \overline{\psi})\) can be written as a composition
  \[H^{\ast}(X;\op{Ind}^{G}_{H}(M)) \xrightarrow{(\op{id}_{X},\overline{\psi})} H^{\ast}(X;\op{Coind}^{G}_{H}(M)) \xrightarrow{(q,p^{G}_{H}(M))} H^{\ast}(Y;M).\]
  The map \((\op{id},\overline{\psi})\) is an isomorphism by Lemma~\ref{lemma:ind-coind-isom}, hence the lemma follows.
\end{proof}

Before we end this section  we make one more observation concerning the interaction between induced representations and the Hermitian pairings defined in Example~\ref{ex:UsualCoeff}.
  This observation will be used in the proof of Lemma~\ref{prop:Shapiros-lemma-bl-forms}.

We recall the set-up: we fix a commutative domain \(R\) with involution \(r \mapsto r^{\#}\) and quotient field~$Q$ with the involution extended from \(R\).
    Recall, from Example~\ref{ex:UsualCoeff}, that for any \(d>0\) there  are pairings
    \[\langle -,- \rangle_{d} \colon (Q/R \otimes_{R} R^{d}) \times R^{d} \to Q/R, \quad \langle q \otimes v, w \rangle = v w^{\# T} \cdot q.\]
We additionally fix a group \(G\),  a normal subroup \(H \leq G\) of finite index \(n\) and coset representatives~\(1=g_{1},g_{2},\ldots,g_{n}\) of \(H \backslash G\).

\begin{lemma}\label{lemma:induced-hermitian-pairing}
Fix a unitary representation \(\beta \colon H \to GL_{d}(R)\).
    Under the identification $\op{Ind}^{G}_{H}(R^{d}_{\beta}) \cong R^{nd}_{\alpha}$ from Example~\ref{ex:induced-representation}, 
the pairing \(\langle -,- \rangle_{nd}\) $\colon (Q/R \otimes_R R^{nd}_\alpha) \times   R^{nd}_\alpha \to Q/R$ is \(G\)-equivariant and furthermore, for any \(x \in (Q/R \otimes_{R} R^{d}_{\beta})\) and \(y \in R^{d}_{\beta}\), we have
    \begin{equation}
      \label{eq:induced-hermitian-pairing-formula}
      \langle (q \otimes v) \otimes g_{i},  w \otimes g_{j} \rangle_{nd} =
      \begin{cases}
        \langle q \otimes v,w \rangle_{d}, & \text{if } i = j, \\
        0, & \text{otherwise}.
      \end{cases}
    \end{equation}
\end{lemma}
\begin{proof}
We first calculate the right hand side of~\eqref{eq:induced-hermitian-pairing-formula} and, more specifically, $ \langle q \otimes v,w \rangle_{d}$.
The unitarity of \(\beta\) guarantees that the pairing \(\langle -,- \rangle_{d}\) is \(H\)-equivariant: for any \(h \in H\), any \(q \in Q/R\), and any \(v,w \in R^{d}_\beta\), we have
\begin{align}
\label{eq:PairingLemma}
  \langle (q \otimes v) \cdot h, w \cdot h \rangle_{d}   & =
    \langle (q \otimes v) \beta(h), w \beta(h) \rangle_{d} = \beta(h)(\beta(h)^\#)^T      \langle q \otimes v, w \rangle_{d} \\  \nonumber
    &= \langle q \otimes v, w \rangle_{d}=q(vw^\#)^T.
    \end{align}
 Now we will calculate the left hand side of~\eqref{eq:induced-hermitian-pairing-formula}.

Next, we assert that under the identification $\op{Ind}^{G}_{H}(R^{d}_{\beta}) \cong R^{nd}_{\alpha}$ from Example~\ref{ex:induced-representation} (specifically the direct sum decomposition from~\eqref{eq:DirectSumInduced}), we have
    \[
\langle (q \otimes v) \otimes g_i,w \otimes g_j \rangle_{nd}=
    q( v
     \otimes g_{i}) (w \otimes g_{j})^{\# T} =
      \begin{cases}
        q(vw^{\# T})& \text{if } i=j, \\
        0& \text{otherwise},
      \end{cases}
    \]
    for any \(v,w \in R^{d}\) and $q \in Q/R$.
    More precisely, this assertion follows from the fact that if \(e_{1},e_{2},\ldots,e_{d}\) denotes the standard basis vectors of \(R^{d}\),  then the standard basis of \(R^{nd}\) is given by
    \[e_{1} \otimes g_{1},e_{2} \otimes g_{1},\ldots, e_{d} \otimes g_{1}, e_{1} \otimes g_{2}, e_{2} \otimes g_{2}, \ldots, e_{d} \otimes g_{n}, \ldots, e_{1} \otimes g_{n}, e_{e} \otimes g_{n} \ldots, e_{d} \otimes g_{n}.\] 
The formula~\eqref{eq:induced-hermitian-pairing-formula} then follows from the assertion and~\eqref{eq:PairingLemma}.
We leave  it to the reader to verify that
 \(\langle -,- \rangle_{nd}\) is \(G\)-equivariant.
\end{proof}


\subsection{Proof of Corollary~\ref{cor:BlanchfieldDownstairs}. }
\label{sub:BlanchfieldAppendix}

We are now ready to prove the main statement of this appendix which will in particular imply Corollary~\ref{cor:BlanchfieldDownstairs} by taking $M=M_K$.

\begin{proposition}[Shapiro's Lemma for Blanchfield forms]
\label{prop:Shapiros-lemma-bl-forms}
Let \(M\) be a closed, oriented, connected \(3\)-manifold and let \(q \colon M_{n} \to M\) be an \(n\)-sheeted covering, where \(M_{n}\) is connected.
  Suppose that \(\beta \colon \pi_{1}(M_{n}) \to GL_{d}(\LF)\) is an acyclic representation and denote \(\alpha = \op{Ind}^{\pi_{1}(M)}_{\pi_{1}(M_n)}(\beta) \colon \pi_{1}(M) \to GL_{nd}(\LF)\), see Example~\ref{ex:induced-representation}.
  The isomorphism \[H_{1}(M_{n};\LF^{d}_{\beta}) \cong H_{1}(M;\op{Ind}^{\pi_{1}(X)}_{\pi_{1}(Y)}(\LF^{d}_{\beta})) = H_{1}(M;\LF^{dn}_{\alpha})\] coming from Shapiro's Lemma~\ref{lemma:shapiros-lemma} yields an isometry of Blanchfield forms
  \[\Bl_{\beta}(M_{n}) \cong \Bl_{\alpha}(M).\]
\end{proposition}
\begin{proof}

Throughout this proof,  we set \(G:=\pi_{1}(M)\) and \(H:= \pi_{1}(M_{n})\).
Writing $\Bl_\beta^\bullet(M_n)$ and~$\Bl_\alpha^\bullet(M)$ for the adjoints of the Blanchfield pairings in the statement of the proposition, we must  show that the following  diagram  commutes: 
  \begin{equation}
    \label{eq:Blanchfield-Shapiro-lemma-comm-diag}
    \begin{tikzcd}
      H_1(M_{n};\F[t^{\pm 1}]^{d}_{\beta}) \ar[r, "{(q,i^G_{H})}"] \ar[d, "{\Bl_\beta^\bullet(M_n)}"] & H_1(M;\F[t^{\pm 1}]^{dn}_{\alpha}) \ar[d, "{\Bl_\alpha^\bullet(M)}"] \\
      \makeithash{\Hom_{\F[t^{\pm 1}]}(H_1(M_{n};\F[t^{\pm 1}]^{d}_{\beta}),\F(t)/\F[t^{\pm 1}])} & \makeithash{\Hom_{\F[t^{\pm 1}]}(H_1(M;\F[t^{\pm 1}]^{dn}_{\alpha}),\F(t)/\F[t^{\pm 1}])}. \ar[l, "{(q,i^{G}_{H})^{\bullet}}"]
    \end{tikzcd}
  \end{equation}
 In diagram~\eqref{eq:Blanchfield-Shapiro-lemma-comm-diag}, the upper horizontal isomorphism comes from Shapiro's Lemma~\ref{lemma:shapiros-lemma} and the bottom horizontal isomorphism is defined as
    \begin{equation}
      \label{eq:dual-Shapiro-map}
      (q,i^G_{H}(\beta))^{\bullet} := \Hom_{\LF}((q,i^{G}_{H}(\beta))_{\ast},\F(t) / \LF).
  \end{equation}
    In order to prove the commutativity of the above square, we rewrite it in a different form by expanding \(\Bl_\beta^\bullet(M_n)\) and \(\Bl_\alpha^\bullet(M)\), see diagram~\eqref{eq:commutative-diagram-blanchfield-downstairs}.
    In this diagram, the first three horizontal maps are the isomorphisms coming from Shapiro's Lemma~\ref{lemma:shapiros-lemma} and in the bottom horizontal line,
    we used the notation~\eqref{eq:dual-Shapiro-map}.
    Furthermore, in diagram~\eqref{eq:commutative-diagram-blanchfield-downstairs}, we identified \(\op{Ind}^{G}_{H}(\LF^{d}_{\beta}) = \LF^{nd}_{\alpha}\), as in Example~\ref{ex:induced-representation}:

\begin{equation}
\label{eq:commutative-diagram-blanchfield-downstairs}
    \begin{tikzcd}
      H_{1}(M_{n};\LF^{d}_{\beta}) \ar[r, "{(q,i^{G}_{H})}"] \ar[d,"PD"] & H_1(M;\LF^{nd}_{\alpha}) \ar[d, "PD"] \\
      H^2(M_{n};\LF^d_{\beta}) \ar[d, "BS^{-1}"] & H^{2}(M;\LF^{nd}_{\alpha}) \ar[l, "{(q,p^{G}_{H} \circ \overline{\phi})}" above] \ar[d, "BS^{-1}"] \\
      H^{1}(M_{n};(\F(t)/\LF)^{d}_{\beta}) \ar[d, "\ev"] & H^{1}(M;(\F(t) /\LF)^{nd}_{\alpha}) \ar[l, "{(q,p^{G}_{H} \circ \overline{\phi})}" above] \ar[d, "\ev"] \\
      \makeithash{\Hom_{\LF}(H_{1}(M_{n},\LF^d_{\beta}),\F(t)/\LF)} &       \makeithash{\Hom_{\LF}(H_{1}(M;\LF^{nd}_{\alpha}), \F(t)/\LF)}. \ar[l, "{(q,i^{G}_{H})^{\bullet}}"]
    \end{tikzcd}
    \end{equation}

    
Since each of the horizontal maps in the diagram are isomorphisms, we deduce that the conclusion of the proposition is equivalent to the commutativity of the diagram in~\eqref{eq:commutative-diagram-blanchfield-downstairs}.
We will now prove that each of the squares commutes, starting with the middle square before moving on to the upper square and concluding with the bottom square.

\begin{itemize}
\item  The commutativity of the middle square is a direct consequence of the naturality of the Bockstein map from Lemma~\ref{lemma:Bockstein-naturality}.
\item We prove the commutativity of the top square.
First we observe that this is equivalent to proving that the diagram in~\eqref{eq:smaller-commutative-diagram-blanchfield-downstairs} commutes, where \([M] \in H_3(M)\) and \([M_{n}] \in H_3(M_{n})\) respectively denote the fundamental classes of \(M\) and \(M_{n}\) and  \(\overline{\psi}\) was defined in the proof Lemma~\ref{lemma:ind-coind-isom}.

\begin{equation}
\label{eq:smaller-commutative-diagram-blanchfield-downstairs}
    \begin{tikzcd}
      H_{1}(M_{n};\LF^{d}_{\beta}) \ar[r, "{(q,i^{G}_{H}(\beta))}"] & H_{1}(M;\op{Ind}^G_H(\LF^{d}_{\beta})) \\
                                                & H^2(M;\op{Ind}^G_H(\LF^{d}_{\beta})) \ar[u, "{(-) \cap [M]}" right] \\
      H^{2}(M_{n};\LF^{d}_{\beta}) \ar[uu, "{(-) \cap [M_{n}] }" left] & H^{2}(M;\op{Coind}^{G}_{H}(\LF^{d}_{\beta})) \ar[l, "{(q,p^{G}_{H})^{\ast}}"] \ar[u, "{(\op{id}_{M},\overline{\psi})}" right] \\
    \end{tikzcd}
\end{equation}

The commutativity of the diagram in~\eqref{eq:smaller-commutative-diagram-blanchfield-downstairs} will be established by a strenuous but direct computation.
Fix \(\eta \in \Hom_{\text{right-}\Z[G]}(\makeithash{C_{2}(\widetilde{M})},\op{Coind}^{G}_{H}(\F[t^{\pm1}]^{d}_{\beta}))\) representing a cohomology class in \(H^{2}(M;\op{Coind}^{G}_{H}(\LF^{d}_{\beta}))\), i.e. an element  in the bottom right entry of the diagram.
We will first compute the image of \(\eta\) under the left-up-right composition and then by the up-up composition.
\begin{itemize}
\item We calculate the image of \(\eta\) under the left-up-right composition in~\eqref{eq:smaller-commutative-diagram-blanchfield-downstairs}.
In other words, we must calculate $(q,i_H^G)\left( \left( (q,p^{G}_{H})(\eta) \right) \cap \xi_{M_{n}} \right)$ where $\xi_{M_n} \in C_3(M_n)$ is a chain representing the fundamental class of $M_n$ that we will now describe.

Indeed, we wish to pick this $\xi_{M_n}$ in terms of a chain representing the fundamental class $[M]$ of $M$.
To achieve this,  we let \(\widetilde{M}\) be the universal covering of (both) \(M\) and \(M_{n}\) and let \(1=g_{1},g_{2},\ldots,g_{n}\) be coset representatives of \(H \backslash G\); observe that \(H\) is a normal subgroup of
\(G\), hence \(g_{1},g_{2},\ldots,g_{n}\) are also coset representatives for \(G / H\).
This way, if we choose a chain \(\xi_{M} \in C_{3}(\widetilde{M})\) representing the fundamental class in \(C_{3}(M) = \Z \otimes_{\Z[G]}C_{3}(\widetilde{M})\), the chain \(\xi_{M_{n}} = \sum_{i=1}^{n} g_{i} \cdot \xi_{M}\) represents the fundamental class of \(M_{n}\) in \(C_{3}(M_{n}) = \Z \otimes_{\Z[H]} C_{3}(\widetilde{M})\). 
Using this choice of $\xi_{M_n}$, we obtain that 
$$ \left( (q,p^{G}_{H})(\eta) \right) \cap \xi_{M_{n}} = \sum_{i=1}^{n} ((q,p^{G}_{H})(\eta)) \cap (g_{i} \cdot \xi_{M}).$$
For the next step in this calculation, recall that for any \((R,\Z[G])\)-bimodule \(A\), the cap product gives a chain map
\begin{align*}
  (-) \cap (-) \colon C^{i}(\widetilde{M};A) \times (\Z \otimes_{\Z[G]} C_{j}(\widetilde{M})) &\to C_{i-j}(\widetilde{M};A), \\
  f \cap (1 \otimes \sigma) &\mapsto f(\intprod \sigma) \otimes (\intprodr \sigma),
\end{align*}
where, for \(\sigma \in C_{j}(\widetilde{M})\), \(\intprod \sigma\) and \(\intprodr \sigma\) denote the respective front face and back face of \(\sigma\), respectively.
For more details concerning the twisted cap product,  refer to~\cite[Section~2.3]{ConwayNagelFibered}.
As a consequence,  we obtain
\begin{align*}  \left( (q,p^{G}_{H})(\eta) \right) \cap \xi_{M_{n}}
 =     \sum_{i=1}^{n} ((q,p^{G}_{H})(\eta)(\intprod (g_{i} \cdot \xi_{M}))) \otimes \intprodr (g_{i} \cdot \xi_{M}).
\end{align*}
But now,  by definition of the map \(p^{G}_{H}\) (see~\eqref{eq:structure-maps-ind-coind}),  we have
$(q,p^{G}_{H})(\eta)(\sigma) = \eta(\sigma)(1)$
for any \(\sigma \in C_{2}(\widetilde{M})\).
Using this, we obtain
\begin{align*}
\left( (q,p^{G}_{H})(\eta) \right) \cap \xi_{M_{n}}
  &= 
    \sum_{i=1}^{n} ((q,p^{G}_{H})(\eta)(\intprod (g_{i} \cdot \xi_{M}))) \otimes \intprodr (g_{i} \cdot \xi_{M}) \\
  &=
    \sum_{i=1}^{n} \eta(g_{i} (\intprod \xi_{M}))(1) \otimes g_{i} \cdot (\intprodr \xi_{M}) \\
  &=
    \sum_{i=1}^{n} (\eta \cdot g_{i}^{-1})(\intprod \xi_{M})(1) \otimes (g_{i} \cdot (\intprodr \xi_{M})) \\
  &=
    \sum_{i=1}^{n} \eta(\intprod \xi_{M})(g_{i}^{-1}) \otimes (g_{i} \cdot (\intprodr \xi_{M})),
\end{align*}
where the second equality follows from our discussion of the map $(q,p_H^G)$,  the third follows from the fact that \(\eta\) is right-\(\Z[G]\)-linear and the last follows from the definition of the right action of \(\Z[G]\) on \(\op{Coind}^{G}_{H}(\LF^{d}_{\beta})\), see~\eqref{eq:action-on-coind}.
    
Finally, using the definition of $(q,i_H^G)$, we deduce that the image of \(\eta\) under the left-up-right composition is given by
\begin{equation}\label{eq:left-up-right-composition}
(q,i^{G}_{H})(\left( (q,p^{G}_{H})^{\ast}(\eta) \right) \cap \xi_{M_{n}}) = \sum_{i=1}^{n} (\eta(\intprod \xi_{M})(g_{i}^{-1}) \otimes 1_{G}) \otimes (g_{i} \cdot (\intprodr \xi_{M})).
\end{equation}
    
\item We show that the image of \(\eta\) under the up-up composition in~\eqref{eq:smaller-commutative-diagram-blanchfield-downstairs}, namely $(\id,\overline{\psi})(\eta) \cap \xi_{M}$,  agrees with~\eqref{eq:left-up-right-composition}.
  By definition of $\overline{\psi}$ (recall~\eqref{eq:inv-nat-isomorphism-ind-coind}), for any chain~\(\sigma \in C_{2}(\widetilde{M})\), we have
  \[(\id,\overline{\psi})(\eta)(\sigma) = \sum_{i=1}^{n} \eta(\sigma)(g_{i}^{-1}) \otimes g_{i}.\]
  Consequently,  using the definition of the cap product,  we deduce that the image of \(\eta\) under the up-up composition is given by
  \begin{align}
    \label{eq:up-up-composition}
(\id,\overline{\psi})(\eta) \cap \xi_{M}
    &= \sum_{i=1}^{n} (\eta(\intprod \xi_{M})(g^{-1}_{i}) \otimes g_{i}) \otimes (\intprodr \xi_{M})  \nonumber \\
    &= \sum_{i=1}^{n} (\eta(\intprod \xi_{M})(g_{i}^{-1}) \otimes 1_{G}) \otimes  g_{i} \cdot (\intprodr \xi_{M}),
  \end{align}
  where in the last equality, we used the fact that the outer tensor product is taken over the ring \(\Z[G]\).
\end{itemize}
Comparing~\eqref{eq:left-up-right-composition} with~\eqref{eq:up-up-composition} shows that the diagram in~\eqref{eq:smaller-commutative-diagram-blanchfield-downstairs} is commutative, hence the second step of the proof is finished.
 \item We prove the commutativity of the bottom square in~\eqref{eq:commutative-diagram-blanchfield-downstairs}.
 In fact,  by Lemma~\ref{lemma:eval-map-naturality} (which  concerned the naturality of evaluation maps), it suffices to prove that
\begin{equation}
  \label{eq:WantForBottomSquare}
\langle x, (p^{G}_{H} \circ \overline{\phi})(y) \rangle_{d} = \langle i^{G}_{H}(x), y  \rangle_{nd}
\end{equation}
for every $x \in (\F(t) / \LF)^{d}_{\beta}$ and every $y \in \LF^{nd}_{\alpha}$,
  where the pairings $\langle -,- \rangle_{d}$ and $\langle -,- \rangle_{nd}$ were considered in Example~\ref{lemma:induced-hermitian-pairing}.
  Our plan is to expand the expression $ \langle x, (p^{G}_{H} \circ \overline{\phi})(y) \rangle_{d}$.

As a first step, we get a better grasp on the map $(p^{G}_{H} \circ \overline{\phi})$.
For that, recall that since~\(1=g_{1}, g_{2}, \ldots , g_{n}\) are coset representative of \(H \backslash G\),  we can write as in Example~\ref{ex:induced-representation}
\[\op{Ind}^{G}_{H}(\LF^{d}_{\beta}) = \bigoplus_{i=1}^{n} \LF^{d}_{\beta} \otimes g_{i} = \LF^{nd}_{\alpha}.\]
Observe that equations~\eqref{eq:nat-isomorphism-ind-coind} and~\eqref{eq:structure-maps-ind-coind}, imply that the map \(p^{G}_{H} \circ \overline{\phi} \colon \LF^{nd}_{\alpha} \to \LF^{d}_{\beta}\) is given by the formula
  \[p^{G}_{H}(v \otimes g_{i}) =
    \begin{cases}
      v& \text{if } i=1, \\
      0& \text{otherwise},
    \end{cases}
  \]
  where \(v \in \LF^{d}\) is arbitrary.

Using the claim, we take our first step in expanding $\langle x, (p^{G}_{H} \circ \overline{\phi})(y) \rangle_{d}$.
Indeed, observe that for any \(x \in (\F(t) / \LF)^{d}_{\beta} \) and any \(y = \sum_{i=1}^{n} y_{i} \otimes g_{i} \in \LF^{nd}_{\alpha}\), the above formula implies that
\begin{equation}
\label{eq:IntermediateBracket}
\langle x, (p^{G}_{H} \circ \overline{\phi})(y) \rangle_{d} = \langle x, y_{1} \rangle_{d}. 
\end{equation}
We can expand \(\langle i^{G}_{H}(x),y \rangle_{nd}\) using Lemma~\ref{lemma:induced-hermitian-pairing}:
  \[\langle i^{G}_{H}(x),y \rangle_{nd} = \langle x \otimes g_{1}, y \rangle_{nd} = \sum_{i=1}^{n} \langle x \otimes g_{1}, y_{i} \otimes g_{i}\rangle_{nd} = \langle x \otimes g_{1}, y_{1} \otimes g_{1} \rangle_{nd} = \langle x,y_{1} \rangle_{d}.\]
  Consequently, the formula~\eqref{eq:WantForBottomSquare} is satisfied.

This was the required equality and as we mentioned above,  we can now apply Lemma~\ref{lemma:eval-map-naturality} (with $\beta=\id$,  $\alpha=p_H^G \circ \overline{\phi}$ and $\alpha'=i^G_H$)
to conclude that the bottom square in the diagram in~\eqref{eq:commutative-diagram-blanchfield-downstairs} is commutative.
 \end{itemize}
 We have therefore proved that the three squares in~\eqref{eq:commutative-diagram-blanchfield-downstairs} are commutative and as we explained at the beginning of the proof, this implies that $\Bl_{\beta}(M_{n}) \cong \Bl_{\alpha}(M)$, thus establishing the proposition.
\end{proof}

\bibliographystyle{plain}
\def\MR#1{}
\bibliography{research}
\end{document}